\documentclass[10pt]{article}
\usepackage[a4paper, top=3cm, bottom=4cm, width=16cm]{geometry}
\usepackage[british]{babel}
\usepackage[T1]{fontenc}
\usepackage[utf8]{inputenc}

\usepackage{tikz}
\usetikzlibrary{arrows, shapes.geometric}
\usetikzlibrary{decorations.pathreplacing}
\tikzstyle{arrow} = [ultrathick,>=stealth]
\tikzstyle{block} = [rectangle, minimum width=3cm, minimum height=1cm, align=flush center, draw=black, thick]

\usepackage{amssymb,amsmath,amsthm,multicol,tabularx,makeidx,bbm,mathrsfs,stmaryrd}
\usepackage{fancyhdr,graphicx,color,listings}
\usepackage[outdir=./]{epstopdf}

\usepackage[pdftex,bookmarks,colorlinks,breaklinks]{hyperref}  % PDF hyperlinks, with coloured links
\usepackage[bookmarks,colorlinks,breaklinks]{hyperref}
\definecolor{dullmagenta}{rgb}{0.4,0,0.4}   % #660066
\definecolor{darkblue}{rgb}{0,0,0.4}
\hypersetup{linkcolor=red,citecolor=blue,filecolor=dullmagenta,urlcolor=darkblue} % coloured links

\usepackage{enumitem}
\usepackage{nccmath}

\newcommand{\assign}{:=}
\newcommand{\comma}{{,}}
\newcommand{\mathd}{\mathrm{d}}
\newcommand{\nobracket}{}

\newcommand{\tmaffiliation}[1]{\\ #1}
\newcommand{\tmcolor}[2]{{\color{#1}{#2}}}
\newcommand{\tmem}[1]{{\em #1\/}}
\newcommand{\tmemail}[1]{\\ \textit{Email:} \texttt{#1}}
\newcommand{\tmop}[1]{\ensuremath{\operatorname{#1}}}
\newcommand{\tmscript}[1]{\text{\scriptsize{$#1$}}}
\newcommand{\tmstrong}[1]{\textbf{#1}}
\newcommand{\tmtextbf}[1]{{\bfseries{#1}}}
\newcommand{\tmtextit}[1]{{\itshape{#1}}}
\newtheorem{theorem}{Theorem}%[section]
\newtheorem{conjecture}[theorem]{Conjecture}
\newtheorem{corollary}[theorem]{Corollary}
\newtheorem{definition}[theorem]{Definition}
\newtheorem{lemma}[theorem]{Lemma}
\newtheorem{proposition}[theorem]{Proposition}
\newtheorem{remark}[theorem]{Remark}
\begin{document}

\title{
  Prevalence of $\rho$-irregularity\\
  and related properties
}

\author{
  L.~Galeati \& M.~Gubinelli
  \tmaffiliation{Institute of Applied Mathematics \&\\
  Hausdorff Center for Mathematics\\
  University of Bonn\\
  Germany}
  \tmemail{\{lucio.galeati,gubinelli\}@iam.uni-bonn.de}
}

\maketitle

\begin{abstract}
  We show that generic H{\"o}lder continuous functions are $\rho$-irregular.
  The property of $\rho$-irregularity has been first introduced by Catellier
  and Gubinelli \ (Stoc.~Proc.~Appl.~126, 2016) and plays a key role in the
  study of well-posedness for some classes of perturbed ODEs and PDEs.
  Genericity here is understood in the sense of \tmtextit{prevalence}. As a
  consequence we obtain several results on regularisation by noise \ ``without
  probability'', i.e. without committing to specific assumptions on the
  statistical properties of the perturbations. We also establish useful
  criteria for stochastic processes to be $\rho$-irregular and study in detail
  the geometric and analytic properties of $\rho$-irregular functions.
\end{abstract}

{\tableofcontents}

\section{Introduction}

In recent years there has been a lot of focus in understanding the role of
noise in improving well-posedness of ordinary and partial differential
equations (ODE/PDEs), or so called \tmtextit{regularisation by noise}
phenomena, see~{\cite{flandoli}} for a review. From the modelling point of
view, the presence of external perturbations to otherwise autonomous
evolutions is a very natural assumption. For the sake of abstraction from any
specific origin of such perturbations, one is usually led to consider it
\tmtextit{random}, i.e. model the noise as a stochastic process. This common
approach introduces into the picture new considerations, some of which do to
quite fit the original deterministic formulation:
\begin{enumerate}
  \item It is not immediate to identify which types of noise are justified in
  any specific problem.
  
  \item Measurability (or adaptedness) of solutions in the sense of stochastic
  processes (i.e. seen as random variables on a filtered probability space)
  must be required; probabilistic notions of uniqueness (like path-wise
  uniqueness) are obtained, which are not easy to compare to their
  deterministic counterparts.
\end{enumerate}
Regarding~$1$., it is natural to require the noise to have statistical
properties like Gaussianity or self-similarity, as they are associated to
features of universality; as the perturbed problem has dynamical nature, it is
also reasonable to impose Markovianity. These considerations make Brownian
motion a natural candidate and a large set of theoretical tools is available
to analyse the effect of Brownian perturbations to deterministic evolutions.
For these reasons, this topic has a long and extensive history, see
e.g.~{\cite{zvonkin}}, {\cite{veretennikov}}, {\cite{krylov}}, {\cite{FGP}},
{\cite{dapratoflandoli}}, {\cite{DFRV}}, {\cite{BFGM}} and the references
therein. More recently, other classes of random perturbations partially
satisfying the above requirements have been analysed, like fractional Brownian
motion, $\alpha$-stable processes and more exotic variants, see
e.g.~{\cite{nualartouknine}}, {\cite{le}}, {\cite{priola2012}},
{\cite{priola2020}}, {\cite{banosproske}}.

In order to avoid the spurious technical problems of Point~$2$. one is led to
modify the probabilistic setting. For example, in the Brownian setting, Davie
and Flandoli introduced the stronger concept of \tmtextit{path-by-path
uniqueness}, see~{\cite{davie}},~{\cite{flandoli}},~{\cite{flandoli2011}}.
Shaposhnikov obtained further results in this direction,
see~{\cite{shaposhnikov1}}, {\cite{shaposhnikov2}} and the
recent~{\cite{shaposhnikov3}} with Wresch.

\

Here instead we follow the approach we initiated
in~{\cite{galeatigubinelli_ode}} and consider the regularisation by noise
problem from the point of view of \tmtextit{generic} perturbations, in
particular without reference to any (specific) probabilistic setting.

We will say that a property holds for a generic path, or equivalently for
\tmtextit{almost every} path, if it holds for the elements of a
\tmtextit{prevalent} set. The notion of prevalence was first introduced by
Christensen in~{\cite{christensen}} in the context of abelian Polish groups
and later rediscovered independently by Hunt, Sauer and Yorke
in~{\cite{huntsauer}} for complete metric vector spaces. It allows to define
consistently ``full Lebesgue measure sets'' in infinite dimensional spaces,
even if they don't admit any $\sigma$-additive, translation invariant measure.

Prevalence has been used in different contexts in order to study the
properties of generic functions of suitable regularity. For instance, it was
proved in~{\cite{hunt}} that almost every continuous function is nowhere
differentiable, while in~{\cite{fraysse2010}},~{\cite{fraysse2006}}, the
multi-fractal nature of generic Sobolev functions was shown. See~{\cite{ott}}
for a review. Recently, prevalence has also attracted attention in the study
of dimension of graphs and images of almost every continuous function, see
among others~{\cite{fraserhyde}},~{\cite{bayart}}.

One of the advantages of prevalence with respect to other notions of
genericity is that it allows the use of probabilistic methods in the proofs.
However the statements are fully non-probabilistic and the kind of problems
one encounters in formulating prevalence results are quite distinct from those
of a purely probabilistic setting, much more explored in the literature.

\

Our strategy and results are based on key ideas introduced by Catellier and
one of the authors in~{\cite{catelliergubinelli}}, while studying the
regularisation properties of fractional Brownian motion. In the following, we
schematically recall some of the main results of that paper and how they fit
in our framework.

\

Consider the perturbed ODE in $\mathbb{R}^d$ in integral form given by
\begin{equation}
  \label{eq:pert-ode} x_t = x_0 + \int_0^t b (s, x_s) \mathd s + w_t
\end{equation}
where $x_t = x (t)$, $w \in C^0 (\mathbb{R}_+, \mathbb{R}^d)$ is our fixed
perturbation and $b : \mathbb{R}_+ \times \mathbb{R}^d \rightarrow
\mathbb{R}^d$ is a time-dependent vector field (assume it to be continuous for
the moment).

Due to the additive nature of the perturbation, there are no particular
regularity requirements for $w$; it is natural to ask for which classes of $b$
equation~\eqref{eq:pert-ode} is well-posed and if, for suitable $w$, one can
obtain well-posedness results in classes which are ill-posed when $w = 0$.

The work~{\cite{catelliergubinelli}} singles out two analytic concepts linked
to the well-posedness of ~\eqref{eq:pert-ode}:
\begin{enumerate}
  \item The \tmtextit{averaged field} $T^w b$, given by
  \[ T^w b (t, x) = \int_0^t b (s, x + w_s) \mathd s, \qquad x \in
     \mathbb{R}^d, t \geqslant 0. \]
  \item The notion of \tmtextit{$\rho$-irregularity} of a path $w$, defined
  precisely below.
\end{enumerate}
Indeed, in order to study~\eqref{eq:pert-ode}, i.e. to give it meaning and
establish existence and uniqueness of solutions, together with a Lipschitz
flow, it is enough to have good regularity estimates for the averaged function
$T^w b$, even when the drift $b$ is merely a distribution. Schematically:

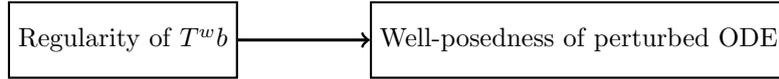
\begin{center}
\begin{tikzpicture}
\node (reg) [block] {Regularity of $T^w b$};
\node (ode) [block, right of=reg, xshift=7cm] {Well-posedness of perturbed ODE};
\draw[->,line width=0.4mm] (reg) --node[anchor=south] {Theorem \ref{sec6 thm ode}} (ode);
\end{tikzpicture}
\end{center}

In order to analyse the regularising effect of $w$ on a collection of drifts
$b$, the \tmtextit{averaging operator} $T^w : b \mapsto T^w b$ is then
introduced; it is closely tied to the \tmtextit{occupation measure} $\mu^w$,
whose regularity is in turn quantified by the concept of $\rho$-irregularity:

\

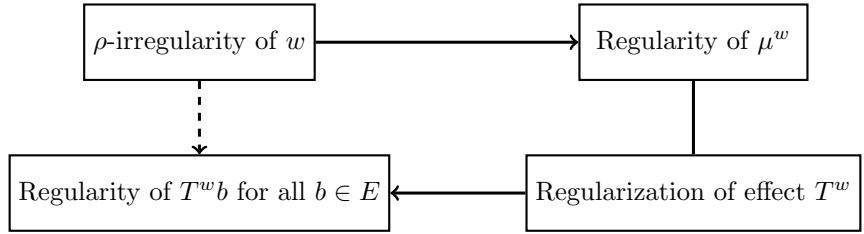
\begin{center}
\begin{tikzpicture}
\node (rho) [block] {$\rho$-irregularity of $w$};
\node (measure) [block, right of=rho, xshift=8cm] {Regularity of $\mu^w$};
\node (averaging1) [block, below of=measure, yshift=-2cm] {Regularization effect of $T^w$};
\node(averaging2) [block, left of=averaging1, xshift=-8cm] {Regularity of $T^w b$ for all $b \in E$};
\draw[->,line width=0.4mm] (rho) --node[anchor=south] {Equation \eqref{eq:link-occupatio-rho}} (measure);
\draw[->,line width=0.4mm] (measure) --node[anchor=west] {Lemma \ref{sec2.3 lemma averaging occupation}} (averaging1);
\draw[->,line width=0.4mm] (averaging1) --node[anchor=south] {Lemmas \ref{sec2.2 lemma regularity averaging 1}, \ref{sec2.2 lemma regularity averaging 2}} (averaging2);
\draw[->,line width=0.4mm] (rho) --node[anchor=east] {Overall implication } (averaging2);
\end{tikzpicture}
\end{center}

\

Here $E$ is a suitable Banach space, for instance
in~{\cite{catelliergubinelli}} Fourier--Lebesgue spaces are considered. The
precise meaning of the above diagram is explained in detail in
Sections~\ref{sec2.2} and~\ref{sec2.3}.

\

Even if $\rho$-irregularity has been first introduced in the study of
perturbed ODEs, it is a concept of independent interest and has been applied
to establish regularisation by noise for several modulated PDEs,
see~{\cite{choukgubinelli1}},~{\cite{choukgubinelli2}},~{\cite{choukgess}}.
More recently, a connection to the property of inviscid mixing for shear flows
has been found in~{\cite{galeatigubinelli_mixing}}.

\

In the companion paper~{\cite{galeatigubinelli_ode}}, we analyse in detail
the regularity of the averaged field $T^w b$ for generic $w \in C^{\delta}$
when $b$ is a \tmtextit{fixed} distribution. Here instead we focus on the
regularity of the operator $T^w$ for a generic $w \in C^{\delta}$, as
quantified by the concept of $\rho$-irregularity. This has the advantage of
showing that a generic additive perturbation $w$ regularises a huge class of
ODEs, namely the ones of the form~\eqref{eq:pert-ode} associated to $b \in E$.
On the other hand, it has the drawback that the space $E$ is strictly smaller
than the class of all possible drifts taken singularly. For instance
in~{\cite{galeatigubinelli_ode}} time~dependent distributions $b$, with only
$L^p$-regularity in time, can be considered, while it is not possible to
obtain regularity estimates for the operator $T^w$ in that class, as explained
in Section~3.1 therein.

\

In this paper, we expand the theory developed in~{\cite{catelliergubinelli}}
in several ways. For instance, we identify suitable locally-nondeterministic
conditions (LND) which ensure $\rho$-irregular trajectories for a large class
of stochastic Gaussian processes:

\

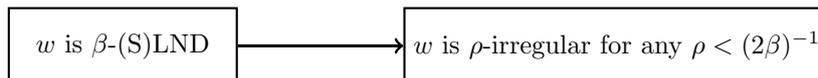
\begin{center}
\begin{tikzpicture}
\node (lnd) [block] {$w$ is $\beta$-(S)LND};
\node (rho) [block, right of=rho, xshift=7cm] {$w$ is $\rho$-irregular for any $\rho < (2 \beta)^{-1}$};
\draw[->,line width=0.4mm] (lnd) --node[anchor=south] {Theorem \ref{sec3 thm1}} (rho);
\end{tikzpicture}
\end{center}

The concept of \tmtextit{local-nondeterminism} is explained in detail in
Section~\ref{sec2.4}. The above probabilistic result then allows to prove a
much stronger (non probabilistic) prevalence statement, thanks to the fact
that the LND condition above is invariant under additive deterministic
perturbations.

\

Usually prevalence statements are established by means of \tmtextit{probes},
i.e. measures supported on finite dimensional subspaces. To the best of our
knowledge, the only other work adopting a approach similar to ours, i.e.
exploiting stochastic processes to establish prevalence results,
is~{\cite{bayart}}. See also~{\cite{peres}} and the references therein for the
study of properties of fractional Brownian motion with deterministic drift
(however without prevalence considerations).

\

We also study more in detail the geometric and analytic properties of
$\rho$-irregular functions, like the Fourier dimension of their image. We
establish that $\rho$-irregularity is indeed a concept of irregularity, as the
terminology suggests, in a precisely quantifiable way:

\

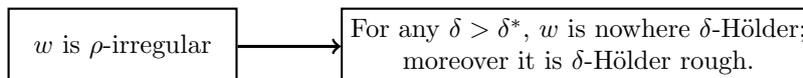
\begin{center}
\begin{tikzpicture}
\node (rho) [block] {$w$ is $\rho$-irregular};
\node (holder) [block, right of=rho, xshift=7cm] {For any $\delta > \delta^{\ast}$, $w$ is nowhere
    $\delta$-H\"{o}lder;\\
    moreover it is $\delta$-H\"{o}lder rough.};
\draw[->,line width=0.4mm] (rho) --node[anchor=south] {Theorem \ref{sec3 thm3}}(holder);
\end{tikzpicture}
\end{center}

\

Here $\delta^{\ast}$ is a critical scaling parameter which depends on $\rho$.
The precise statements of our results are given in Section~\ref{sec3}, but let
us give here informally some of the main highlights:
\begin{itemize}
  \item Almost every $w \in C^{\delta} ([0, T] ; \mathbb{R}^d)$ is
  $\rho$-irregular for any $\rho < (2 \delta)^{- 1}$. Moreover we introduce a
  stronger notion of irregularity, which allows us to prove that almost every
  $\varphi \in C^{\delta} ([0, T] ; \mathbb{R}^d)$ remains irregular under
  smooth additive perturbations. See Theorems \ref{sec3 thm2}-i) and \ref{sec3 thm5}.
  
  \item Almost every $w \in C^0 ([0, T] ; \mathbb{R}^d)$ is infinitely
  irregular, in the sense that it is $\rho$-irregular for any $\rho < \infty$.
  The associated averaging operator is infinitely regularising, in the sense
  that $T^w$ maps $\mathcal{S}' (\mathbb{R}^d)$ into $C^{\infty}
  (\mathbb{R}^d)$. See Theorems \ref{sec3 thm2}-ii).
  
  \item Almost every $w \in C^0 ([0, T] ; \mathbb{R}^d)$ has infinite
  regularisation effect on the ODE. Namely, it renders~\eqref{eq:pert-ode}
  well-posed and with a smooth flow for a large class of fields $b$, for
  instance for any $b \in H^{\alpha} (\mathbb{R}^d ; \mathbb{R}^d)$ for any
  $\alpha \in \mathbb{R}$. See Theorem \ref{sec6 thm ode}.
\end{itemize}
An overview of our results, combined with the path-wise theory developed
in~{\cite{catelliergubinelli}}~{\cite{choukgubinelli1}},
{\cite{choukgubinelli2}} discussed above, is given by the conceptual map of
Figure~\ref{fig1}.

\begin{figure}[h]\label{fig1}
\resizebox{.95\textwidth}{!}{  
      \begin{tikzpicture}
%
% row 1
%
\node (lnd) [block] {LND};
%
% row 2
\node (rho) [block, below of=lnd, yshift=-3cm] {$\rho$-irregularity of $w$};
\node (dummy) [left of=rho, xshift=-2cm, inner sep = 0pt] {};
\node (fourier) [block, above of=dummy, xshift=-4cm] {Fourier dimension\\ of $w([s,t])$};
\node (holder) [block, below of=dummy, xshift=-4cm] {$w$ is nowhere $\delta$-H\"{o}lder\\ for $\delta\geqslant \delta^\ast$};
\node (dummy2) [above of=fourier, xshift=-2cm,yshift=-1cm] {};
\node (dummy3) [below of=fourier, yshift=-10.3cm, xshift=-2cm] {};
\node (dummy1) [above of=dummy2, yshift=2.5cm] {};
%
%
%
% row 3
\node (occupation) [block, below of=rho, yshift=-1.5cm]  {Regularity of $\mu^w$}; 
% row 4
\node (averaging1) [block, below of=occupation, yshift=-1.5cm]  {Regularization effect of $T^w$};
%row 5
\node (averaging2) [block, below of=averaging1, yshift=-1.5cm]  {Regularity of $T^w b$\\ for all $b\in E$};
%
%row 6
\node (ode) [block, below of=averaging2, yshift=-1.5cm] {Well-posedness of\\ perturbed ODE
\\ $\dot{x}=b(x)+\dot{w}$
};
\node (transport) [block, left of=ode, xshift=-3.5cm]  {Well-posedness of\\
perturbed transport PDE
\\ $\partial_t u+ b\cdot\nabla u + \dot{w}\cdot\nabla u=0$
};
\node (modulated) [block, right of=ode, xshift=3.5cm]  {Well-posedness of\\
modulated PDE
\\ $\partial_t \varphi =A\varphi \dot{w} + \mathcal{N}(\varphi)$
};

 \draw[->,line width=0.4mm] (lnd) --node[anchor=west] {Theorem \ref{sec3 thm1}} (rho);
 \draw[->,line width=0.4mm] (rho) --node[anchor=west] {Equation \eqref{eq:link-occupatio-rho}} (occupation);
 \draw[->,line width=0.4mm] (occupation) --node[anchor=west] {Lemma \ref{sec2.3 lemma averaging occupation}} (averaging1);
 \draw[->,line width=0.4mm] (averaging1) --node[anchor=west] {Lemmas \ref{sec2.2 lemma regularity averaging 1}, \ref{sec2.2 lemma regularity averaging 2}} (averaging2);
 \draw[->,line width=0.4mm]    (averaging2) --node[anchor=west] {Theorem \ref{sec6 thm ode}} (ode); 
 \draw[->,line width=0.4mm]    (averaging2) -|node[anchor=south] {Theorem \ref{sec6 thm transport}} (transport);
 \draw[->,line width=0.4mm]    (rho) -|node[anchor=west,xshift=0.1cm, yshift=-3cm, align=center] {Theorems\\ \ref{sec6 thm nls}, \ref{sec6 thm kdv}} (modulated);
 \draw[-,line width=0.4mm]    (rho) --
 %node[anchor=east, xshift=-1.5cm] {Geometric properties}
  (dummy);
 \draw[->,line width=0.4mm]    (dummy) |-node[anchor=south, xshift=-0.5cm, yshift=0.1cm] {Theorem \ref{sec3 thm4}} (fourier);
 \draw[->,line width=0.4mm]    (dummy) |-node[anchor=north, xshift=-0.5cm, yshift=-0.1cm] {Theorem \ref{sec3 thm3}} (holder);
 \draw[decorate,decoration={brace,amplitude=15pt,mirror},line width=0.4mm]
(dummy2) -- (dummy3) node [black,midway,xshift=-1.5cm, align=center] 
{ {\footnotesize Analytical}\\ {\footnotesize part}};
 \draw[decorate,decoration={brace,amplitude=15pt,mirror},line width=0.4mm]
(dummy1) -- (dummy2) node [black,midway,xshift=-1.5cm,align=center] 
{ {\footnotesize Probabilistic}\\ {\footnotesize step}};
 
\end{tikzpicture}
}
\caption{Scheme of the results}
%\end{minipage}
\end{figure}
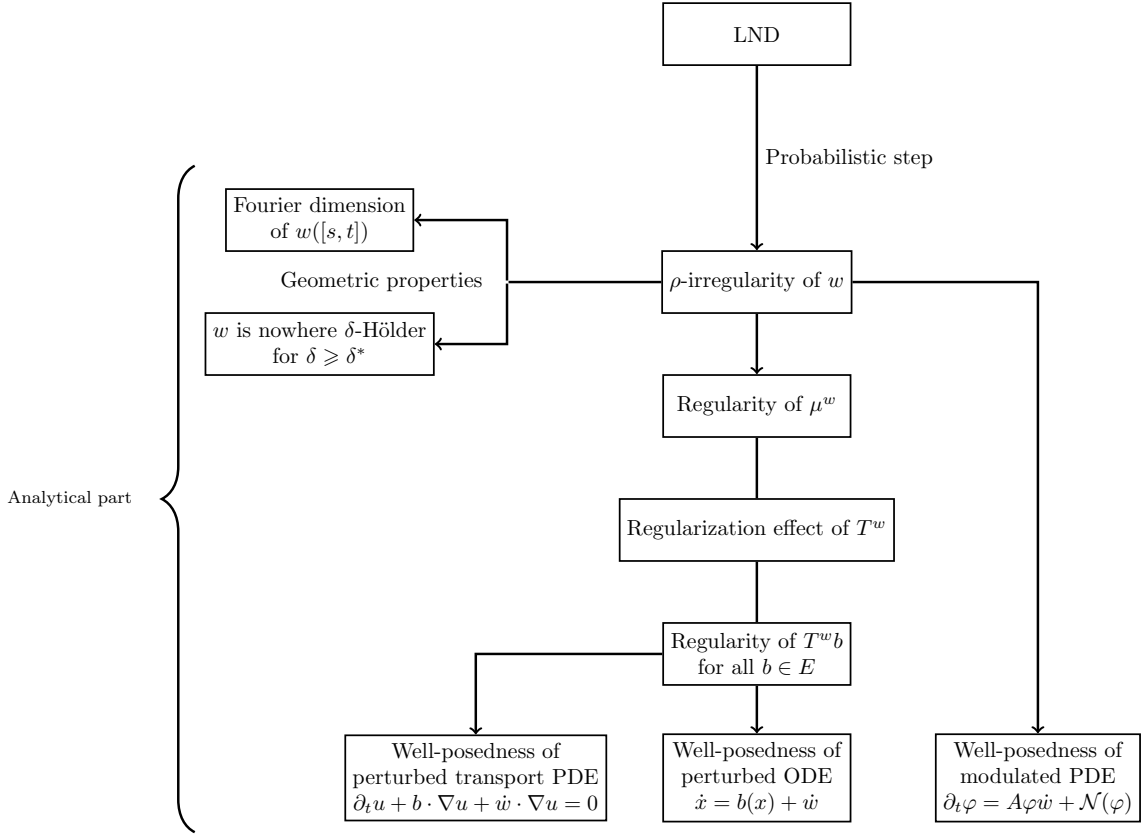

While finalising the first version of this paper, we become aware of the
preprint~{\cite{perkowski}} which indepentently develop some result similar to
ours. The authors also identify some LND conditions that ensure time-space
regularity of $\mu^w$ and apply them in combination with the theory
from~{\cite{catelliergubinelli}} to establish that suitable stochastic
processes are ``infinitely regularising'' the perturbed ODE. By introducing
genericity into the picture our work carries such probabilistic considerations
into a wider scope and quantifies the regularisation effect precisely wrt. the
H{\"o}lder regularity of the paths. We also answer some of the questions left
open in~{\cite{perkowski}}, for instance on the link between regularity of
$\mu^w$ and true roughness of $w$.

Let us finally mention the upcoming works~{\cite{tolomeo1,tolomeo2}}, further
analysing notions of irregularity of paths and giving sufficient conditions to
check them for stochastic processes; among other results, it seems the authors
are able to give a negative answer to Conjecture~\ref{conj:rho} from
Section~\ref{sec5.3}.

\

\tmtextbf{Structure of the paper.} Section~\ref{sec2} introduces all our main
actors, namely prevalence, $\rho$-irregularity, occupation measures and
averaging operators and describes the basic relations between them. Fractional
Brownian motion (fBm) enters into the picture as a transverse measure for
prevalence and local-nondeterminism (LND) as a key feature of fBm paths. The
reader acquainted with these concepts can skip this preliminary section. With
the above preparations, we can formulate in Section~\ref{sec3.1} the
statements of our main results. Section~\ref{sec3.2} shows how they can be
combined with already existing results to provide various ``noiseless
regularisation by noise'' results both for ODEs and PDEs. Sections~\ref{sec4}
and~\ref{sec5} constitute the main body of the paper and contain the proofs of
the statements from Section~\ref{sec3.1}. These sections split naturally into
a probabilistic part, in which stochastic criteria for establishing
$\rho$-irregularity are developed, and an analytic one, focused on geometric
and analytical properties of $\rho$-irregular paths. We choose to put in the
Appendix reminders of useful standard facts.

\

{\tmstrong{Notation.}} We will use the notation $a \lesssim b$ to mean that
there exists a positive constant $c$ such that $a \leqslant c \hspace{0.17em}
b$; we use the index $a \lesssim_x b$ to highlight the dependence $c = c (x)$.
$a \sim b$ if and only if $a \lesssim b$ and $b \lesssim a$, similarly for $a
\sim_x b$.

We will always work on a finite time interval $[0, T]$ unless stated
otherwise. Whenever useful we adopt the convention that $f_t$ stands for $f
(t)$ for a function $f$ indexed on $t \in [0, T]$, but depending on the
context we will use both notations; similarly for the increments of $f_{s, t}
= f_t - f_s$.

For $x \in \mathbb{R}^d$, $|x|$ denotes the Euclidean norm, $x \cdot y$ the
scalar product, $\langle x \rangle = (1 + |x|^2)^{1 / 2}$. For any $R > 0$,
$B_R$ stands for $B (0, R) = \{x \in \mathbb{R}^d : |x| \leqslant R\}$,
$\mathbb{S}^{d - 1} = \{x \in \mathbb{R}^d : |x| = 1\}$.

We denote by $\mathcal{S} (\mathbb{R}^d ; \mathbb{R}^m)$ and $\mathcal{S}'
(\mathbb{R}^d ; \mathbb{R}^m)$ respectively the spaces of vector-valued
Schwarz functions and tempered distributions on $\mathbb{R}^d$, similarly
$C^{\infty}_c (\mathbb{R}^d ; \mathbb{R}^m)$ is the set of vector-valued
smooth compactly supported functions; $\hat{f}$ stands for the Fourier
transform of $f \in \mathcal{S}' (\mathbb{R}^d)$. Standard Lebesgue spaces are
denoted by $L^p (\mathbb{R}^d ; \mathbb{R}^m)$. Whenever it doesn't create
confusion, we will just write $C^{\infty}_c$ and $L^p$ for short.

Given a Banach space $E$, $\alpha \in (0, 1)$, $C^{\alpha} ([0, T] ; E) =
C^{\alpha}_t E$ is the space of $E$-valued $\alpha$-H{\"o}lder continuous
functions, i.e. $f : [0, T] \rightarrow E$ such that
\[ \|f\|_{C^{\alpha} E} \assign \|f\|_{C^0 E} + \llbracket f
   \rrbracket_{C^{\alpha} E} = \sup_{t \in [0, T]} \|f_t \|_E +
   \sup_{\tmscript{\begin{array}{c}
     s \neq t \in [0, T]
   \end{array}}}  \frac{\|f_{s, t} \|_E}{|t - s|^{\alpha}} < \infty . \]
A similar definition holds for $\tmop{Lip} ([0, T] ; E) = \tmop{Lip}_t E$.
More generally, for a given modulus of continuity $\omega$ (possibly defined
only in a neighbourhood of $0$), we denote by $C^{\omega} ([0, T] ; E) =
C^{\omega} E$ the set of all $E$-valued continuous functions with modulus of
continuity $\omega$. $C^0_t E = C^0 ([0, T] ; E)$ is the space of $E$-valued
continuous functions, endowed with the supremum norm.

Whenever $E =\mathbb{R}^d$, we will refer to $w \in C^{\alpha}_t = C^{\alpha}
([0, T] ; \mathbb{R}^d)$ as a {\tmem{path}} and in this case we allow $\alpha
\in [0, \infty)$ with the convention that $w \in C_t^{\alpha}$ it is has
continuous derivatives up to order~$\lfloor \alpha \rfloor$ and $D^{\lfloor
\alpha \rfloor} \varphi$ is $\{\alpha\}$--H{\"o}lder continuous, where
$\lfloor \alpha \rfloor$ and $\{\alpha\}$ denote integer and fractional part.

Whenever a stochastic process $X = \{X_t \}_{t \geqslant 0}$ is considered,
even when it is not specified we imply the existence of an abstract underlying
filtered probability space $(\Omega, \mathcal{F}, \{\mathcal{F}_t \}_{t
\geqslant 0}, \mathbb{P})$ such that $\mathcal{F}$ and $\mathcal{F}_t$ satisfy
the usual assumptions and $X_t$ is adapted to $\mathcal{F}_t$. If
$\mathcal{F}_t$ is said to be the natural filtration generated by $X$, then it
is tacitly implied that it is actually its right continuous, normal
augmentation. We denote by $\mathbb{E}$ integration (equiv. expectation)
w.r.t. the probability $\mathbb{P}$ and by $\mathbb{E} \left[ \, \cdot \, |
\mathcal{G} \right]$ conditional expectation w.r.t. a $\sigma$-algebra
$\mathcal{G}$.

\

\tmtextbf{Acknowledgements.} We thank Nicolas Perkowski and Leonardo Tolomeo
for useful discussions, leading respectively to Lemma~\ref{lem:perkowski} and
Remark~\ref{rem:tolomeo}; we are also grateful to the anonymous referee for
their very carefully reading and numerous insights, which highly improved the
paper. The authors were supported by the Deutsche Forschungsgemeinschaft (DFG,
German Research Foundation) through the Hausdorff Center for Mathematics under
Germany's Excellence Strategy -- EXC-2047/1 -- 390685813 and through CRC 1060
- projekt number 211504053.

\section{Preliminaries}\label{sec2}

\subsection{Introduction to prevalence}\label{sec2.1}

We follow the exposition and the terminology given in~{\cite{huntsauer}}. The
reader might however keep in mind that in the following we will be only
interested in the case of a Banach space $E$.

\begin{definition}
  \label{definition prevalence}Let $E$ be a complete metric vector space. A
  Borel set $A \subset E$ is said to be \tmtextbf{shy} if there exists a
  measure $\mu$ such that:
  \begin{enumerate}[label=\roman*.]
    \item There exists a compact set $K \subset E$ such that $0 < \mu (K) <
    \infty$.
    \item For every $v \in E$, $\mu (v + A) = 0$.
  \end{enumerate}
  In this case, the measure $\mu$ is said to be transverse to $A$. More
  generally, a subset of $E$ is shy if it is contained in a shy Borel set. The
  complement of a shy set is called a prevalent set.
\end{definition}

Sometimes it is said more informally that the measure $\mu$ ``witnesses'' the
prevalence of $A^c$ in $E$. It follows immediately from part \tmtextit{i.} of
the definition that if needed one can assume $\mu$ to be a compactly supported
probability measure on $E$. If $E$ is separable, then any probability measure
on $E$ is tight and therefore \tmtextit{i.} is automatically satisfied.

\

The following properties hold for prevalence:
\begin{enumerate}
  \item[1.] If $E$ is finite dimensional, then a set $A$ is shy if and only if it
  has zero Lebesgue measure (cf. Fact~6 from~{\cite{huntsauer}}).
  
  \item[2.] If $A$ is shy, then so is $v + A$ for any $v \in E$ (Fact~1
  from~{\cite{huntsauer}}).
  
  \item[3.] Prevalent sets are dense (Fact~2' from~{\cite{huntsauer}}).
  
  \item[4.] If $\dim (E) = + \infty$, then compact subsets of $E$ are shy (Fact~8
  from~{\cite{huntsauer}}).
  
  \item[5.] Countable union of shy sets is shy; conversely, countable intersection
  of prevalent sets is prevalent (Fact~3'' from~{\cite{huntsauer}}).
\end{enumerate}
\begin{remark}
  \label{rem:prevalence-consistent}Although it is not entirely trivial, it
  follows from the definition that a Borel set $A$ cannot be at the same time
  both shy and prevalent. To see it, assume by contradiction this to be the
  case; then we can find two compactly supported probability measures $\mu,
  \nu$ such that $\mu (v + A) = 0$ and $\nu (v + A^c) = 0$ for all $v \in E$,
  where $A^c$ denotes the complement of $A$. It would then follow from Fact~3
  from~{\cite{huntsauer}} that $(\mu \ast \nu) (v + A) = 0 = (\nu \ast \mu) (v
  + A^c)$ for all $v \in E$, yielding to $(\mu \ast \nu) (A)$ having at the
  same time both values $0$ and $1$.
\end{remark}

From now, whenever we say that a statement holds for a.e. $v \in E$, we mean
that the set of elements of $E$ for which the statement holds is a prevalent
set; property~1. states that this is consistent with the finite dimensional
case.

In the context of a function space $E$, it is natural to consider as
probability measure the law induced by an $E$\mbox{-}valued stochastic
process. Namely, given stochastic process $W$ defined on a probability space
$(\Omega, \mathcal{F}, \mathbb{P})$ taking values in a separable Banach space
$E$, in order to show that a property $\mathcal{P}$ holds for a.e. $f \in E$,
it suffices to show that
\[ \mathbb{P} \left( \text{$f + W$ satisfies property $\mathcal{P}$} \right) =
   1 \qquad \forall \, f \in E. \]
Clearly, we are assuming that the set $A = \left\{ w \in E : \text{$w$
satisfies property $\mathcal{P}$} \right\}$ is Borel measurable; if $E$ is not
separable, then we need in addition to require that the law of $W$ is tight,
so that point~\tmtextit{i.} of Definition~\ref{definition prevalence} is
satisfied.

\

Let us point out that, as a consequence of Properties~4. and~5. above, the
set of all possible realisations of any probability measure $\mu$ on a
separable infinite dimensional Banach $E$ space (actually, any tight measure $\mu$ on an infinite dimensional Banach
space $E$) is contained in a shy set, in the following sense: we can always
find a Borel set $A \subset E$, which is a countable union of compacts, such
that $\mu (A) = 1$, yet $A$ is shy. If we are given another Borel set
$\tilde{A} \subset E$, in order to verify its prevalence we really need to
check the translation property that $\mu (\tilde{A} + f) = 1$ for any $f \in
E$; if we only know that $\mu (\tilde{A}) = 1$, then $\tilde{A}$ can be shy or
prevalent (or neither).

\ This highlights the difference between a statement of the form ``Property
$\mathcal{P}$ holds for a.e. $f$'' and for instance ``Property $\mathcal{P}$
holds for all Brownian trajectories'', where the last statement corresponds to
$\mu \left( \text{Property $\mathcal{P}$ holds} \right) = 1$, $\mu$ being the
Wiener measure on $C^0 ([0, 1])$. Indeed, the second statement doesn't provide
any information regarding whether the property might be prevalent or not;
intuitively, the elements satisfying a prevalence statement are ``many more''
than just the realisations of the Wiener measure.

\subsection{Basic properties of $\rho$\mbox{-}irregularity and related
concepts}\label{sec2.2}

The concept of $\rho$\mbox{-}irregularity was introduced
in~{\cite{catelliergubinelli}} as an analytic property of continuous functions
which allows to quantitatively measure both their irregularity and their
smoothing effect on perturbations of ODEs.

\begin{definition}\label{sec2.2 defn rho irregularity}
Let $\gamma\in [0,1]$, $\rho > 0$. A measurable
  path $w : [0, T] \rightarrow \mathbb{R}^d$ is $(\gamma,
  \rho)$\mbox{-}irregular{\tmstrong{}} if there exists a constant $C$ such
  that
  \begin{equation}
    \left| \int_s^t e^{i \xi \cdot w_r} \, \mathd r \right| \leqslant C | \xi
    |^{- \rho} | t - s |^{\gamma} \quad \forall \, \xi \in \mathbb{R}^d
    \setminus \{ 0 \}, \, s, t \in [0, T] . \label{sec2.2 defn rho-irr}
  \end{equation}
  We denote by $\| \Phi^w \|_{\mathcal{W}^{\gamma, \rho}}$ the optimal
  constant $C$; with the notation $\Phi^w_t (\xi) = \int_0^t e^{i \xi \cdot
  w_r} \mathd r$, it holds
  \[ \| \Phi^w \|_{\mathcal{W}^{\gamma, \rho}} \assign \sup_{\xi \in
     \mathbb{R}^d, s \neq t} \frac{| \Phi^w_{s, t} (\xi) | | \xi |^{\rho}}{| t
     - s |^{\gamma}} . \]
  We say that $w$ is $\rho$\mbox{-}irregular if there exists $\gamma > 1 / 2$
  such that $w$ is $(\gamma, \rho)$-irregular.
\end{definition}

We have the trivial bound $| \Phi^w_{s, t} (\xi) | \leqslant | t - s |$, so
that~\eqref{sec2.2 defn rho-irr} is always satisfies for $| \xi |$ small and
the relevant information in the above definition is given by the uniform bound
as $\xi \rightarrow \infty$. Therefore we can always without problems replace
$| \xi |$ with another function with the same asymptotic behaviour (in the
original definition from~{\cite{catelliergubinelli}}, $1 + | \xi |$ appeared,
but we prefer here the choice $| \xi |$ for its better scaling properties).

\

Let us collect now some elementary facts on $\rho$\mbox{-}irregular
functions.

\begin{lemma}
  \label{sec2.2 lemma properties rho irr}Let $w : [0, T] \rightarrow
  \mathbb{R}^d$ be a $(\gamma, \rho)$\mbox{-}irregular continuous path. Then
  the following hold:
  \begin{enumerate}[label=\roman*.]
    \item Symmetry invariance: $- w$ is $(\gamma, \rho)$\mbox{-}irregular with
    $\| \Phi^w \|_{\mathcal{W}^{\gamma, \rho}} = \| \Phi^{- w}
    \|_{\mathcal{W}^{\gamma, \rho}}$.    
    \item Translation invariance: for any $r \in [0, T]$, $w_{\cdot} - w_r$ is
    $(\gamma, \rho)$\mbox{-}irregular with $\| \Phi^w \|_{\mathcal{W}^{\gamma,
    \rho}} = \| \Phi^{w_{\cdot} - w_r} \|_{\mathcal{W}^{\gamma, \rho}}$. 
    \item Scaling invariance: for any $\lambda \in (0, 1)$, $w^{\lambda} (t)
    \assign \lambda^{- (1 - \gamma) / \rho} \, w (\lambda t)$ is $(\gamma,
    \rho)$\mbox{-}irregular.
    \item Rotation invariance: for any $O \in \tmop{SO} (d)$, $O \, w$ is
    $(\gamma, \rho)$-irregular with $\left\| \Phi^{O \, w}
    \right\|_{\mathcal{W}^{\gamma \comma \rho}} = \| \Phi^w
    \|_{\mathcal{W}^{\gamma \comma \rho}}$.
    \item More generally if $A \in \mathbb{R}^{d \times d}$ is invertible,
    then $A \, w$ is $(\gamma, \rho)$-irregular with
    \[ \left\| \Phi^{A \, w} \right\|_{\mathcal{W}^{\gamma \comma \rho}}
       \leqslant \| (A^T)^{- 1} \|^{\rho} \, \| \Phi^w \|_{\mathcal{W}^{\gamma
       \comma \rho}} . \]
  \end{enumerate}
\end{lemma}

\begin{proof}
  All the statements follow from elementary calculations; let us prove only
  {\tmem{iii.}} and {\tmem{v.}} . Fix $\lambda \in (0, 1)$, then
  \[ \Phi_{s, t}^{w^{\lambda}} (\xi) = \int_s^t e^{i \xi \cdot \lambda^{- (1 -
     \gamma) / \rho} w_{\lambda r}} \mathd r = \lambda^{- 1} \int_{\lambda
     s}^{\lambda t} e^{i \lambda^{- (1 - \gamma) / \rho} \xi \cdot w_r} \mathd
     r = \lambda^{- 1} \Phi^w_{\lambda s, \lambda t} (\lambda^{- (1 - \gamma)
     / \rho} \xi), \]
  so that
  \[ \frac{| \Phi_{s, t}^{w^{\lambda}} (\xi) | | \xi |^{\rho}}{| t - s
     |^{\gamma}} = \lambda^{- 1} \frac{| \Phi^w_{\lambda s, \lambda t}
     (\lambda^{- (1 - \gamma) / \rho} \xi) | | \xi |^{\rho}}{| t - s
     |^{\gamma}} = \frac{| \Phi^w_{\lambda s, \lambda t} (\lambda^{- (1 -
     \gamma) / \rho} \xi) | | \lambda^{- (1 - \gamma) / \rho} \xi |^{\rho}}{|
     \lambda t - \lambda s |^{\gamma}} \leqslant \| \Phi^w
     \|_{\mathcal{W}^{\gamma, \rho}} . \]
  Regarding {\tmem{v.}}, similarly we have
  \[ \frac{| \Phi^{A w}_{s, t} (\xi) | | \xi |^{\rho}}{| t - s |^{\gamma}} =
     \left| \int_s^t e^{i (A^T \xi) \cdot w_r} \mathd r \right| \frac{| \xi
     |^{\rho}}{| t - s |^{\gamma}} \leqslant \| \Phi^w
     \|_{\mathcal{W}^{\gamma, \rho}} \left( \frac{| A^T \xi |}{| \xi |}
     \right)^{- \rho} \leqslant \| (A^T)^{- 1} \|^{\rho} \, \| \Phi^w
     \|_{\mathcal{W}^{\gamma \comma \rho}} . \]
\end{proof}

\begin{remark}
  There is a striking similarity between \tmtextit{i.}--\tmtextit{iii.} and
  properties of stochastic processes like self\mbox{-}similarity and
  stationarity; however the latter properties are of statistical nature,
  namely they preserve the \tmtextit{law} of the process, but not fixed
  trajectories, while $\rho$\mbox{-}irregularity is an analytical property
  which holds for deterministic trajectories. From property \tmtextit{iii.} we
  also deduce the existence of a critical scaling parameter associated to the
  pair $(\gamma, \rho)$ given by
  \begin{equation}
    \delta^{\ast}_{\gamma, \rho} = \frac{1 - \gamma}{\rho} ; \label{sec2.2
    critical parameter}
  \end{equation}
  we will see in Theorem~\ref{sec3 thm3} how $\delta^{\ast}_{\gamma, \rho}$
  relates to regularity of $(\gamma, \rho)$\mbox{-}irregular paths.
\end{remark}

\begin{remark}
  It's easy to see from the definition that $(\gamma,
  \rho)$\mbox{-}irregularity for $w$ is equivalent to the following: there
  exists a constant $C$ such that, for any $v \in \mathbb{S}^{d - 1}$, $v
  \cdot w$ is $(\gamma, \rho)$\mbox{-}irregular with $\| \Phi^{v \cdot w}
  \|_{\mathcal{W}^{\gamma, \rho}} \leqslant C.$ The latter condition is not
  equivalent to checking $(\gamma, \rho)$\mbox{-}irregularity of the
  coordinates $w^{(i)}$ (i.e. to $v = e_i$, $i = 1, \ldots, d$): for instance
  if $w$ is a $1$\mbox{-}dimensional $(\gamma, \rho)$-irregular function and
  we define $\tilde{w}_t \assign (w_t, - w_t)$, then the single coordinates of
  $\tilde{w}$ are $(\gamma, \rho)$-irregular but $\tilde{w}$ is not, since
  $(1, 1) \cdot \tilde{w} \equiv 0$.
\end{remark}

\begin{lemma}
  \label{sec2.2 lemma interpolation rho irr}Let $w$ be $(\gamma,
  \rho)$\mbox{-}irregular, then for any $\theta \in [0, 1]$ it is also
  $(\gamma^{\theta}, \rho^{\theta})$\mbox{-}irregular for the choice
  $\gamma^{\theta} = 1 - \theta + \theta \gamma$, $\rho^{\theta} = \theta
  \rho$ and it holds $\| \Phi^w \|_{\mathcal{W}^{\gamma^{\theta},
  \rho^{\theta}}} \leqslant \| \Phi^w \|^{\theta}_{\mathcal{W}^{\gamma,
  \rho}}$.
\end{lemma}

\begin{proof}
The conclusion follows immediately by interpolating the two inequalities
\begin{equation*}
| \Phi^w_{s, t} (\xi) | \leqslant | t - s |, \quad | \Phi^w_{s, t} (\xi)
     | \leqslant \| \Phi^w \|_{\mathcal{W}^{\gamma, \rho}} | t - s |^{\gamma}
     | \xi |^{- \rho} .
\end{equation*}
\end{proof}

The above lemma shows that we can always trade space regularity for time
regularity, i.e. we can decrease the parameter $\rho$ in order to increase
$\gamma$. Let us also point out that
\begin{equation}
  \delta^{\ast}_{\gamma, \rho} = \frac{1 - \gamma}{\rho} = \frac{1 -
  \gamma^{\theta}}{\rho^{\theta}} = \delta^{\ast}_{\gamma^{\theta},
  \rho^{\theta}}, \label{eq:critical-parameter}
\end{equation}
i.e. the critical scaling parameter is left unchanged by this procedure.

\

In general dimension $d \geqslant 2$, it is in general difficult to establish
if a given function is $(\gamma, \rho)$\mbox{-}irregular. This fact is one of
the main motivations of our interest in establishing the prevalence of this
property. The situation is different in the case $d = 1$, in which there are
simple conditions to establish $\rho$\mbox{-}irregularity (at least for some
values of $\rho$).

\begin{proposition}[Proposition 1.4 from {\cite{choukgubinelli1}}]
  \label{sec2.2 prop d=1}Let $w \in C^1 ([0, T] ; \mathbb{R})$ be such that
  $\inf_t | w'_t | \geqslant \delta > 0$ and $w'' \in L^1$, then $w$ is
  $(\gamma, 1 - \gamma)$\mbox{-}irregular for any $\gamma \in (0, 1)$.
\end{proposition}

Let us also introduce the following concept of \tmtextit{exponential
irregularity}, which is a novel contribution of the present work.

\begin{definition}
  \label{sec2.2 defn exp irregularity}A measurable path $w : [0, T]
  \rightarrow \mathbb{R}^d$ is exponentially irregular if there exist positive
  constants $c_1$, $c_2$ and $\gamma \in (0, 1)$ such that
  \begin{equation}
    | \Phi^w_{s, t} (\xi) | \leqslant c_1 e^{- c_2 | \xi |} | t - s |^{\gamma}
    \quad \forall \, \xi \in \mathbb{R}^d \setminus \{ 0 \}, \, s, t \in [0,
    T] . \label{sec2.2 defn exp irr}
  \end{equation}
\end{definition}

Contrary to Definition~\ref{sec2.2 defn rho irregularity}, in
Definition~\ref{sec2.2 defn exp irregularity} we do not impose $\gamma > 1 /
2$; this is because, arguing by interpolation as in Lemma~\ref{sec2.2 lemma
interpolation rho irr}, it is easy to check that if~\eqref{sec2.2 defn exp
irr} holds \tmtextit{some} $\gamma \in (0, 1)$, then the path $w$ actually
satisfies~\eqref{sec2.2 defn exp irr} for \tmtextit{all} $\gamma \in (0, 1)$,
up to modifying the constant $c_i = c_i (\gamma)$, $i = 1, 2$. It is then also
clear that exponential irregularity implies $\rho$-irregularity, for any value
$\rho \in (0, + \infty)$.

\

We now prove that the $\rho$\mbox{-}irregularity and exponential irregularity
properties define Borel sets, which is the first step in order to establish
prevalence.

\begin{lemma}
  \label{sec2.2 lemma borel rho irr}For any $\rho > 0$, the set
  \begin{equation*}
    A = \left\{ w : [0, T] \rightarrow \mathbb{R}^d \, | \text{$w$ is
    $\rho$-irregular} \right\}
  \end{equation*}
  is Borel measurable w.r.t to the topology induced by any of the following
  norms: $\| \cdot \|_{L^p}$, $p \in [1, \infty]$, $\| \cdot \|_{C^0}$, $\|
  \cdot \|_{C^{\alpha}}$, $\alpha \in (0, 1)$.
\end{lemma}

\begin{proof}
  We can write the set $A$ as the following countable union:
  \begin{equation*}
    A = \bigcup_{n, m \in \mathbb{N}} A_{n, m} = \bigcup_{n, m \in
    \mathbb{N}} \left\{ w : [0, T] \rightarrow \mathbb{R}^d \hspace{1em} |
    \hspace{1em} \sup_{\xi \in \mathbb{R}^d, s \neq t} \frac{| \Phi^w_{s, t}
    (\xi) | \, | \xi |^{\rho}}{| t - s |^{1 / 2 + 1 / m}} \leqslant n \right\}
    .
  \end{equation*}
  It will be then sufficient to show that for every $m, n$ the set $A_{m, n}$
  is closed in the above topologies. We will actually show that it is closed
  under convergence in measure, which is weaker than any of the norms
  considered and therefore yields the conclusion.
  
  Let $w_k$ be a sequence of elements of $A_{n, m}$ such that $w_k \rightarrow
  w$ in measure, then by dominated convergence it's easy to see that for any
  fixed $s < t$, $\xi \in \mathbb{R}^d$ it holds $\Phi^{w_k}_{s, t} (\xi)
  \rightarrow \Phi^w_{s, t} (\xi)$. But then
  \begin{equation*}
    \frac{| \Phi^w_{s, t} (\xi) | \, | \xi |^{\rho}}{| t - s |^{1 / 2 + 1/m}}
    = \lim_{k \rightarrow \infty} \frac{| \Phi^{w_k}_{s, t} (\xi) | \, |
    \xi |^{\rho}}{| t - s |^{1 / 2 + 1 / m}} \leqslant n \nonumber
  \end{equation*}
  and since the reasoning holds for any fixed $s<t$ and $\xi$ we can
  conclude that $f \in A_{n, m}$ as well.
\end{proof}

\begin{remark}
  \label{sec2.2 lemma borel general}More generally, given a modulus of
  continuity $\varphi$ and a function $F : \mathbb{R}^d \rightarrow
  \mathbb{R}^+$, the same proof shows that the set
  
  \begin{align}
    B = & \left\{ w : [0, T] \rightarrow \mathbb{R}^d \hspace{1em} |
    \hspace{1em} \sup_{\xi \in \mathbb{R}^d, s \neq t} \frac{| \Phi^w_{s, t}
    (\xi) | F (\xi)}{\varphi (| t - s |)} < \infty \right\} \nonumber
  \end{align}
  
  is Borel measurable in any of the above topologies. The fact that
  exponential irregularity defines Borel sets is established similarly.
\end{remark}

We conclude this section with a brief detour on Carath{\'e}odory functions and
their connection with the exponential irregularity property. Here $\lambda$
denotes the Lebesgue measure on $\mathbb{R}^d$, while $\mathcal{L}$ denotes
the Lebesgue measure on $[0, T]$.

\begin{definition}
  \label{sec2.2 defn charatheodory}A measurable function $f : [0, T]
  \rightarrow \mathbb{R}^d$ is a Carath{\'e}odory function if for any set $D
  \subset \mathbb{R}^d$ such that $\lambda (D) > 0$ it holds $\mathcal{L}
  (f^{- 1} (D) \cap B) > 0$ for every (non empty) interval $B$.
\end{definition}

Observe in particular that if $f$ is Carath{\'e}odory, then it is unbounded on
every interval and thus discontinuous.

\begin{lemma}
  \label{sec2.2 exp implies charatheodory}Let $w$ be an exponentially
  irregular measurable path, then $w$ is Carath{\'e}odory.
\end{lemma}

\begin{proof}
  The statement follows immediately from the considerations given at the
  beginning of Section~6 from~{\cite{berman69}}, see also Sections~11 and~28
  from~{\cite{geman}}. Let us briefly sketch the proof.
  
  Denoting by $\mu_{s, t}^w$ the occupation measure associated to $w$, which
  is defined in Definition~\ref{sec2.3 defn occupation measure} below, it
  follows from the exponentially irregularity property that we can find $c >
  0$ such that for any $s < t$ it holds
  \[ \int_{\mathbb{R}^d} e^{c | \xi |} | \hat{\mu}_{s, t}^w (\xi) |^2 \mathd
     \xi = \int_{\mathbb{R}^d} e^{c | \xi |} | \Phi^w_{s, t} (\xi) |^2 \mathd
     \xi < \infty . \]
  It then follows from the Paley--Wiener Theorem that $\mu^w_{s, t}$ is
  analytic and therefore it cannot vanish on any set $D \subset \mathbb{R}^d$
  such that $\lambda (D) > 0$, namely for $B = [s, t]$ it holds
  \[ \mathcal{L} (w^{- 1} (D) \cap [s, t]) = \int_D \mu_{s, t}^w (y) \mathd y
     > 0. \]
\end{proof}

\subsection{Occupation measures and averaging operators}\label{sec2.3}

So far we have discussed several properties of $\rho$\mbox{-}irregularity, but
we haven't motivated the importance of such notion and its relation with
regularisation by noise phenomena. It turns out that
$\rho$\mbox{-}irregularity is closely tied to the \tmtextit{occupation
measure} of the path $w$.

\

In the following $\mathcal{M} (\mathbb{R}^d)$ denotes the set of all finite
Radon measures on $\mathbb{R}^d$, endowed with the total variation norm $\|
\cdot \|_{\tmop{TV}}$; $\mathcal{M}_+ (\mathbb{R}^d)$ is the closed subset of
non\mbox{-}negative measures.

\begin{definition}
  \label{sec2.3 defn occupation measure}Given a measurable path $w : [0, T]
  \rightarrow \mathbb{R}^d$, we define its occupation measure as the family
  $(\mu_{s, t}^w)_{0 \leqslant s \leqslant t \leqslant T} \subset
  \mathcal{M}_+ (\mathbb{R}^d)$ given by $\mu^w_{s, t} = w_{\ast}
  (\mathcal{L}_{[s, t)})$, namely
  \[ \int_{\mathbb{R}^d} f (y) \, \mu_{s, t}^w (\mathd y) = \int_{[s, t)} f
     (w_r) \, \mathd r \quad \forall \, f \in C^0_b (\mathbb{R}^d) . \]
  Observe that by definition $\mu^w_{s, t} = \mu^w_{0, t} - \mu^w_{0, s}$; for
  this reason we will identity the family $(\mu_{s, t}^w)_{s \leqslant t}$
  with the map $\mu^w \in C^0 ([0, T] ; \mathcal{M}_+ (\mathbb{R}^d))$ given
  by $t \mapsto \mu^w_t = \mu^w_{0, t}$, so that $\mu^w_{s, t}$ represents an
  increment of $\mu^w_{\cdot}$.
\end{definition}

Note that $\mu^w \in \tmop{Lip} ([0, T] ; \mathcal{M}_+ (\mathbb{R}^d))$ with
$\| \mu^w_{s, t} \|_{\tmop{TV}} = | t - s |$ and Gateaux derivative
$\dot{\mu}^w_t = \delta_{w_t}$.

\

The Fourier transform of $\mu^w_{s, t}$ is given by
\[ \widehat{\mu_{s, t}^w} (\xi) = \int_{\mathbb{R}^d} e^{- i \xi \cdot y} \,
   \mu_{s, t}^w (\mathd y) = \int_s^t e^{- i \xi \cdot w_r} \mathd r =
   \overline{\Phi^w_{s, t} (\xi)} \]
which shows that $w$ is $(\gamma, \rho)$\mbox{-}irregular if and only if the
map $t \mapsto \mu^w_t$ belongs to $C^{\gamma} ([0, T] ; \mathcal{F} L^{\rho,
\infty} (\mathbb{R}^d))$, where the Fourier--Lebesgue spaces $\mathcal{F}
L^{\rho, p}$ with $p \in [1, \infty]$ are defined by
\[ \mathcal{F} L^{\rho, p} = \left\{ f \in \mathcal{S}' (\mathbb{R}^d) \, :
   \langle \xi \rangle^{\rho} | \hat{f} (\xi) | \in L^p \right\}, \]
see Appendix~\ref{appendixA2} for their main properties. In particular we have
\begin{equation}
  \| \Phi^w \|_{\mathcal{W}^{\gamma, \rho}} \sim \| \mu^w \|_{C^{\gamma}
  \mathcal{F} L^{\rho, \infty}} . \label{eq:link-occupatio-rho}
\end{equation}
\begin{remark}
  \label{sec2.3 remark equivalence class}We will mostly work with given
  measurable paths $w$, but both definitions of $\Phi^w$ and $\mu^w$ are not
  affected by changing $w$ on an $\mathcal{L}_{[0, T]}$\mbox{-}negligible
  subset of $[0, T]$; therefore they also makes sense when dealing with
  equivalence classes like $w \in L^p (0, T ; \mathbb{R}^d)$ for $p \in [1,
  \infty]$. Similarly, it makes sense for $w$ in an equivalence class to say
  that it is $(\gamma, \rho)$\mbox{-}irregular (resp. exponentially
  irregular).
\end{remark}

Occupation measures are also closely related to averaging operators.

\begin{definition}
  \label{sec2.3 defn averaging operator}Let $w : [0, T] \rightarrow
  \mathbb{R}^d$ be a measurable bounded function, then we define the averaging
  operator associated to $w$ as the family of operators $\{ T^w_{s, t}, 0
  \leqslant s \leqslant t \leqslant T \} \subset \mathcal{L} (\mathcal{S}'
  (\mathbb{R}^d) ; \mathcal{S}' (\mathbb{R}^d))$ given by
  \[ T^w_{s, t} b = \int_s^t b \left( \cdot \, + w_r \right) \mathd r \]
  or equivalently by duality as follows: for any $\varphi \in \mathcal{S}
  (\mathbb{R}^d)$ and any $b \in \mathcal{S}' (\mathbb{R}^d)$ it holds
  \begin{equation}
    \langle T^w_{s, t} b, \varphi \rangle = \langle b, \int_s^t \varphi \left(
    \cdot \, - w_r \right) \mathd r \rangle . \label{sec2.3 defn averaging
    duality}
  \end{equation}
  As before, $T^w_{s, t} = T^w_{0, t} - T^w_{0, s}$ and therefore we identify
  $(T^w_{s, t})_{s \leqslant t}$ with the map $t \mapsto T^w_t = T^w_{0, t}$.
\end{definition}

\begin{remark}
  Differently from Definitions~\ref{sec2.2 defn rho irregularity}
  and~\ref{sec2.3 defn occupation measure}, in Definition~\ref{sec2.3 defn
  averaging operator} we have required that $w$ is bounded; the reason for
  this is that otherwise it is unclear a priori if, for a given $\varphi \in
  \mathcal{S} (\mathbb{R}^d)$, $\int_s^t \varphi \left( \cdot \, - w_r \right)
  \mathd r \in \mathcal{S} (\mathbb{R}^d)$ uniformly in $s < t$ and so if the
  above is a good definition. However by looking at the Fourier transform
  $\widehat{\varphi \left( \cdot \, - w_t \right)} = e^{- i \xi \cdot w_t} 
  \hat{\varphi}$ one can check that $w \in L^{\infty}$ can be relaxed to
  requiring
  \[ \int_0^T | w_t |^N \mathd t < \infty \quad \text{for all } N \in
     \mathbb{N}, \]
  namely $w \in L^p_t$ for all $p < \infty$. It is still true, analogously to
  Remark~\ref{sec2.3 remark equivalence class}, that $T^w = T^{\tilde{w}}$ if
  $w = \tilde{w}$ up to $\mathcal{L}_{[0, T]}$\mbox{-}negligible sets; for
  this reason from now on when dealing with $T^w$ we will always implicitly
  assume $w \in L^{\infty}_t$.
\end{remark}

Averaging operators can also be defined for time\mbox{-}dependent
distributions, as done in~{\cite{galeatigubinelli_ode}}. In the time dependent
case however we lose the following fundamental property which relates the
averaging operator to the occupation measure.

\begin{lemma}
  \label{sec2.3 lemma averaging occupation}Let $w \in L^{\infty}$, $\mu^w$ and
  $T^w$ as above. Then for any $b \in \mathcal{S}' (\mathbb{R}^d)$, $T^w_{s,
  t} b = \tilde{\mu}^w_{s, t} \ast b$, where $\tilde{\mu}$ denotes the
  reflection of $\mu$, i.e. $\tilde{\mu}^w_{s, t} (A) = \mu^w_{s, t} (- A)$.
\end{lemma}

\begin{proof}
  Observe that by definition of occupation measure, for any $s \leqslant t$
  and any $R \geqslant \| w \|_{L^{\infty}}$, it holds $\tmop{supp} \mu^w_{s,
  t} \subset B_R$. In particular since $\mu^w_{s, t}$ is a measure with
  compact support, for any $b \in \mathcal{S}' (\mathbb{R}^d)$ the convolution
  $b \ast \mu^w_{s, t}$ is well defined. The same goes for $\tilde{\mu}^w_{s,
  t}$. Let $\varphi \in \mathcal{S} (\mathbb{R}^d)$, then for any $x \in
  \mathbb{R}^d$ it holds
  \[ \int_s^t \varphi (x - w_r) \mathd r = \int \varphi (x - y) \mu^w_{s, t}
     (\mathd y) = (\varphi \ast \mu^w_{s, t}) (x) . \]
  The conclusion for $b$ follows by the duality formula~\eqref{sec2.3 defn
  averaging duality} and $\langle b, \varphi \ast \mu^w_{s, t} \rangle =
  \langle \tilde{\mu}_{s, t}^w \ast b, \varphi \rangle$.
\end{proof}

As a consequence, estimating the regularisation properties of $T^w$ is
equivalent to estimating the regularity of $\mu^w$ in suitable function
spaces. This is exactly where the notion of $\rho$\mbox{-}irregularity comes
into play.

\begin{lemma}
  \label{sec2.2 lemma regularity averaging 1}Let $w \in L^{\infty} (0, T ;
  \mathbb{R}^d)$ be $(\gamma, \rho)$\mbox{-}irregular. Then for any $\alpha
  \in \mathbb{R}$ and $p \in [1, \infty]$, the averaging operator $T^w$
  belongs to $C^{\gamma}_t \mathcal{L} (\mathcal{F} L^{\alpha, p} ;
  \mathcal{F} L^{\alpha + \rho, p})$ and for any $b \in \mathcal{F} L^{\alpha,
  p}$ it holds
  \begin{equation}
    \| T^w_{s, t} b \|_{\mathcal{F} L^{\alpha + \rho, p}} \lesssim | t - s
    |^{\gamma} \| b \|_{\mathcal{F} L^{\alpha, p}} \| \Phi^w
    \|_{\mathcal{W}^{\gamma, \rho}} .
  \end{equation}
\end{lemma}

\begin{proof}
  The statement follows from the considerations given in the introduction
  of~{\cite{catelliergubinelli}}. Alternatively, using Lemma~\ref{appendixA2
  fourier lebesgue convolution} from Appendix~\ref{appendixA2}, it holds
  \[ \| T^w_{s, t} b \|_{\mathcal{F} L^{\alpha + \rho, p}} = \|
     \tilde{\mu}^w_{s, t} \ast b \|_{\mathcal{F} L^{\alpha + \rho, p}}
     \leqslant \| \tilde{\mu}^w_{s, t} \|_{\mathcal{F} L^{\rho, \infty}}  \| b
     \|_{\mathcal{F} L^{\alpha, p}} \lesssim | t - s |^{\gamma} \| b
     \|_{\mathcal{F} L^{\alpha, p}} \| \Phi^w \|_{\mathcal{W}^{\gamma, \rho}}
     . \]
\end{proof}

Unfortunately, Fourier--Lebesgue spaces are not always very useful, with the
exclusion of the scale $p = 2$, in which case $\mathcal{F} L^{\alpha, 2} =
H^{\alpha}$. We can however use Fourier--Lebesgue embeddings to deduce
regularity for $\mu^w_{s, t}$ in other scales of spaces, which in turn imply
different estimates for $T^w_{s, t}$. To this end, following~{\cite{geman}},
we introduce the concept of \tmtextit{occupation density}. In the
probabilistic literature it is usually referred to as \tmtextit{local time}
and we will indifferently use both terminologies.

\begin{definition}
  \label{sec2.2 defn occupation density}We say that a measurable $w : [0, T]
  \rightarrow \mathbb{R}^d$ admits an occupation density if for any $s < t$,
  $\mu^w_{s, t}$ is absolutely continuous w.r.t. $\mathcal{L}_{[s, t)}$, in
  which case we denote by $\ell^w_{s, t}$ its density, so that $\mu^w_{s, t}
  (\mathd x) = \ell^w_{s, t} (x) \mathd x$. As usual it holds $\ell^w_{s, t} =
  \ell^w_{0, t} - \ell^w_{0, s}$ and we set $\ell^w_{0, t} = \ell^w_t$.
  Sometimes we will also use the notation $\ell^w_t (x) = \ell^w (t, x)$.
\end{definition}

The spaces $B^s_{p, q} = B^s_{p, q} (\mathbb{R}^d)$ appearing in the next
lemma are inhomogeneous Besov spaces, see~{\cite{bahouri}}. Similar statements
can be given for more classical Sobolev spaces $W^{k, p}$ or Bessel potential
spaces $L^{s, p} = (1 - \Delta)^{- s / 2} L^p$.

\begin{lemma}
  \label{sec2.2 lemma regularity averaging 2}Let $w : [0, T] \rightarrow
  \mathbb{R}^d$ be a $(\gamma, \rho)$\mbox{-}irregular measurable path. Then:
  \begin{enumerate}[label=\roman*.]
    \item If $\rho > d / 2$, then $w$ admits an occupation density $\ell^w \in
    C^{\gamma}_t L^2_x \cap \tmop{Lip}_t L^1_x$.
    
    \item If $\rho > d$, then $\ell^w$ is jointly continuous in $(t, x)$ and
    $\ell^w \in C^{\gamma}_t C^0_x$.
    
    \item If $\rho > d / 2 + k$ for some $k \in \mathbb{N}$, then $\ell^w \in
    C^{\gamma}_t W^{k, 2}_x$; in particular, if $w \in L^{\infty}$, then
    $\ell$ is compactly supported on $[0, T] \times \mathbb{R}^d$ and
    therefore $\ell^w \in C^{\gamma}_t W^{k, 1}_x$.
    
    \item As a consequence, if $\rho > d / 2 + k$ for some $k \in \mathbb{N}$
    and $w \in L^{\infty}$, then for any $\alpha \in \mathbb{R}$, $p, q \in
    [1, \infty]$ it holds
    \[ T^w \in C^{\gamma} ([0, T] ; \mathcal{L} (B^{\alpha}_{p, q} ; B^{\alpha
       + k}_{p, q})) \]
    where $B^{\alpha}_{p, q}$ denote Besov spaces. In particular, for any $b
    \in B^{\alpha}_{p, q}$ it holds
    \[ \| T^w_{s, t} b \|_{B^{\alpha + k}_{p, q}} \lesssim | t - s |^{\gamma}
       \| b \|_{B^{\alpha}_{p, q}} \| w \|_{L^{\infty}}^{d / 2} \| \Phi^w
       \|_{\mathcal{W}^{\gamma, \rho}} \quad \text{uniformly in } s < t. \]
  \end{enumerate}
\end{lemma}

\begin{proof}
  By equation~\eqref{eq:link-occupatio-rho} and the Fourier--Lebesgue
  embedding $\mathcal{F} L^{\rho, \infty} \hookrightarrow \mathcal{F} L^{0, 2}
  = L^2$ which holds for $\rho > d / 2$, we deduce that $\mu^w \in
  C^{\gamma}_t L^2_x$, i.e. for any $s < t$ the measure $\mu_{s, t}^w$ can be
  identified with a function in $L^2$, which is exactly $\ell^w_{s, t}$.
  Moreover $\mu^w_{s, t}$ is a positive measure with total variation $\|
  \mu^w_{s, t} \|_{\mathcal{M}} = t - s$, which implies that $\ell^w_{s, t}
  \in L^1_x$ with $\| \ell^w_{s, t} \|_{L^1} = t - s$. From this
  follows~\tmtextit{i.}; \tmtextit{ii.} and the first part of \tmtextit{iii.}
  follow similarly by using the embeddings $\mathcal{F} L^{\rho, \infty}
  \hookrightarrow \mathcal{F} L^{0, 1} \hookrightarrow C^0_b$, valid for $\rho
  > d$, and $\mathcal{F} L^{\rho, \infty} \hookrightarrow \mathcal{F} L^{\rho
  - d / 2, 2} \hookrightarrow W^{k, 2}_x$, valid for $\rho - d / 2 > k$. The
  second half of \tmtextit{iii.} follows from the fact that if $w \in
  L^{\infty}$, then $\mu^w_{s, t}$ is supported on $B_{\| w \|_{L^{\infty}}}$
  and so we have the estimate $\| \ell^w_{s, t} \|_{W^{k, 1}_x} \lesssim \| w
  \|_{L^{\infty}}^{d / 2} \| \ell^w_{s, t} \|_{W^{k, 2}}$. Finally statement
  \tmtextit{iv.} can be deduced from a combination of the previous estimate,
  the general Young\mbox{-}type inequality
  \[ \| f \ast g \|_{B^{\alpha + k}_{p, q}} \lesssim \| f \|_{B^{\alpha}_{p,
     q}}  \| g \|_{W^{k, 1}}, \]
  the identity $T^w_{s, t} b = \tilde{\mu}^w_{s, t} \ast b$ and the estimate
  $\| \mu^w_{s, t} \|_{W^{k, 1}} \lesssim \| w\|_{L^\infty}^{d/2} \| \mu_{s, t}^w \|_{\mathcal{F} L^{\rho,
  \infty}} \lesssim | t - s |^{\gamma}  \| \mu^w \|_{C^{\gamma} \mathcal{F}
  L^{\rho, \infty}}$.
\end{proof}

\begin{remark}
  We stated for simplicity Points~\tmtextit{iii.} and~\tmtextit{iv.} of
  Lemma~\ref{sec2.2 lemma regularity averaging 2} for $k \in \mathbb{N}$, so
  that the definition of $W^{k, 1}$ is classical and unambiguous. However, up
  to coming up with a suitable concept for $W^{k, 1}$ for $k \in (0, \infty)
  \setminus \mathbb{N}$ (e.g. replacing it with ${B^k_{1, 1}} $), one can
  immediately extend the statements to the case of general $k \in
  \mathbb{R}_{\geqslant 0}$, since Fourier--Lebesgue embedding $\mathcal{F}
  L^{\rho} \hookrightarrow W^{k, 1}$ and Young\mbox{-}type inequality for
  convolutions still hold.
\end{remark}

The reader might wonder if the restriction $\rho > d / 2$ appearing in
Lemma~\ref{sec2.2 lemma regularity averaging 2} can be weakened, given that it
does not appear when dealing with the scales of spaces $\mathcal{F} L^{\alpha,
p}$ as in Lemma~\ref{sec2.2 lemma regularity averaging 1}, nor in more
probabilistic $\mathbb{P}$-a.s. statements where $b$ is fixed and $w$ is
sampled as a stochastic process (cf. the results
from~{\cite{catelliergubinelli,galeatigubinelli_ode}}). The next statement,
which can be regarded as a useful rewriting of Remark~3.5
from~{\cite{perkowski}}, strongly suggests this not to be possible in general;
indeed if the averaging operator $T^w$ acts in a regularising way over the
Besov--H\"{o}lder scales $B^{\alpha}_{\infty, \infty}$, then $w$ must be
necessarily enjoy limited regularity.

\begin{lemma}
  \label{lem:perkowski}Let $w \in C^{\delta} ([0, T] ; \mathbb{R}^d)$ for some $\delta\in (0,1]$; suppose
  that there exist $\alpha \in \mathbb{R}$, $\varepsilon > 0$ s.t. $T^w_{0, T}
  \in \mathcal{L} (B^{\alpha}_{\infty, \infty}, B^{\alpha +
  \varepsilon}_{\infty, \infty}) .$ Then it must hold $\delta \leqslant 1 /
  d$.
\end{lemma}

\begin{proof}
  Without loss of generality we can assume $\varepsilon \leqslant 1$. Observe that, since averaging
  and convolutions commute, i.e. $T^w (K \ast b) = K \ast (T^w b)$ (cf.
  Lemma~3.3 from~{\cite{galeatigubinelli_ode}}), if the hypothesis holds for
  \tmtextit{some} $\alpha \in \mathbb{R}$, then it holds for \tmtextit{all}
  $\alpha \in \mathbb{R}$ (one can take for instance $K = \Delta_j$ to be
  Littlewood--Paley blocks).
  
  Next, choose $\alpha = \varepsilon / 2$ and $p = 2 d / \varepsilon$, so that
  we have the embedding $L^p \hookrightarrow B^{- \varepsilon / 2}_{\infty,
  \infty}$, and observe that for any $f \in \mathcal{S}$ it holds
  \[ | \langle f, \tilde{\mu}_{0, T}^w \rangle | = | (f \ast \tilde{\mu}^w_{0,T}) (0)
     | \lesssim \| f \ast \tilde{\mu}_{0, T}^{w} \|_{B^{\varepsilon
     / 2}_{\infty, \infty}} = \| T^w_{0, T} f \|_{B^{\varepsilon / 2}_{\infty,
     \infty}} \lesssim \| f \|_{B^{- \varepsilon / 2}_{\infty, \infty}}
     \lesssim \| f \|_{L^p} ; \]
  by density the inequality extends to all $f \in L^p$ and by duality we can
  conclude that there exists $g \in L^{p'}$, $p'$ being the conjugate exponent
  to $p$, such that $\langle f, \tilde{\mu}^w_{0, T} \rangle = \langle f,
  \tilde{g} \rangle$ for all $f \in L^p$. This necessarily implies that
  $\mu^w_{0, T} (\mathd x) = g (x) \mathd x$; since $w$ is continuous,
  $\mu^w_{0, T}$ is compactly supported, therefore it holds $g = \ell^w_{0, T}
  \in L^1 \cap L^{p'}$. Overall we have concluded that $\mu^w_{0, T}$ admits
  an occupation density $\ell^w_{0, T} \in L^1$.
  
  Since $\mu^w_{0, T} (w ([0, T])) = T$ and $\mu^w_{0, T} \ll \lambda$,
  where $\lambda$ is the Lebesgue measure on $\mathbb{R}^d$, the set $w ([0, T])$ must have
  positive Lebesgue measure; by standard facts from geometric measure theory,
  $w ([0, T])$ then has Hausdorff dimension $d$. The conclusion then follows
  from the relation $d=\text{dim}_H (w ([0, T])) \leqslant \delta^{- 1}$ (cf.
  Proposition~3.3~(a) from~{\cite{falconer}}).
\end{proof}

Slightly anticipating upcoming results, it is interesting to observe that the
restriction $\delta \leqslant 1 / d$ matches almost exactly the sufficient
condition $\rho > d / 2$ given in Lemma~\ref{sec2.2 lemma regularity averaging
2}\mbox{-}\tmtextit{i.} for the local time to exist, once we combine it with
the fact (from Theorem~\ref{sec3 thm2} below) that a.e. $w \in C^{\delta}_t$
is $\rho$-irregular for any $\rho < 1 / (2 \delta)$.

\begin{remark}
  \label{rem:tolomeo}It follows from Lemma~\ref{lem:perkowski} and the
  upcoming Theorem~\ref{sec3 thm2} that one can find many examples of paths $w
  : [0, T] \rightarrow \mathbb{R}^d$ whose averaging $T^w b$ has a
  regularizing effect when acting on the scales $\mathcal{F} L^{\alpha, p}$,
  but none on the scales $B^{\alpha}_{\infty, \infty}$. This issue boils down
  to the fact that in general convolutional operators with polynomially
  decaying Fourier transform are not necessarily nice Fourier multiplies, with
  strikingly similar examples coming from dispersive PDEs. Indeed, if one
  considers the linear wave PDE on $\mathbb{R}^d$
  \[ \partial_t^2 u = \Delta u, \quad u|_{t = 0} = 0, \quad \partial_t u|_{t =
     0} = f, \]
  then the solution at time $t > 0$ is given by the linear operator
  \[ T_t f = \frac{\sin (t | \nabla |)}{| \nabla |} f, \]
  in the sense that $T_t$ acts in Fourier space by multiplying $\hat{f}$ by
  $\sin (t | \xi |) / | \xi |$. It is then clear that $T_t$ maps $H^s$ into
  $H^{s + 1}$ for any $s \in \mathbb{R}$; but if one tries to have a similar
  result in $L^p$-based results, this is not true in general, with explicit
  counterexamples given in~{\cite{peral1970}}. What is more, when $d = 3$,
  $T_t$ has the alternative harmonic mean representation $T_t f = \mu_t \ast
  f$, where $\mu_t$ is the unit measure on the sphere $S_t = \{ x \in
  \mathbb{R}^d : | x | = t \}$, which is singular w.r.t. the Lebesgue measure;
  in particular, a similar argument to the one from Lemma~\ref{lem:perkowski}
  readily implies that $T_t$ cannot map $B^{\alpha}_{\infty, \infty}$ into
  $B^{\alpha + \varepsilon}_{\infty, \infty}$ for any $\alpha \in \mathbb{R}$
  and $\varepsilon > 0$. Roughly speaking $T_t$ here behaves exactly like
  $T^w_t$ would, in the case where the path $w$ is $1$-irregular but also
  $\delta$-H\"{o}lder continuous with $\delta > 1 / d$ (e.g. for $d = 3$,
  typical realization of fBm with $H \in (1 / 3, 1 / 2)$ would work by virtue
  of the upcoming Theorem~\ref{sec2.4 thm catelliergubinelli}).
\end{remark}

In general, as pointed out in Remark~7 of~{\cite{galeatigubinelli_ode}}, the
operator $T^w$ cannot regularise time\mbox{-}dependent fields $b = b (t, x)$,
at least not uniformly in all $b \in C^0_t E$ for suitable Banach spaces $E$.
Intuitively, the reason is that the oscillations in time of $b$ could
compensate the oscillations of $w$ and limit the regularisation effect.
However if $b$ is required to be sufficiently regular in $t$, it is still
possible to obtain a regularisation effect, as we are going to show now.

\begin{lemma}
  \label{sec2.2 lemma averaging time}Let $w \in L^{\infty}$ be $(\gamma,
  \rho)$\mbox{-}irregular, $b \in C^{\beta}_t \mathcal{F} L^{\alpha, p}$ with
  $\beta > 1 - \gamma$. Then $T^w b \in C^{\gamma}_t \mathcal{F} L^{\alpha +
  \rho, p}$ and there exists a constant $C = C (\gamma + \beta, T) > 0$ such
  that
  \begin{equation}
    \| T^w_{s, t} b \|_{\mathcal{F} L^{\alpha + \rho, p}} \leqslant C \| b
    \|_{C^{\beta} \mathcal{F} L^{\alpha, p}} \| \Phi^w
    \|_{\mathcal{W}^{\gamma, \rho}} | t - s |^{\gamma} \quad \tmop{for}
    \tmop{all} \quad 0 \leqslant s \leqslant t \leqslant T. \label{sec2.2
    lemma time eq1}
  \end{equation}
  Namely, the linear map $T^w : C^{\beta}_t \mathcal{F} L^{\alpha, p}
  \rightarrow C^{\gamma}_t \mathcal{F} L^{\alpha, p + \rho}$ is bounded with
  constant $C \| \Phi^w \|_{\mathcal{W}^{\gamma, \rho}}$.
\end{lemma}

\begin{proof}
  Let us first assume that $b$ is a smooth function. In this case for any $[s,
  t] \subset [0, T]$ and any sequence $\Pi$ of partitions of $[s, t]$ with
  infinitesimal mesh, it holds
  \begin{align*}
    T^w_{s, t} b (x) & = \int_s^t b (r, x + w_r) \mathd r = \lim_{| \Pi |
    \rightarrow 0} \sum_i \int_{t_i}^{t_{i + 1}} b (t_i, x + w_r) \mathd r\\
    & = \lim_{| \Pi | \rightarrow 0} \sum_i T^w_{t_i, t_{i + 1}} [b (t_i,
    \cdot)] (x) = \lim_{| \Pi | \rightarrow 0} \sum_i (b_{t_i} \ast
    \tilde{\mu}^w_{t_i, t_{i + 1}}) (x)
  \end{align*}
  which implies the functional equality
  \[ T^w_{s, t} b = \lim_{| \Pi | \rightarrow 0} \sum_i b_{t_i} \ast
     \tilde{\mu}^w_{t_i, t_{i + 1}} . \]
  Since $b \in C^{\beta}_t \mathcal{F} L^{\alpha, p}$, $\mu^w \in C^{\gamma}_t
  \mathcal{F} L^{\rho, \infty}$, $\gamma + \beta > 1$ and the map $\ast : (f,
  g) \mapsto f \ast g$ is bilinear and bounded from $\mathcal{F} L^{\alpha, p}
  \times \mathcal{F} L^{\rho, \infty}$ into $\mathcal{F} L^{\alpha + \rho,
  p}$, it follows from Young integration in Banach spaces (see
  Appendix~\ref{appendixA1}) that
  \[ T^w_{s, t} b = \int_s^t b_r \ast \tilde{\mu}^w_{\mathd r}, \]
  as well as estimate~\eqref{sec2.2 lemma time eq1}, since $\| \tilde{\mu}^w
  \|_{C^{\gamma}_t \mathcal{F} L^{\rho, \infty}} = \| \mu^w \|_{C^{\gamma}_t
  \mathcal{F} L^{\rho, \infty}} \sim \| \Phi^w \|_{\mathcal{W}^{\gamma,
  \rho}}$. The case of general $b$ follows from approximation procedures:
  given $b \in C^{\beta}_t \mathcal{F} L^{\alpha, p}$, we can find a sequence
  $b^n$ of smooth functions such that $\| b^n \|_{C^{\beta} \mathcal{F}
  L^{\alpha, p}} \leqslant \| b \|_{C^{\beta} \mathcal{F} L^{\alpha, p}}$ and
  that for any $\varepsilon > 0$ $b^n \rightarrow b$ locally in $C^{\beta -
  \varepsilon}_t \mathcal{F} L^{\alpha - \varepsilon, p}$; by properties of
  averaging, for any $s < t$ $T^w_{s, t} b^n$ converges to $T^w_{s, t} b$
  weakly\mbox{-}$\ast$ in $\mathcal{F} L^{\alpha, p}$. The conclusion then
  follows from taking the liminf as $n \rightarrow \infty$ on both sides
  of~\eqref{sec2.2 lemma time eq1} applied to $b^n$ and using the Fatou
  property of weak\mbox{-}$\ast$ convergence.
\end{proof}

\begin{remark}
  \label{sec2.2 remark averaging time}It is clear that the proof can be
  readapted in a more general setting: given $E, F, G$ function spaces such
  that $\ast : (f, g) \mapsto f \ast g$ is a bilinear bounded map from $E
  \times F$ into $G$, if $\mu^w \in C^{\gamma}_t F$ and $\beta > 1 - \gamma$,
  then $T^w : C^{\beta}_t E \rightarrow C^{\gamma}_t G$ \ is a linear bounded
  map. This can be applied in combination with Lemma~\ref{sec2.2 lemma
  regularity averaging 2}, obtaining regularising effects of $T^w$ when $E$
  and $G$ are taken in suitable Besov scales.
\end{remark}

\subsection{Fractional Brownian motion and
local\mbox{-}nondeterminism}\label{sec2.4}

The main tool in order to establish our prevalence results will be a family of
measures associated to locally nondeterministic Gaussian processes. The main
examples from this family are the laws $\mu^H$ of fractional Brownian motions
(fBm) of Hurst parameter $H \in (0, 1)$. General references on fBm
are~{\cite{nualart2006}} and~{\cite{picard}}.

\

A one dimensional fBm $(W^H_t)_{t \geqslant 0}$ of Hurst parameter $H \in (0,
1)$ is a centred Gaussian process with covariance
\[ \mathbb{E} [W^H_t W^H_s] = \frac{1}{2} (| t |^{2 H} + | s |^{2 H} - | t - s
   |^{2 H}) . \]
Up to a multiplicative constant, it is the unique centred Gaussian process
with stationary increments and $H$\mbox{-}self-similarity, i.e. with law
invariant under the scaling $\tilde{W}^H_{\cdot} \assign \lambda^{- H}
W^H_{\lambda \cdot}$ for any $\lambda > 0$.

When $H = 1 / 2$, it coincides with standard Brownian motion, while for $H
\neq 1 / 2$ it is not a semi-martingale nor a Markov process. However it
shares many properties of Brownian motion, for instance fBm trajectories are
$\mathbb{P}$-a.s. $\alpha$-H{\"o}lder continuous for any $\alpha < H$ and
nowhere $\alpha$-H{\"o}lder continuous for any $\alpha \geqslant H$. A
$d$-dimensional fBm $W^H$ of Hurst parameter $H \in (0, 1)$ is an
$\mathbb{R}^d$\mbox{-}valued Gaussian process with components given by
independent one dimensional fBms.

\

Given a two-sided Brownian motion $(B_t)_{t \in \mathbb{R}}$, a fBm of
parameter $H \neq 1 / 2$ can be constructed by the formula
\begin{equation}
  W^H_t = c_H \int_{- \infty}^t [(t - r)^{H - 1 / 2}_+ - (- r)_+^{H - 1 / 2}]
  \, \mathd B_r \label{sec2.4 non canonical representation}
\end{equation}
where $c_H = \Gamma (H + 1 / 2)^{- 1}$ is a suitable renormalising constant
and the integral is in the It\^{o} sense. As a consequence, for any $0
\leqslant s < t$, the variable $W^H_t$ decomposes into the sum $W^H_t = W^{1,
H}_{s, t} + W^{2, H}_{s, t}$, where
\[ W^{1, H}_{s, t} : = c_H \int_s^t (t - r)^{H - 1 / 2} \, \mathd B_r, \quad
   W^{2, H}_{s, t} : =\mathbb{E} [W^H_t | \mathcal{F}_s] = c_H
   \int_{- \infty}^s [(t - r)^{H - 1 / 2}_+ - (- r)_+^{H - 1 / 2}] \, \mathd
   B_r ; \]
specifically, $W^{2, H}_{s, t}$ is $\mathcal{F}_s$-measurable, while $W^{1,
H}_{s, t}$ is Gaussian, independent of $\mathcal{F}_s$ and with variance
\[ \tmop{Var} (W^{1, H}_{s, t}) = \tilde{c}_H | t - s |^{2 H}, \]
where $\tilde{c}_H = c_H^2 / (2 H)$. In particular this implies that
\begin{equation}
  \tmop{Var} (W^H_t | \sigma (W^H_r, r \leqslant s \nobracket)) \geqslant
  \tmop{Var} (W^H_t | \mathcal{F}_s \nobracket) = \tmop{Var} (W^{1, H}_{s, t})
  = \tilde{c}_H | t - s |^{2 H} . \label{sec2.4 lnd fbm}
\end{equation}
Equation~\eqref{sec2.4 lnd fbm} is known in the literature as a
local\mbox{-}nondeterminism (LND) property. LND was first introduced by Berman
in~{\cite{berman73}} in order to analyse the properties of the occupation
measure (more precisely the local time) of Gaussian processes. Loosely
speaking, it means that for any $s < t$, the increment $W_t^H - W^H_s$
contains a part which is independent of the the history of the path
$W^H_{\cdot}$ up to time $s$ and therefore makes the path $W^H_{\cdot}$
``intrinsically chaotic''.

\

There is now a huge literature on local\mbox{-}nondeterminism and several
alternative definitions, which are not in general equivalent,
see~{\cite{xiao}} for a survey; here we identify three types of LND which are
closely tied with $\rho$\mbox{-}irregularity and exponential irregularity of
sample paths of Gaussian processes. They will play a major role in the proofs
respectively of Sections~\ref{sec4.2} and~\ref{sec4.4}.

\begin{definition}
  \label{sec2.4 defn strong lnd}Let $(X_t)_{t \in [0, T]}$ be an
  $\mathbb{R}^d$-valued separable Gaussian process adapted to a given
  filtration $\mathcal{F}_t$. We say that $X$ is strongly locally
  nondeterministic with parameter $\beta > 0$, $X$ is $\beta$\mbox{-}SLND for
  short, if there exists $\delta > 0$ s.t.
  \begin{equation}
    \tmop{Var} (X_t | \mathcal{F}_s \nobracket) \gtrsim | t - s |^{2 \beta}
    I_d \quad \text{uniformly in } s, t \text{ such that } 0 < t - s < \delta
    . \label{sec2.4 strong lnd eq}
  \end{equation}
\end{definition}

\begin{definition}
  \label{sec2.4 defn lnd}Let $(X_t)_{t \in [0, T]}$ be an
  $\mathbb{R}^d$-valued separable Gaussian process adapted to a given
  filtration $\mathcal{F}_t$. We say that $X$ is locally nondeterministic with
  parameter $\beta > 0$, $X$ is $\beta$\mbox{-}LND for short, if for every
  integer $n \geqslant 2$ there exists positive constants $c_n$ and $\delta_n$
  such that
  \begin{equation}
    \tmop{Var} \left( \sum_{k = 1}^n v_k \cdot (X_{t_{k + 1}} - X_{t_k})
    \right) \geqslant c_n \sum_{k = 1}^n \tmop{Var} (v_k \cdot (X_{t_{k + 1}}
    - X_{t_k})) \geqslant c_n \sum_{k = 1}^n | v_k |^2  | t_{k + 1} - t_k |^{2 \beta}
    \label{sec2.4 defn lnd eq}
  \end{equation}
  for all ordered points $t_1 < t_2 < \ldots < t_n$ with $t_n - t_1 <
  \delta_n$ and $v_k \in \mathbb{R}^m$.
\end{definition}

Properly speaking, in the terminology of~{\cite{xiao}}, Definition~\ref{sec2.4
defn strong lnd} is that of a one\mbox{-}sided strong local~nondeterminism,
but we have preferred to adopt the terminology $\beta$\mbox{-}SLND for
simplicity. It follows from Remark~2.3 from~{\cite{xiao}} that
Definition~\ref{sec2.4 defn strong lnd} is strictly stronger than
Definition~\ref{sec2.4 defn lnd}, namely any $\beta$\mbox{-}SLND Gaussian
process is also $\beta$\mbox{-}LND, while the converse is not true.
Equation~\eqref{sec2.4 lnd fbm} implies that $W^H$ is $H$\mbox{-}SLND.

\begin{definition}
  \label{sec2.4 defn exp lnd}Let $\{ X_t \}_{t \in [0, T]}$ be an
  $\mathbb{R}^d$-valued separable Gaussian process adapted to a given
  filtration $\mathcal{F}_t$. We say that $X$ is exponentially locally
  nondeterministic with parameter $\beta > 0$, $X$ is $\beta$\mbox{-}eSLND for
  short, if there exists $\delta > 0$ s.t.
  \begin{equation}
    \tmop{Var} (X_t | \mathcal{F}_s \nobracket) \gtrsim | \log (t - s) |^{-
    \beta} I_d \quad \text{uniformly in } s, t \text{ such that } 0 < t - s <
    \delta . \label{sec2.4 exp lnd eq}
  \end{equation}
\end{definition}

\begin{remark}
  \label{sec2.4 remark invariance lnd}Definitions~\ref{sec2.4 defn strong
  lnd},~\ref{sec2.4 defn lnd} and~\ref{sec2.4 defn exp lnd} only involve the
  variance (resp. conditional variance) of the process $X$ and are independent
  of its mean. This implies that they are all properties invariant under
  deterministic perturbations, namely if $X$ is a $\beta$\mbox{-}(e)(S)LND
  process and $f$ is a given measurable function, then $X + f$ is still
  $\beta$\mbox{-}(e)(S)LND. This can be interpreted as the chaoticity
  represented by local nondeterminism being too strong to be disrupted by a
  deterministic additive perturbations; this fundamental feature will allow us
  to prove prevalence of $\rho$\mbox{-}irregularity and exponential
  irregularity.
\end{remark}

We conclude this section by recalling the following result.

\begin{theorem}[Theorem 1.4 from {\cite{catelliergubinelli}}]
  \label{sec2.4 thm catelliergubinelli}Let $H \in (0, 1)$ and denote by
  $\mu^H$ the law of a $d$\mbox{-}dimensional fBm of Hurst parameter $H$. Then
  for any $\rho < (2 H)^{- 1}$ there exists $\gamma > 1 / 2$ such that
  \[ \mu^H \left( w \in C^0 ([0, T] ; \mathbb{R}^d)  \left| \, w \text{ is
     $(\gamma, \rho)$-irregular} \right. \right) = 1. \]
\end{theorem}

Let us point out that, combined with Lemma~\ref{sec2.2 lemma interpolation rho
irr}, this last theorem implies the existence of continuous $(\gamma,
\rho)$\mbox{-}irregular functions for any choice of $\gamma \in (0, 1)$ and
$\rho < \infty$.

\section{Main results}\label{sec3}

\subsection{Statements}\label{sec3.1}

We are now ready to present the main result of this paper. In the next
statement, the space $C^{\infty}_{\tmop{slow}}$ denotes the collection of all
infinitely differentiable functions $\psi : \mathbb{R}^d \rightarrow
\mathbb{R}$ such that $\psi$ and all its derivatives grow at most polynomially
at infinity; instead we set $C^{\infty}_b = \cap_n C^n_b$ as the set of
bounded smooth functions with bounded derivatives.

\begin{theorem}
  \label{sec3 thm2}It holds that:
  \begin{enumerate}[label=\roman*.]
    \item For any $\delta \in (0, \infty)$, almost every $\varphi \in
    C^{\delta}_t$ is $\rho$\mbox{-}irregular for any $\rho < (2 \delta)^{-
    1}$. If $\delta \geqslant 1$, then in addition for any $k < \delta$,
    $D^{(k)} \varphi$ is $\rho$\mbox{-}irregular for any $\rho < (2 (\delta -
    k))^{- 1}$.
    
    \item Almost every $\varphi \in C^0_t$ is $\rho$\mbox{-}irregular for any
    $\rho < \infty$; in particular, its occupation measures
    $(\mu^{\varphi}_{s, t}) \subset C^{\infty}_c$ and its averaging operator
    $T^w$ maps $\mathcal{S}'$ into $C^{\infty}_{\tmop{slow}}$, as well as
    $B^s_{p, q}$ into $C^{\infty}_b$ for any $s \in \mathbb{R}$, $p, q \in [1,
    \infty]$.
    
    \item For any $p \in [1, \infty)$, almost every $\varphi \in L^p_t$ is
    exponentially irregular; in particular it is Carath{\'e}odory and its
    occupation measures $(\mu^{\varphi}_{s, t})$ are analytic.
  \end{enumerate}
\end{theorem}

We avoid providing very similar statements, but like Point~\tmtextit{ii.}
above, Point~\tmtextit{i.} in combinations with Lemmata~\ref{sec2.2 lemma
regularity averaging 1},~\ref{sec2.2 lemma regularity averaging 2}
and~\ref{sec2.2 lemma averaging time} provides several other prevalence
statements in $C^{\delta}_t$ regarding the regularity of $\ell^w$ and the
regularising effect of $T^w$ acting on suitable function spaces.

\

The cornerstone in the proof of Theorem~\ref{sec3 thm2} is the following
probabilistic result.

\begin{theorem}
  \label{sec3 thm1}Let $X$ be a continuous $\beta$\mbox{-}(S)LND Gaussian
  process with $\beta \in (0, \infty)$; then for any $\rho < (2 \beta)^{- 1}$
  there exists $\gamma = \gamma (\rho, \beta) > 1 / 2$ such that $X$ is
  $(\gamma, \rho)$-irregular with probability~$1$. Let $X$ be a
  $\beta$\mbox{-}eSLND Gaussian with measurable, $L^2$\mbox{-}integrable
  trajectories and $\beta \in (0, 1]$; then $X$ is exponentially irregular
  with probability~$1$.
\end{theorem}

The proof of Theorem~\ref{sec3 thm1} of is given in Section~\ref{sec4}. For
simplicity, in the statement above no distinction between $\beta$\mbox{-}SLND
and $\beta$\mbox{-}LND processes appears; in Sections~\ref{sec4.2}
and~\ref{sec4.4} more refined results are given, so that the differences in
the two cases become more evident.

\

Theorem~\ref{sec3 thm2} combined with geometric considerations implies also
the following prevalence result.

\begin{theorem}
  \label{sec3 thm4}Let $\delta \in [0, 1)$. The following hold:
  \begin{enumerate}[label=\roman*.]
    \item Almost every $\varphi \in C^{\delta}_t$ has the property that
    \[ \dim_F (\varphi ([s, t])) = \dim_H (\varphi ([s, t])) =
       \frac{1}{\delta} \wedge d \quad \forall \, [s, t] \subset [0, T] . \]
    \item If additionally $\delta < 1 / (2 d)$, then almost every $\varphi \in
    C^{\delta}_t$ has the property that, for all $[s, t] \subset
    \mathbb{R}^d$, $\varphi ([s, t])$ contains an open set.
  \end{enumerate}
  In particular, for all $\delta \in [0, 1)$, the image of almost every
  function $\varphi \in C^{\delta}_t$ is a Salem set.
\end{theorem}

In the statement above, $\dim_F$ and $\dim_H$ denote respectively the Fourier
and Hausdorff dimensions. The proof is given in Section~\ref{sec5.1}, where
also the definitions of $\dim_F$, $\dim_H$ and Salem sets are recalled.

\begin{remark}
  Combining Theorem~\ref{sec3 thm1} with the reasoning from Theorem~\ref{sec3
  thm4} gives an alternative proof of the fact that Fourier and Hausdorff
  dimension of images of intervals under a $d$-dimensional fBm are
  $\mathbb{P}$-a.s. $H^{- 1} \wedge d$; see Theorem 4.29-b)
  from~{\cite{mortersperes}} (for $H = 1 / 2$) and Section~18
  from~{\cite{kahane}} (for general $H \in (0, 1)$) for more classical proofs.
\end{remark}

The study of analytic properties of $(\gamma, \rho)$\mbox{-}irregular paths
allows to show that, as the name suggests, they have a highly oscillatory
behaviour; this can be related to other notions of roughness already existing
in the literature.

\begin{theorem}
  \label{sec3 thm3} Let $w$ be $(\gamma, \rho)$\mbox{-}irregular,
  $\delta^{\ast}_{\gamma, \rho}$ defined as in~\eqref{sec2.2 critical
  parameter}; assume $\delta^{\ast}_{\gamma, \rho}\in (0,1)$.
  Then for any $\delta > \delta^{\ast}_{\gamma, \rho}$, $w$ is
  nowhere $\delta$\mbox{-}H\"{o}lder continuous and it has infinite modulus of
  $\delta$\mbox{-}H\"{o}lder roughness; \ for any $p < (\delta^{\ast}_{\gamma,
  \rho})^{- 1}$ and any interval $[s, t] \subset [0, T]$, $w$ has infinite
  $p$\mbox{-}variation on $[s, t]$.
\end{theorem}

The proof is given in Section~\ref{sec5.2}, where also the concept of modulus
of $\delta$\mbox{-}H\"{o}lder roughness is recalled.

A relation between $(\gamma,\rho)$-irregularity and failure of $\delta$-H\"older regularity was already observed for $d=1$ in \cite[Theorem 1.5]{choukgubinelli1}; our Theorem \ref{sec3 thm3} is significantly sharper, as it avoids the additional constraint $\delta>1-\gamma$ as therein, holds in any dimension $d$ and provides a nowhere $\delta$\mbox{-}H\"{o}lder continuity condition (in particular, regularity will fail to hold at \textit{all} points $t_0\in [0,T]$, not just \textit{some} point as in \cite[Theorem 1.5]{choukgubinelli1}). Furthermore, Theorem \ref{sec5.2 proposition irregularity}, on which the proof is based, gives a more precise bound on the Lebesgue density of the points around any $t_0$ where an approximate $\delta$-H\"older condition fails.

Quite nicely, Theorem~\ref{sec3 thm3} provides an alternative proof of Hunt's original
results from~{\cite{hunt}}.

\begin{corollary}
  \label{sec3 cor1}Let $\delta \in [0, 1)$, then almost every $\varphi \in
  C^{\delta}_t$ is nowhere $(\delta + \varepsilon)$\mbox{-}H\"{o}lder for any
  $\varepsilon > 0$.
\end{corollary}

\begin{proof}
  By Theorem~\ref{sec3 thm2}, almost every $\varphi \in C_t^{\delta}$ has the
  following property: for any $\rho < 1 / (2 \delta)$, there exists $\gamma >
  1 / 2$ such that $\varphi$ is $(\gamma, \rho)$\mbox{-}irregular. It holds
  \[ \delta^{\ast}_{\gamma, \rho} = \frac{1 - \gamma}{\rho} < \frac{1}{2 \rho}
  \]
  which implies by Theorem~\ref{sec3 thm3} that any such function is nowhere
  $\tilde{\delta}$\mbox{-}H\"{o}lder for any $\tilde{\delta} > 1 / (2 \rho)$.
  Taking $\rho = 1 / (2 \delta + 2 \varepsilon)$ the conclusion follows.
\end{proof}

In the current work we are also able to solve a problem left open
in~{\cite{catelliergubinelli}} and~{\cite{choukgubinelli1}}, which amounts to
showing that if $\varphi$ is $\rho$\mbox{-}irregular and $\psi \in
C^{\infty}_t$, then $\varphi + \psi$ is $\rho$\mbox{-}irregular. We give a
positive answer, up to strengthening the notion of $\rho$\mbox{-}irregularity.

\begin{theorem}
  \label{sec3 thm5}Let $\varphi$ be strongly $\rho$\mbox{-}irregular, then for
  any $\psi \in C^{\infty}_t$, $\varphi + \psi$ is
  $\tilde{\rho}$\mbox{-}irregular for any $\tilde{\rho} < \rho$. For $\delta
  \in [0, 1)$, almost every $\varphi \in C^{\delta}_t$ is strongly
  $\rho$\mbox{-}irregular for any $\rho < (2 \delta)^{- 1}$.
\end{theorem}

The proof is given in Section~\ref{sec5.3.1}, where the concept of strong
$\rho$\mbox{-}irregularity is introduced. It is a very technical concept,
which is the reason why we have preferred not to introduce it in
Section~\ref{sec2.2}. We hope in the future simpler criteria will be developed
for establishing the invariance of $\rho$\mbox{-}irregularity under additive
perturbations.

\begin{remark}
  \label{sec5.2 discussion optimality}Let us discuss already here the
  optimality of Theorem~\ref{sec3 thm2}, in light of Theorems~\ref{sec3
  thm1}\mbox{-}\ref{sec3 thm5} and the results presented in the rest of the
  paper.
  \begin{enumerate}[label=\arabic*.]
    \item For $\delta \in (0, 1)$, optimality of the condition $\rho < (2
    \delta)^{- 1}$ follows from Theorem~\ref{sec3 thm3}: it must hold $\delta
    \leqslant \delta^{\ast}_{\gamma, \rho} < (2 \rho)^{- 1}$ in light of the
    condition $\gamma > 1 / 2$.
    
    \item By applying Lemma~\ref{sec2.2 lemma interpolation rho irr}, in the
    case $\delta = 0$ the result can be strengthened to the fact that almost
    every $\varphi \in C^0_t$ is $(\gamma, \rho)$\mbox{-}irregular for any
    $\gamma < 1$ and any $\rho < \infty$. Time regularity cannot be improved
    to $\ell^{\varphi}$ being differentiable in time, since we know that (in
    the weak sense)
    \[ \frac{\mathd}{\mathd s} \ell (s, \cdot) |_{s = t} \nobracket =
       \delta_{\varphi_t} \quad \forall \, t \in [0, T] . \]
    Moreover $\{ \ell_{s, t}^{\varphi} \} \subset C^{\infty}_c$ cannot be
    improved to $\ell^{\varphi}_{s, t}$ being analytic, since this would imply
    that $\varphi ([s, t])$ is an unbounded set, thus violate continuity of
    $\varphi$.
    
    \item One might wonder if, since by Lemma~\tmcolor{blue}{\ref{sec2.2 lemma
    interpolation rho irr}} we can always raise the value of $\gamma$ by
    lowering the one of $\rho$, we can also do the opposite; in particular if,
    without imposing the restriction $\gamma > 1 / 2$, we can find functions
    $\varphi \in C^{\delta}$ which are $(\gamma, \rho)$-irregular for a pair
    $(\gamma, \rho)$ satisfying $\delta \leqslant \delta^{\ast}_{\gamma,
    \rho}$ but also $\rho > (2 \delta)^{- 1}$. In the case $\delta> 1/d$, $\delta\in (0,1]$ (thus implying $d\geqslant 2$)
    this possibility is ruled out by reasoning with Fourier dimensions (see
    Section~\ref{sec5.1} for more details), since it must hold
    \[ 2 \rho = \min (d, 2 \rho) \leqslant \dim_F (\varphi ([s, t])) \leqslant
       \dim_H (\varphi ([s, t])) \leqslant \delta^{- 1} \]
    independently of the value of $\gamma$.
    
    \item If $\delta \leqslant 1 / d$, the problem posed above is currently
    open. The only information we are able to provide in this case is that for
    $d = 1$, by Proposition~\ref{sec2.2 prop d=1} there exist indeed $C^1$
    functions (for instance $\varphi (t) = t$) which are $(\gamma, 1 -
    \gamma)$\mbox{-}irregular for any $\gamma \in (0, 1)$.
    
    \item Although Point~$4.$ is open in terms of generic $\varphi \in
    C^{\delta}$, we are at least able to establish that fBm paths do not have
    this property, see Lemma~\ref{lem:negative-answer-fbm} from
    Section~\ref{sec5.2}.
  \end{enumerate}
\end{remark}

We conclude this section with a proof of Theorem~\ref{sec3 thm2}, based on the
other results contained in the paper.

\begin{proof}[Proof of Theorem~\ref{sec3 thm2}]
  We have already seen in Lemma~\ref{sec2.2 lemma borel rho irr} that both
  $\rho$\mbox{-}irregularity and exponential irregularity properties define
  Borel measurable sets in all of the above function spaces, so we only need
  to provide suitable measures ``witnessing'' their prevalence.
  
  Let us start from the case $\delta \in (0, 1)$. Let $\mu^H$ to be the law on
  $C^{\delta}_t$ of a fBm $W^H$ of parameter $H = \delta + \varepsilon$ for
  some $\varepsilon > 0$, which is tight on $C^{\delta}_t$; let $\varphi \in
  C^{\delta}_t$ be fixed. The process $W^H$ is $H$\mbox{-}SLND and by
  Remark~\ref{sec2.4 remark invariance lnd} so is $\varphi + W^H$; then by
  Theorem~\ref{sec3 thm1} it holds
  \[ \mu^H \left( \varphi + w \text{ is $\rho$\mbox{-}irregular for any } \rho
     < 1 / (2 H) \right) =\mathbb{P} \left( \varphi + W^H \text{ is
     $\rho$\mbox{-}irregular for any } \rho < 1 / (2 H) \right) = 1. \]
  This implies that almost every $\varphi \in C^{\delta}_t$ is
  $\rho$\mbox{-}irregular for any $\rho < 1 / (2 \delta + 2 \varepsilon)$;
  taking a sequence $\varepsilon_n \downarrow 0$ and using the fact that
  countable intersection of prevalent sets is still prevalent, we obtain the
  conclusion in this case.
  
  Consider now the case $\delta \in [n, n + 1)$, $n \geqslant 1$; set $\delta
  = n + \theta$, $\theta \in [0, 1)$. Denote by $\mu^{n + H}$ the law of the
  process $Y$ obtained by integrating $n$ times an fBm, which is discussed at
  Point~\tmtextit{iv.} of the list of examples from Section~\ref{sec4.2}.
  Choose $H > \theta$, so that $\mu^{n + H}$ is tight in $C^{\delta}_t$. Now
  fix $\varphi \in C^{\delta}_t$; by the discussion at Point~\tmtextit{iv.},
  the process $Y$ is $(n + H)$\mbox{-}SLND with $D^{(k)} Y$ being $(n + H -
  k)$\mbox{-}SLND for any $k \in \{ 1, \ldots, n \}$ therefore $Y + \varphi$
  and $D^{(k)} (Y + \varphi) = D^{(k)} Y + D^{(k)} \varphi$ have the same
  properties by Remark~\ref{sec2.4 remark invariance lnd}. By applying
  Theorem~\ref{sec3 thm1} and arguing as in the previous point, taking a
  sequence $H_n \downarrow \delta$, the proof of claim~\tmtextit{i.} is
  complete.
  
  The first part of claim~\tmtextit{ii.} is identical, relying this time on
  the fact that we can take any $H > 0$ and we obtain
  $\rho$\mbox{-}irregularity for any $\rho < 1 / (2 H)$, together with the
  property that countable intersection of prevalent sets is prevalent; the
  fact that $(\mu^{\varphi}_{s, t}) \subset C^{\infty}_c$ readily follows from
  $\text{supp} \mu^{\varphi}_{s, t} \subset B_{\| \varphi \|_{C^0}}$ and the
  Fourier--Lebesgue embddings $\mathcal{F} L^{\rho, \infty} \hookrightarrow
  C^{\rho - d / 2 - \varepsilon} $. Concerning the second part of
  claim~\tmtextit{ii.}, observing that $\tilde{\mu}^{\varphi}_{s, t} \in
  C^{\infty}_c \hookrightarrow \mathcal{S}$, we can combine Lemma~\ref{sec2.3
  lemma averaging occupation} with Proposition~9.10~a) from~{\cite{folland}}
  to deduce that $T^{\varphi}$ maps $\mathcal{S}'$ into
  $C^{\infty}_{\tmop{slow}}$; the case of $B^s_{p, q}$ can be treated
  similarly, this time first embedding it into $B^{s - d / p}_{\infty, \infty}
  = C^{s - d / p}$ and then use standard Young-type estimates for
  convolutions.
  
  The second part of claim~\tmtextit{iii.} follows from Lemma~\ref{sec2.2 exp
  implies charatheodory}, once we have shown the first part. For any $\beta
  \in (0, 1)$, denote by $\mu^{\beta}$ the law of (a measurable version of)
  the process $X$ constructed in Proposition~\ref{sec4.2 prop wierd process},
  which is $\beta$\mbox{-}eSLND. Let $\varphi \in L^p_t$; we can require
  $\varphi$ to be an actual measurable function in its equivalence class,
  since the property of exponential irregularity does not depend on the chosen
  representative. By Remark~\ref{sec2.4 remark invariance lnd}, the process
  $\varphi + X$ is also $\beta$\mbox{-}eSLND and so by Theorem~\ref{sec3 thm1}
  \[ \mu^{\beta} \left( \varphi + w \text{ is exponentially irregular} \right)
     =\mathbb{P} \left( \varphi + X \text{ is exponentially irregular} \right)
     = 1 \]
  which implies the conclusion.
\end{proof}

\subsection{Applications to regularisation by noise for ODEs and
PDEs}\label{sec3.2}

In this section we show how our main result, Theorem~\ref{sec3 thm2}, can be
combined with already existing results in the literature, obtaining results on
the regularising effect of almost every $w \in C^{\delta}_t$ on ODEs and PDEs.

\

As mentioned before, $(\gamma, \rho)$\mbox{-}irregularity is closely related
to the regularising properties of the averaging operator $T^w$. Moreover, in
order to develop a good solution theory for the perturbed ODE
\begin{equation}
  \frac{\mathd}{\mathd t} x_t = b (x_t) + \frac{\mathd}{\mathd t} w_t
  \label{sec6 ode}
\end{equation}
it is enough to have good regularity properties for $T^w b$. Results in this
direction have been obtained in~{\cite{catelliergubinelli}}
and~{\cite{galeatigubinelli_ode}}. They can be combined with Lemma~\ref{sec2.2
lemma regularity averaging 1} and Theorem~\ref{sec3 thm2} as follows.

\begin{theorem}
  \label{sec6 thm ode}Suppose that either $T^w b \in C^{\gamma}_t C^2_x$ or $b
  \in L^{\infty}_{t, x}$ and $T^w b \in C^{\gamma}_t C^{3 / 2}_x$, for some
  $\gamma > 1 / 2$. Then for any $x_0 \in \mathbb{R}^d$ there exists a unique
  global solution to~\eqref{sec6 ode}; moreover the ODE admits a locally
  Lipschitz flow. If, for some $n\geqslant 1$, $T^w b \in C^{\gamma}_t C^{n + 1}_x$, or $b \in
  L^{\infty}_{t, x}$ and $T^w b \in C^{\gamma}_t C^{n + 1 / 2}_x$, then the
  flow is locally $C^n$.
  
  Let $\delta \in [0, 1)$, then almost every $w \in C^{\delta}_t$ is such that
  for any $\alpha \in \mathbb{R}$, the following hold:
  \begin{enumerate}[label=\roman*.]
    \item if $\alpha > \max \{ 3 / 2 - (2 \delta)^{- 1}, 0 \}$ or $\alpha > 2
    - (2 \delta)^{- 1}$, then for any $b \in \mathcal{F} L^{\alpha, 1}$ the
    perturbed ODE~\eqref{sec6 ode} is well posed and admits a locally
    Lipschitz flow;
    
    \item if $\alpha > \max \{ n + 1 / 2 - (2 \delta)^{- 1}, 0 \}$ or $\alpha
    > n + 1 - (2 \delta)^{- 1}$, for some $n\geqslant 1$, then for any $b \in \mathcal{F} L^{\alpha,
    1}$ the flow is locally $C^n$.
  \end{enumerate}
  In particular for almost every $\varphi \in C^0_t$, for any $\alpha \in
  \mathbb{R}$ and any $b \in \mathcal{F} L^{\alpha, 1}$, the ODE~\eqref{sec6
  ode} is well\mbox{-}posed and admits a $C^{\infty}$ flow.
\end{theorem}

\begin{proof}
  The first part of the statement comes from Remark~4.7 and Theorem~4.33
  from~{\cite{galeatigubinelli_ode}}.
  
  Point~\tmtextit{i.} then follows from a combination of Theorem~\ref{sec3
  thm2}\mbox{-}\tmtextit{i.}, Lemma~\ref{sec2.2 lemma regularity averaging 1}
  and the Fourier--Lebesgue embeddings $\mathcal{F} L^{\beta, 1}
  \hookrightarrow C^{\beta}_b$ for $\beta \geqslant 0$. For instance in the
  case $\alpha > 2 - (2 \delta)^{- 1}$, a.e. $w \in C^{\delta}_t$ is $(\gamma,
  \rho)$-irregular for $\rho = 2 - \alpha < (2 \delta)^{- 1}$ and some $\gamma
  > 1 / 2$, implying that for any $b \in \mathcal{F} L^{\alpha, 1}$ it holds
  $T^w b \in C^{\gamma}_t \mathcal{F} L^{\alpha + \rho,1} \hookrightarrow
  C^{\gamma}_t C^2_x$, so that we can apply the first part of the theorem. The
  case $\alpha > \max \{ 3 / 2 - (2 \delta)^{- 1}, 0 \}$ is similar, only in
  this case by embeddings we have $b \in C^0_b$ to begin with, and our
  assumptions allow to verify that $T^w b \in C^{\gamma}_t C^{3 / 2}_x$.
  Point~\tmtextit{ii.} can be established in a similar manner and the final
  claim comes from Theorem~\ref{sec3 thm2}\mbox{-}\tmtextit{ii.}
\end{proof}

The last part of the statement is an instance of the
$\infty$\mbox{-}regularising effect of generic continuous functions
on~\eqref{sec6 ode}. Applying Lemma~\ref{sec2.2 lemma regularity averaging 2},
analogue statements can be given replacing Fourier--Lebesgue $\mathcal{F}
L^{\alpha, 1}$ with other scales like $B^s_{p, q}$ spaces. In addition, we can
also consider time\mbox{-}dependent fields $b$, for instance such that $b \in
C^{1 / 2}_t \mathcal{F} L^{\alpha, 1}$, thanks to Lemma~\ref{sec2.2 lemma
averaging time}.

\

The theory developed for solving~\eqref{sec6 ode} can be also successfully
applied to the study first order linear PDEs of the form
\begin{equation}
  \partial_t u + b \cdot \nabla u + c u + \frac{\mathd w}{\mathd t} \cdot
  \nabla u = 0. \label{sec6 transport1}
\end{equation}
In general, if $w$ is $(\gamma, \rho)$\mbox{-}irregular then it is nowhere
differentiable, but assume for the moment $w \in C^1$. Setting $v (t, x) = u
(t, x + w_t)$, $\tilde{b} (t, x) = b (t, x + w_t)$, $\tilde{c} (t, x) = c (t,
x + w_t)$, $u$ would be a classical solution of~\eqref{sec6 transport1} if and
only if $v$ solved
\begin{equation}
  \partial_t v + \tilde{b} \cdot \nabla v + \tilde{c} v = 0, \label{sec6
  transport2}
\end{equation}
which is meaningful (under suitable assumptions on $\tilde{b}$, $\tilde{c}$)
even when $w \in C^0$. For this reason we will say that $u$ is a solution
to~\eqref{sec6 transport1} if and only if $v$ solves~\eqref{sec6 transport2},
without properly defining the meaning of equation~\eqref{sec6 transport1}. Let
us mention that in the context of rough path theory, if $w$ admits a rough
lift, then it is possible to give meaning to~\eqref{sec6 transport1} and show
the equivalence with~\eqref{sec6 transport2}, see~{\cite{catellier}}. If $b$
and $c$ are distributional but $T^w b$ and $T^w c$ have good regularity in
$C^{\gamma}_t C^{\beta}_x$\mbox{-}spaces, then it is possible to give meaning
to~\eqref{sec6 transport2} with Young integrals; in this case
equation~\eqref{sec6 transport2} (and consequently~\eqref{sec6 transport1})
admit a natural interpretation in the sense of distributions, thus we will
talk of weak solutions, cf. Definitions~5.5 and~5.13
from~{\cite{galeatigubinelli_ode}}; we refrain for giving here more details,
for which we refer the interested reader to Section~5
from~{\cite{galeatigubinelli_ode}}. In the next statement, $\mathcal{M}_x =
\mathcal{M} (\mathbb{R}^d)$ denotes the set of all finite Radon measures.

\begin{theorem}
  \label{sec6 thm transport}Let $\delta \in [0, 1)$ then almost every $w \in
  C^{\delta}_t$ is such that for any $\alpha \in \mathbb{R}$, the following
  hold:
  \begin{enumerate}[label=\roman*.]
    \item if $\alpha > \max \{ 3 / 2 - (2 \delta)^{- 1}, 0 \}$ or $\alpha > 2
    - (2 \delta)^{- 1}$, then for any $b \in \mathcal{F} L^{\alpha, 1}$ the
    transport PDE
    \[ \partial_t u + b \cdot \nabla u + \frac{\mathd w}{\mathd t} \cdot
       \nabla u = 0 \]
    has a unique solution $u \in C^0_t C^1_{\tmop{loc}}$ for any $u_0 \in
    C^2_x$;
    
    \item if for some $n\geqslant 1$, $\alpha > \max \{ n + 1 / 2 - (2 \delta)^{- 1}, 0 \}$ or $\alpha
    > n + 1 - (2 \delta)^{- 1}$, then for any $b \in \mathcal{F} L^{\alpha,
    1}$ and any $u_0 \in C^{n + 1}_x$, the solution $u \in C^0_t
    C^n_{\tmop{loc}}$;
    
    \item if $\alpha > \max \{ 3 / 2 - (2 \delta)^{- 1}, 0 \}$ or $\alpha > 2
    - (2 \delta)^{- 1}$, then for any $b \in \mathcal{F} L^{\alpha, 1}$ the
    continuity equation
    \[ \partial_t u + \nabla \cdot (b u) + \frac{\mathd w}{\mathd t} \cdot
       \nabla u = 0 \]
    has a unique weak solution $u \in C^{\gamma}_t (C^1_x)^{\ast} \cap
    L^{\infty}_t \mathcal{M}_x$ for any $u_0 \in \mathcal{M}_x$.
  \end{enumerate}
\end{theorem}

\begin{proof}
  It follows from the results of~{\cite{galeatigubinelli_ode}} that, if $b \in
  C^0_x$ and $T^w b \in C^{\gamma}_t C^{3 / 2}_x$ (cf. Proposition~5.1
  therein) or $T^w b \in C^{\gamma}_t C^2_x$ (cf. Lemma~5.7 and Theorem~5.12)
  for some $\gamma > 1 / 2$, then existence and uniqueness of solutions to the
  transport PDE holds for any $u_0 \in C^2_x$ and we have the representation
  formula $u_t (x) = u_0 (\phi^{- 1}_t (x))$, where $\phi$ is the flow
  associated to the ODE~\eqref{sec6 ode}. Arguing as in the proof of
  Theorem~\ref{sec6 thm ode} it is easy to check that, under our assumptions
  on $\delta$ and $\alpha$, a.e. $w \in C^{\delta}_t$ is
  $\rho$\mbox{-}irregular for the right range of parameters $\rho$ needed to
  establish the required regularity of $T^w b$, proving Point~\tmtextit{i.}
  
  Point~\tmtextit{ii.} follows from the representation formula for $u$ and the
  higher regularity of the flow $\phi$ associated to the ODE, which comes
  from~Theorem~\ref{sec6 thm ode}\mbox{-}\tmtextit{ii.}
  
  Point~\tmtextit{iii.} comes from similar reasonings, only this time applying
  the results from~{\cite{galeatigubinelli_ode}} for the continuity equation,
  which reduce again the problem to verify the regularity of $T^w b$; in
  particular one can invoke Proposition~5.3 therein in the case $b \in C^0_x$
  and Lemma~5.15, Theorem~5.16 in the case of general $b$.
\end{proof}

In the above examples, $w$ enters the equation as a perturbation that can be
reabsorbed by shifting the phase space, which is why the operator $T^w$
appears. In the next examples instead $w$ has the role of
\tmtextit{modulating} a given group of transformations.

\

In the papers~{\cite{choukgubinelli1}} and~{\cite{choukgubinelli2}}, the
authors study the regularising properties of $(\gamma, \rho)$\mbox{-}irregular
paths on non\mbox{-}linear dispersive PDEs of the general form
\begin{equation}
  \frac{\mathd}{\mathd t} \varphi_t = A \varphi_t  \frac{\mathd w_t}{\mathd t}
  + \mathcal{N} (\varphi_t), \label{sec6 choukgubinelli eq1}
\end{equation}
where $w \in C^0_t$, $\varphi : D \rightarrow \mathbb{R}$ (or $\mathbb{C}$),
$A$ is a linear unbounded operator and $\mathcal{N}$ is a nonlinearity
(typically of polynomial type). Their cases of interest are:
\begin{enumerate}[label=\arabic*.]
  \item (cubic NLS) Non\mbox{-}linear cubic Schr\"{o}dinger, $D =\mathbb{T}^d$
  or $\mathbb{R}^d$, $d = 1, 2$, $A = i \Delta$, $\mathcal{N} (\varphi) = i |
  \varphi |^2 \varphi$;
  
  \item General NLS on $D =\mathbb{R}$ with $A = i \partial_x^2$, $\mathcal{}
  \mathcal{N} (\varphi) = i | \varphi |^{\mu} \varphi$, $\mu \in (1, 4]$;
  
  \item (dNLS) Non\mbox{-}linear derivative cubic Schr\"{o}dinger on
  $\mathbb{T}$, $A = i \partial_x^2$, $\mathcal{N} (\varphi) = i \partial (|
  \varphi |^2 - \| \varphi \|_{L^2}^2) \varphi$;
  
  \item (KdV) Korteweg--de~Vries, $D =\mathbb{T}$ or $\mathbb{R}$, $A =
  \partial_x^3$, $\mathcal{N} (\varphi) = \partial_x (\varphi^2)$;
  
  \item (mKdV) Modified Korteweg--de~Vries, $D =\mathbb{T}$, $A =
  \partial_x^3$, $\mathcal{N} (\varphi) = \partial_x (\varphi^3 - 3 \varphi \|
  \varphi \|_{L^2}^2)$.
\end{enumerate}
From now on we will refer to these PDEs as {\tmem{$w$-modulated equations}},
e.g. $w$-modulated cubic NLS, $w$-modulated KdV, etc. In all cases
1.\mbox{-}5. above, although the original system would be of integrable
nature, the presence of $w$ doesn't allow to exploit any integrability
features; moreover the group $\{ e^{t A} \}_{t \in \mathbb{R}}$ associated to
$A$ acts isometrically on all $H^{\alpha}$\mbox{-}spaces, thus doesn't provide
a priori any regularisation.

In order to give meaning to~\eqref{sec6 choukgubinelli eq1}, the authors adopt
the mild formulation (which would be justified for $w \in C^1$ by the chain
rule, but afterwards makes sense for any $w \in C^0$)
\[ \varphi_t = U^w_t \varphi_0 + U^w_t \int_0^t (U^w_s)^{- 1} \mathcal{N}
   (\varphi_s) \mathd s \]
where $U^w_t = e^{w_t A}$; applying the change of variables $\psi_t =
(U^w_t)^{- 1} \varphi_t$, the equation becomes
\begin{equation}
  \psi_t = \varphi_0 + \int_0^t (U^w_s)^{- 1} \mathcal{N} (U^w_s \psi_s)
  \mathd s \label{sec6 choukgubinelli eq2} .
\end{equation}
It is finally possible to give meaning to~\eqref{sec6 choukgubinelli eq2},
even in spaces $H^{\alpha}$ for which the original nonlinearity $\mathcal{N}$
is not well defined, using the $\rho$\mbox{-}irregularity property to show
that the operators $\{ X_{s, t} \}_{s < t}$
\[ X_{s, t} (\phi) = \int_s^t (U^w_r)^{- 1} \mathcal{N} (U^w_r \phi) \mathd s,
   \quad s < t \]
are continuous from $H^{\alpha}$ to itself (actually $C^{\infty}$ since they
are the monoid associated to an $n$\mbox{-}linear bounded operator). It is
then imposed that $\varphi$ is a solution to~\eqref{sec6 choukgubinelli eq1}
if and only if the associated $\psi$ solves~\eqref{sec6 choukgubinelli eq2}.
We refrain from giving further details on the topic and only point out that
our Theorem~\ref{sec3 thm2} combined with their results give the following
statements.

\begin{theorem}
  \label{sec6 thm nls}Let $\delta \in [0, 1)$. Then for almost every $w \in
  C^{\delta} ([0, T])$, the $w$\mbox{-}modulated cubic NLS on $\mathbb{T}$ and
  $\mathbb{R}$ has a global solution in $H^{\alpha}$ for any $\alpha \geqslant
  0$ and any $\varphi_0 \in H^{\alpha}$, which is unique in a suitable class;
  moreover the equation admits a locally Lipschitz continuous flow.
\end{theorem}

\begin{theorem}
  \label{sec6 thm kdv}Let $\delta \in [0, 2 / 3)$. Then:
  \begin{enumerate}[label=\roman*.]
    \item For almost every $w \in C^{\delta} ([0, T])$, the
    $w$\mbox{-}modulated KdV on $\mathbb{T}$ has a unique local solution in
    $H^{\alpha}$ for any $\varphi_0 \in H^{\alpha}$ with $\alpha > - (2
    \delta)^{- 1}$, which is global if $\alpha > - \min \{ 3 / 2, (4
    \delta)^{- 1} \}$.
    
    \item For almost every $w \in C^{\delta} ([0, T])$, the
    $w$\mbox{-}modulated KdV on $\mathbb{R}$ has a unique local solution in
    $H^{\alpha}$ for any $\varphi_0 \in H^{\alpha}$ with $\alpha > - \min \{ 3
    / 4, (2 \delta)^{- 1} \}$, which is global if $\alpha > - \min \{ 3 / 4, (4
    \delta)^{- 1} \}$.
  \end{enumerate}
  Moreover for any $\delta \in [0, 1)$, for almost every $w \in C^{\delta}_t$,
  the $w$\mbox{-}modulated mKdV on $\mathbb{T}$ has a unique local solution in
  $H^{\alpha}$ for any $\varphi_0 \in H^{\alpha}$ with $\alpha \geqslant 1 /
  2$.
\end{theorem}

\begin{proof}[Proof of Theorems~\ref{sec6 thm nls}\mbox{-}\ref{sec6 thm kdv}]
  Theorem~\ref{sec6 thm nls} follows from a combination of Theorem~\ref{sec3
  thm2}-i) and Theorem~1.8 from~{\cite{choukgubinelli1}}; indeed the condition
  $\delta < 1$ always ensures that a.e. $w \in C^{\delta}$ is $\rho$-irregular
  for some $\rho > 1 / 2$. Theorem~\ref{sec6 thm kdv} is instead implied by
  Theorem~\ref{sec3 thm2}-i) and Theorems~1.6 from~{\cite{choukgubinelli2}}.
  Indeed, Theorem ~\ref{sec3 thm2}-i) guarantees that a.e. $w\in C^\delta$ is $(\gamma,\rho)$-irregular, for $\gamma$ arbitrarily close to $1/2$ and $\rho$ arbitrarily close to $(2\delta)^{-1}$; the conclusion then follows by plugging the above consideration in the conditions for local and global wellposedness of solutions in $H^\alpha$ provided by Theorems~1.6 from~{\cite{choukgubinelli2}}, which are expressed in terms of $(\gamma,\rho)$.
  For instance, concerning Point \textit{i}., local wellposedness holds for $\rho\sim (2\delta)^{-1}>3/4$ (thus $\delta<2/3$) and $\alpha>-\rho\sim -(2\delta)^{-1}$, while global existence holds for $\alpha> -\min\{3/2,\rho/(3-2\gamma)\}\sim -\min\{3/2,(4\delta)^{-1}\}$. Point \textit{ii}. is similar.
%  indeed taking $\rho = (2 \delta)^{- 1} - \varepsilon$ for $\varepsilon > 0$
%  small enough, $\gamma > 1 / 2$, conditions $\rho > 3 / 4$ and $\alpha > -
%  \rho / (3 - 2 \gamma)$ from Theorems~1.6 from~{\cite{choukgubinelli2}} are
%  implied respectively by $\delta < 2 / 3$ and $\alpha > - (4 \delta)^{- 1} .$  
\end{proof}

Analogue statements can be obtained combining Theorem~\ref{sec3 thm2} with
other results from the aforementioned papers, for instance Theorems~1.9 and
1.10 from~{\cite{choukgubinelli1}} or Theorem~1.7
from~{\cite{choukgubinelli2}}.

\

In the setting of standard dispersive equations, a key role in establishing
uniqueness of solutions is played by Strichartz estimates. In the
paper~{\cite{duboscq}}, for a given path $w \in C^0 ([0, T])$, the authors
study under which conditions the operator $A$ given by
\[ f \mapsto (A f)_t \assign \int_0^t | w_t - w_s |^{- \alpha} f_s \mathd s \]
is bounded from $L^p (0, T)$ to $L^q (0, T)$ for suitable values of $(p, q)$;
the idea is to apply this kind of modulated Hardy--Littlewood--Sobolev
inequality to obtain Strichartz estimates for the modulated semigroup
\[ P_{s, t} \psi (x) = e^{i \Delta (w_t - w_s)} \psi (x) = \frac{1}{(4 \pi
   (w_t - w_s))^{d / 2}} \int_{\mathbb{R}^d} \exp \left( i \frac{| x - y
   |^2}{4 (w_t - w_s)} \right) \psi (y) \mathd y. \]
They only consider $w$ sampled as a stochastic process, specifically a fBm of
parameter $H \in (0, 1)$, but a closer inspection of the proof shows that
their result (Theorem~1.1 from~{\cite{duboscq}}) can be restated in an
analytic fashion as follows.

\begin{theorem}
  \label{sec6 thm hardy}Suppose that $w \in C^0 ([0, T])$ admits an occupation
  density $\ell^w \in C^{\beta}_t C^0_x$ for some $\beta \in (0, 1)$. Then for
  any $p, q \in (1, \infty)$ and $\alpha \in (0, 1)$ satisfying
  \[ 2 - \alpha = \frac{1}{p} + \frac{1}{q} \]
  there exists a constant $C > 0$ such that for all $f \in L^p (0, T)$ and $g
  \in L^q (0, T)$ it holds
  \[ \left| \int_0^T \int_0^T f_t  | w_t - w_s |^{- \alpha} g_s \mathd s
     \mathd t \right| \leqslant C T^{\alpha \beta}  \| f \|_{L^p} \| g
     \|_{L^q} . \]
  For any $\delta \in [0, 1 / 2)$, almost every $\varphi \in C^{\delta} ([0,
  T])$ satisfies the above assumption for any $\beta < 1 - \delta$.
\end{theorem}

\begin{proof}
  The proof of Theorem~1.1 in~{\cite{duboscq}} is entirely analytical, as it
  follows closely the proof of the standard Hardy--Littlewood--Sobolev
  inequality from Lieb--Loss~{\cite{lieb}}, but it requires a key property
  satisfied by fBm paths, given in Lemma~2.1 therein: if we define
  \[ W (r, T) \assign \sup_{t \in [0, T]} \int_0^T \mathbbm{1}_{| w_t - w_s |
     < r} \mathd s, \]
  then there must exist a constant $c$ such that
  \begin{equation}
    W (r, T) \leqslant 2 r c T^{\beta} \quad \text{for all } r > 0.
    \label{sec6 hardy eq1}
  \end{equation}
  It is not difficult to see that requirement~\eqref{sec6 hardy eq1} is
  equivalent to the request that $\ell^w_T (x) \leqslant c T^{\beta}$ for all
  $x \in \mathbb{R}$; indeed, assume first that $\ell^w_T (x) \leqslant c
  T^{\beta}$ holds, then
  \[ W (r, T) = \sup_{t \in [0, T]} \int_0^T \mathbbm{1}_{w_s \in B (w_t, r)}
     \mathd s = \sup_{t \in [0, T]} \int_{B (w_t, r)} \ell^w_T (x) \mathd x
     \lesssim 2 r c T^{\beta} . \]
  On the other side, if $w$ admits a continuous density $\ell^w$
  and~\eqref{sec6 hardy eq1} holds, then
  \[ \ell^w_T (w_t) = \lim_{r \rightarrow 0} \frac{1}{2 r} \int_{B (w_t, r)}
     \ell^w_T (x) \mathd x \leqslant \lim_{r \rightarrow 0} \frac{1}{2 r} W
     (r, T) \leqslant c T^{\beta}, \]
  and since we know that $\ell^w_T$ is supported on $w ([0, T])$, the above
  estimate extends to all $x \in \mathbb{R}^d$.
  
  It is now clear that requirement~\eqref{sec6 hardy eq1} can be expressed in
  entirely analytical terms and so does the proof of Theorem~1.1
  from~{\cite{duboscq}}, and the authors are only using the additional fact
  that almost every fBm trajectory satisfies~\eqref{sec6 hardy eq1}. Another
  analytical condition ensuring that $\ell^w \in C^{\beta}_t C^0_x$ is the
  $(\gamma, \rho)$\mbox{-}irregularity property, since by Lemma~\ref{sec2.2
  lemma regularity averaging 2}, if $\rho > 1$, then $\ell^w \in C^{\gamma}_t
  C^0_x$. Moreover we can apply Lemma~\ref{sec2.2 lemma interpolation rho irr}
  to deduce that for any $\theta \in (0, 1)$ such that $\theta \rho > 1$,
  $\ell^w \in C^{1 - \theta + \theta \gamma}_t C^0_x$, which implies that if
  $w$ is $\rho$\mbox{-}irregular for any $\rho < \bar{\rho}$, $\bar{\rho} >
  1$, then $\ell^w \in C^{\beta}_t C^0_x$ for any $\beta < 1 - (2
  \bar{\rho})^{- 1}$. The conclusion follows applying the fact that for
  $\delta < 1 / 2$, almost every $\varphi \in C^{\delta}_t$ is
  $\rho$\mbox{-}irregular for any $\rho < (2 \delta)^{- 1}$.
\end{proof}

Similarly, the proofs of Strichartz estimates and well-posedness for
$w$\mbox{-}modulated NLS (Proposition~1.1 and Theorem~1.2 respectively)
from~{\cite{duboscq}} are entirely deterministic and only rely on the validity
of the above modulated Hardy--Littlewood--Sobolev inequality; they can
therefore be fully generalised to prevalence results, similarly to
Theorems~\ref{sec6 thm nls} and~\ref{sec6 thm kdv} above.

In~{\cite{choukgess}}, the authors provide regularity estimates for solutions
to scalar conservation laws modulated by a path $w$ of the form
\begin{equation}
  \partial_t u + \sum_{i = 1}^d \partial_{x_i} A^i (u) \circ \frac{\mathd
  w^i_t}{\mathd t} = 0 \quad \text{ on } \mathbb{T}^d, \quad u (0) = u_0 \in
  L^{\infty} (\mathbb{T}^d) . \label{sec6 gess eq1}
\end{equation}

They use the concept of $(\gamma, \rho)$\mbox{-}irregularity to show
regularisation by noise phenomena whenever $w$ is sampled as an fBm, but their
results are of analytic (or path\mbox{-}by\mbox{-}path) type; before stating
their result, let us point out a simplification: given a $(\gamma,
\rho)$\mbox{-}irregular $w \in C^{\delta}_t$, the authors impose on a suitable
parameter $\lambda$, depending on another parameter $\nu \geqslant 1$, the
condition
\[ \lambda < \frac{\rho (\delta + 1) - (1 - \gamma)}{(\nu \rho \vee 1) (\delta
   + 1) + (1 - \gamma)} \wedge \frac{\rho + 2 (\nu \rho \vee 1)}{(\nu \rho
   \vee 1) (2 \delta + 1) + (1 - \gamma)} =: c_1 \wedge c_2 . \]
Thanks to Theorem~\ref{sec3 thm3}, we can actually simplify the above
expression; we claim that $c_1 \leqslant c_2$. Since it must always hold
$\delta \leqslant \delta^{\ast}_{\gamma, \rho}$, we have
\[ \rho (\delta + 1) - (1 - \gamma) \leqslant \rho + \rho
   \delta^{\ast}_{\gamma, \rho} - (1 - \gamma) \leqslant \rho ; \]
to check that $c_1 \leqslant c_2$, it then suffices to verify that
\[ \frac{\rho}{(\nu \rho \vee 1) (\delta + 1) + 1 - \gamma} \leqslant
   \frac{\rho + 2 (\nu \rho \vee 1)}{(\nu \rho \vee 1) (2 \delta + 1) + (1 -
   \gamma)} \]
and after a few algebraic manipulations we see that this is equivalent to
\[ \rho \delta \leqslant 2 [(\nu \rho \vee 1) (\delta + 1) + 1 - \gamma] ; \]
this is now clearly always true since $\nu \geqslant 1$, so that $\rho \delta
\leqslant (\nu \rho \vee 1) \delta \leqslant 2 (\nu \rho \vee 1) (\delta +
1)$.

The main result of~{\cite{choukgess}} can then be restated as follows:

\begin{theorem}[Theorem~2.4 from~{\cite{choukgess}}]
  \label{sec6 thm chouk gess} Let $w \in C^{\delta}_t$ be $(\gamma,
  \rho)$\mbox{-}irregular and let $u$ be a quasi\mbox{-}solution
  to~\eqref{sec6 gess eq1}. Assume $A = (A^1, \ldots, A^d) \in C^2
  (\mathbb{R}; \mathbb{R}^d)$ satisfies the following non\mbox{-}degeneracy
  condition: there exist $\nu \geqslant 1$ and $c > 0$ such that, for $A' = a
  = (a^1, \ldots, a^d) \in C^1 (\mathbb{R}; \mathbb{R}^d)$, it holds
  \[ \inf_{v \in \mathbb{R}^d} \max_{i = 1, \ldots, d}  | v_i (a^i (x) - a^i
     (y)) | \geqslant c | x - y |^{\nu} \quad \text{for all } x, y \in
     \mathbb{R}. \]
  Then there exists a constant $C = C (\| \Phi^w \|_{\mathcal{W}^{\gamma,
  \rho}})$ such that for all $T > 0$ and all
  \[ \lambda < \frac{\rho (\delta + 1) - (1 - \gamma)}{(\nu \rho \vee 1)
     (\delta + 1) + (1 - \gamma)} \]
  it holds
  \begin{equation}
    \int_0^T \| u_t \|_{W^{\lambda, 1}} \mathd t \leqslant C (\| u_0
    \|_{L^1_x} + \| u \|_{L^1_{t, x}} + \| w \|_{C^{\delta}_t}  \| a' (v) m
    \|_{\tmop{TV}}) . \label{sec6 gess eq2}
  \end{equation}
  If $u$ is an entropy solution then in addition
  \[ \| u_t \|_{W^{\lambda, 1}} < \infty \quad \text{for all } t > 0. \]
\end{theorem}

We avoid here providing all the details on the result above for which we refer
the reader to~{\cite{choukgess}}; let us only mention that the definition of
quasi\mbox{-}solution to~\eqref{sec6 gess eq1} requires the existence of a
finite Radon measure $m$ associated to $u$, which is the one appearing in
estimate~\eqref{sec6 gess eq2}; $u$ is an entropy solution if $m$ is
non\mbox{-}negative.

\

A few algebraic manipulations together with Theorem~\ref{sec3 thm2} imply the
following result.

\begin{corollary}
  Let $\delta \in (0, 1)$. For almost every $\varphi \in C^{\delta}_t$, the
  statement of Theorem~\ref{sec6 thm chouk gess} holds for any
  \[ \lambda < \frac{1}{(\nu \vee 2 \delta) (\delta + 1) + \delta} . \]
\end{corollary}

\section{Criteria for $\rho$-irregularity of stochastic processes}\label{sec4}

This section is devoted to the study of probabilistic properties sufficient to
ensure that a stochastic process has $(\gamma, \rho)$\mbox{-}irregular sample
paths. It includes the proof of Theorem~\ref{sec3 thm1}, which is the
cornerstone for our main result prevalence Theorem~\ref{sec3 thm2}, but
develops several criteria which are of independent interest. In particular we
establish $\rho$\mbox{-}irregularity for processes like fBm,
$\alpha$\mbox{-}stable process, Ornstein--Uhlenbeck as well as $X_t = \int_0^t
B_s \mathd s$; many of these process have appeared in regularisation by noise
phenomena, see for
instance~{\cite{le}},~{\cite{athreya}},~{\cite{priola2020}},~{\cite{chaudru}}
among others.

\subsection{General criteria}\label{sec4.1}

We provide here useful general criteria to establish
$\rho$\mbox{-}irregularity for a given stochastic process, which will then be
applied to several examples in the next section.

\

We adopt the following convention: although we always write statements to
hold for any $\xi \in \mathbb{R}^d$, they must be interpreted as ``for all
$\xi$ big enough'', i.e. $| \xi | \geqslant C$ for some universal
deterministic constant $C > 0$, so that for instance expressions like $\log |
\xi |$ are meaningful. We have seen that in the case of
$\rho$\mbox{-}irregularity this is not an issue, since the only relevant
information given by $\| \Phi^w \|_{\mathcal{W}^{\gamma, \rho}}$ is for big
values of $| \xi |$. Similarly, for a modulus of continuity $\varphi$ defined
only on a neighbourhood of $0$, $t$ and $s$ are tacitly assumed to be
sufficiently close whenever $\varphi (| t - s |)$ appears.

\

The next statement given in a general form, but keep in mind that our primary
focus is the case $F (\xi) = | \xi |^{\alpha}$ for suitable values of
$\alpha$.

\begin{theorem}
  \label{sec4.1 thm general criterion}Let $(X_t)_{t \in [0, T]}$ be an
  $\mathbb{R}^d$\mbox{-}valued stochastic process with $\mathbb{P}$-a.s.
  measurable trajectories, $T < \infty$. Let $F : \mathbb{R}^d \rightarrow
  \mathbb{R}_{\geqslant 0}$ be a continuous function with the following
  properties:
  \begin{enumerate}[label=\roman*.]
    \item there exist constants $c_1, c_2 > 0$ such that $F (x) \leqslant c_1
    F (y)$ whenever $| x - y | \leqslant c_2$;
    
    \item $F$ has at most polynomial growth, i.e. $F (\xi) \leqslant c_3 | \xi
    |^{\alpha}$ as $\xi \rightarrow \infty$ for some $\alpha < \infty$, $c_3 >
    0$.
  \end{enumerate}
  Also assume that there exist positive constants $\mu, \theta, \delta$ such
  that the following hold:
  \begin{enumerate}
    \item (Integrability condition)
    \begin{equation}
      \mathbb{E} \left[ \exp \left( \mu \int_0^T | X_t |^{\theta} \mathd t
      \right) \right] < \infty \label{sec4.1 general criterion integrability}
      ;
    \end{equation}
    \item (Continuity condition)
    \begin{equation}
      \sup_{\substack{        0 < t - s < \delta\\
        \xi \in \mathbb{R}^d }}
        \mathbb{E} \left[ \exp \left( \mu \frac{| \Phi^X_{s, t}
      (\xi) |^2  | F (\xi) |^2}{| t - s |} \right) \right] < \infty
      \label{sec4.1 general criterion continuity} .
    \end{equation}
  \end{enumerate}
  Then, for the choice $\varphi (x) = \sqrt{x | \log x |}$, defining the
  random variable
  \[ Y \assign \sup_{s \neq t, \xi \in \mathbb{R}^d} \frac{| \Phi^X_{s, t}
     (\xi) | F (\xi)}{\varphi (| t - s |) \sqrt{\log | \xi |}}, \]
  there exists $\lambda > 0$ such that $\mathbb{E} [\exp (\lambda Y^2)] <
  \infty$.
\end{theorem}

\begin{proof}
  First of all let us show that, starting from~\eqref{sec4.1 general criterion
  continuity}, we can find another constant $\tilde{\mu}$ such that the same
  bound holds taking the supremum over all $s < t$, without the restriction $|
  t - s | < \delta$. Let $[s, t]$ be such that $| t - s | > \delta$; we can
  split the interval $[s, t]$ in at most $n = \lfloor T / \delta \rfloor + 1$
  intervals of the form $[t_i, t_{i + 1}]$, of size at most $\delta$; we have
  the estimate
  \begin{align*}
    \mathbb{E} \left[ \exp \left( \tilde{\mu}  \frac{| \Phi^X_{s, t} (\xi) |^2
    | F (\xi) |^2}{| t - s |} \right) \right] & \leqslant \mathbb{E} \left[
    \exp \left( \frac{n \tilde{\mu}}{\delta}  \sum_i | \Phi^X_{t_i, t_{i + 1}}
    (\xi) |^2 | F (\xi) |^2 \right) \right]\\
    & =\mathbb{E} \left[ \prod_i \exp \left( \frac{n \tilde{\mu}}{\delta} |
    \Phi^X_{t_i, t_{i + 1}} (\xi) |^2 | F (\xi) |^2 \right) \right]\\
    & \leqslant \prod_i \mathbb{E} \left[ \exp \left( n^2 \tilde{\mu} 
    \frac{| \Phi^X_{t_i, t_{i + 1}} (\xi) |^2 | F (\xi) |^2}{| t_{i + 1} - t_i
    |} \right) \right]^{1 / n}
  \end{align*}
  and choosing $\tilde{\mu}$ such that $ (\lfloor T / \delta \rfloor + 1)^2
  \tilde{\mu} \leqslant \mu$ we obtain
  \[ \sup_{\tmscript{\begin{array}{c}
       | t - s | > \delta\\
       \xi \in \mathbb{R}^d
     \end{array}}} \mathbb{E} \left[ \exp \left( \tilde{\mu}  \frac{|
     \Phi^X_{s, t} (\xi) |^2  | F (\xi) |^2}{| t - s |} \right) \right] <
     \infty . \]
  From now on with a slight abuse we will denote by $\mu$ the new constant
  under which we have a bound of the form~\eqref{sec4.1 general criterion
  continuity} over all possible $t \neq s$.
  
  Let us define, for any $s \neq t$ and suitable $\lambda > 0$, the following
  quantity:
  \begin{equation*}
    J_{s, t} (\lambda) \assign \sum_{n \in \mathbb{N}} 2^{- n} \sum_{\xi \in
    2^{- n} \mathbb{Z}^d} 2^{- n (d + 1)} (1 + | \xi |)^{- (d + 1)} \exp
    \left( \lambda \frac{| \Phi^X_{s, t} (\xi) |^2 | F (\xi) |^2}{| t - s |}
    \right) .
  \end{equation*}
  It follows from~\eqref{sec4.1 general criterion continuity} that, for all
  $\lambda \leqslant \mu$, $\mathbb{E} [J_{s, t} (\lambda)] \leqslant K$
  uniformly in $s, t$; moreover by Jensen inequality it is clear that for any
  $\beta > 1$ it holds $J_{s, t} (\lambda)^{\beta} \lesssim J_{s, t} (\beta
  \lambda)$. Let us also define
  \[ Y_{s, t} \assign \frac{1}{| t - s |^{1 / 2}} \sup_{\xi \in \mathbb{R}^d}
     \frac{| \Phi^X_{s, t} (\xi) | F (\xi)}{\sqrt{\log | \xi |}} . \]
  In order to conclude, it suffices to show that there exists $\lambda \in (0,
  \mu)$ such that $\mathbb{E} [\exp (\lambda Y^2_{s, t})] \leqslant K$
  uniformly in $s < t$; indeed, we can then apply Lemma~\ref{appendixA2
  chaining lemma} from Appendix~\ref{appendixA2} for the choice $\alpha = 1 /
  2$, $E = C^0 (\mathbb{R}^d)$ and $Z : [0, T] \rightarrow E$ given by $Z_t
  (\xi) \assign \Phi^X_t (\xi) F (\xi) / \sqrt{\log | \xi |}$, to get the
  associated bound for $Y$.
  
  Fix $\xi \in \mathbb{R}^d$. For any $n \in \mathbb{N}$, we can find
  $\tilde{\xi} \in 2^{- n} \mathbb{Z}^d$ such that $| \xi - \tilde{\xi} |
  \lesssim 2^{- n}$; for such $\tilde{\xi}$ it holds
  \[ \frac{| \Phi_{s, t}^X (\tilde{\xi}) | F (\tilde{\xi})}{| t - s |^{1 / 2}}
     \lesssim \lambda^{- 1 / 2} \sqrt{\log \, J_{s, t} (\lambda) + n + \log |
     \tilde{\xi} |} . \]

  We can assume $\theta\in (0,1]$; indeed, if \eqref{sec4.1 general criterion integrability} is satisfied for some $\theta>1$, then by H\"older's inequality it also holds for $\theta=1$.
  Combining the trivial estimates $| e^{i \xi \cdot x} -
  e^{i \tilde{\xi} \cdot x} | \leqslant 2$, $| e^{i \xi \cdot x} - e^{i
  \tilde{\xi} \cdot x} | \leqslant | \xi - \tilde{\xi} | | x |$, we then find
  \begin{align}
    | \Phi_{s, t}^X (\tilde{\xi}) - \Phi_{s, t}^X (\xi) | & \leqslant \int_s^t
    | e^{i \xi \cdot X_r} - e^{i \tilde{\xi} \cdot X_r} | \, \mathd r \lesssim
    | \xi - \tilde{\xi} |^{\theta} \int_s^t | X_r |^{\theta} \, \mathd r
    \leqslant | \xi - \tilde{\xi} |^{\theta} \| X \|_{L^{\theta}}, \nonumber
  \end{align}
  which interpolated together with $| \Phi^X_{s, t} (\xi) | \leqslant | t - s
  |$ gives
  \[ | \Phi_{s, t}^X (\tilde{\xi}) - \Phi_{s, t}^X (\xi) | \lesssim | t - s
     |^{1 / 2} | \xi - \tilde{\xi} |^{\theta / 2} \| X \|^{\theta /
     2}_{L^{\theta}} . \]
  Gathering everything together and using the fact that for $n$ big enough it
  holds $F (\xi) \lesssim F (\tilde{\xi})$, we obtain
  \begin{eqnarray*}
    | \Phi^X_{s, t} (\xi) | F (\xi) & \lesssim & | \Phi^X_{s, t} (\xi) -
    \Phi_{s, t}^X (\tilde{\xi}) | F (\xi) + | \Phi_{s, t}^X (\tilde{\xi}) | F
    (\tilde{\xi})\\
    & \lesssim & | t - s |^{1 / 2}  | \xi - \tilde{\xi} |^{\theta / 2}  \| X
    \|^{\theta / 2}_{L^{\theta}} F (\xi)\\
    &  & + | t - s |^{1 / 2} \lambda^{- 1 / 2}  \sqrt{\log \, J_{s, t}
    (\lambda) + \, n + \log | \tilde{\xi} |}\\
    & \lesssim & | t - s |^{1 / 2}  \, \| X \|^{\theta / 2}_{L^{\theta}} 2^{-
    n \theta / 2} F (\xi)\\
    &  & + | t - s |^{1 / 2} \lambda^{- 1 / 2}  \sqrt{\log \, J_{s, t}
    (\lambda) + n + \log | \xi | + c}
  \end{eqnarray*}
  Therefore, choosing $n \sim \log | \xi |$, which is by assumption enough for
  $F (\xi)  \, 2^{- n \theta / 2} \lesssim 1$, we obtain
  \[ | \Phi^X_{s, t} (\xi) | F (\xi) \lesssim | t - s |^{1 / 2} \left[ \| X
     \|^{\theta / 2}_{L^{\theta}} + \lambda^{- 1 / 2} \sqrt{\log \, J_{s, t}
     (\lambda) + \log | \xi |} \right] . \]
  Dividing by $\sqrt{\log | \xi |}  | t - s |^{1 / 2}$ and taking the supremum
  we get
  \[ Y_{s, t} \lesssim \| X \|^{\theta / 2}_{L^{\theta}} + \lambda^{- 1 / 2} +
     \lambda^{- 1 / 2} \sqrt{\log \, J_{s, t} (\lambda)}, \]
  and so there exists a constant $C$ such that
  \begin{eqnarray*}
    \exp (\lambda Y_{s, t}^2) & \lesssim & \exp (\lambda C \| X
    \|_{L^{\theta}}^{\theta})   \, J_{s, t} (\lambda)^C\\
    & \lesssim & \exp (2 \lambda C \| X \|_{L^{\theta}}^{\theta}) + J_{s, t}
    (\lambda)^{2 C}\\
    & \lesssim & \exp (2 \lambda C \| X \|_{L^{\theta}}^{\theta}) + J_{s, t}
    (2 \lambda C) .
  \end{eqnarray*}
  Choosing $\lambda$ such that $2 \lambda C \leqslant \mu$ we therefore obtain
  a uniform bound for $\mathbb{E} [\exp (\lambda Y_{s, t}^2)]$ which by the
  above reasoning implies the conclusion.
\end{proof}

\begin{remark}
  \label{sec4.1 remark general criterion}Going through the same proof, one can
  obtain a similar statement for $F$ such that:
  \begin{enumerate}[label=\roman*.]
    \item there exist constants $c_1, c_2, c_3 > 0$ such that $F (x) \leqslant
    c_1 F (c_2 y)$ whenever $| x - y | \leqslant c_3$;
    
    \item $F$ has exponential\mbox{-}type growth, i.e. $\log F (\xi) \leqslant
    c_4 | \xi |^{\alpha}$ as $\xi \rightarrow \infty$ for some $\alpha <
    \infty$, $c_4 > 0$.
  \end{enumerate}
  Then under conditions~\eqref{sec4.1 general criterion integrability}
  and~\eqref{sec4.1 general criterion continuity} it is possible to find $c >
  0$ such that defining
  \begin{equation}
    Y \assign \sup_{s \neq t, \xi \in \mathbb{R}^d} \frac{| \Phi^X_{s, t}
    (\xi) | F (c \xi)}{\varphi (| t - s |)  | \xi |^{\alpha}} \label{sec4.1 eq
    remark}
  \end{equation}
  the same conclusion as in Theorem~\ref{sec4.1 thm general criterion} holds.
  The choice $F (\xi) = \exp (\lambda | \xi |^{\alpha})$ satisfies the above
  requirements and in this case we can get rid of $| \xi |^{\alpha}$ in the
  denominator of~\eqref{sec4.1 eq remark} by changing $c$.
\end{remark}

We immediately deduce the following result.

\begin{corollary}
  \label{sec4.1 corollary exp integr rho}Let $X$ be a process satisfying the
  assumptions of Theorem~\ref{sec4.1 thm general criterion} for $F (\xi) = |
  \xi |^{\alpha}$. Then for any $\rho < \alpha$, there exists $\gamma = \gamma
  (\rho) > 1 / 2$ such that $X$ is $(\gamma, \rho)$-irregular and moreover
  \begin{equation}
    \mathbb{E} [\exp (\lambda \| \Phi^X \|_{\mathcal{W}^{\gamma, \rho}}^2) ] <
    \infty \quad \forall \, \lambda \in \mathbb{R}. \label{sec4.1 exp integr
    rho irr}
  \end{equation}
\end{corollary}

\begin{proof}
  Let $Y$ be defined as in Theorem~\ref{sec4.1 thm general criterion};
  combining the trivial estimate $| \Phi^X_{s, t} (\xi) | \leqslant | t - s |$
  with
  \[ | \Phi^X_{s, t} (\xi) | \leqslant Y | \xi |^{- \alpha} \log^{1 / 2} |
     \xi |  | t - s |^{1 / 2}  | \log | t - s | |^{1 / 2} \]
  by interpolation we obtain, for any fixed $\varepsilon > 0$,
  \begin{align}
    | \Phi^X_{s, t} (\xi) | & \leqslant \, Y^{1 - 2 \varepsilon}  | \xi |^{-
    \alpha (1 - 2 \varepsilon)} \log^{1 / 2 - \varepsilon} | \xi |  | t - s
    |^{1 / 2 + \varepsilon}  | \log | t - s | |^{1 / 2} \nonumber\\
    & \lesssim_{\varepsilon} \, Y^{1 - 2 \varepsilon} | \xi |^{- \alpha (1 -
    3 \varepsilon)}  | t - s |^{1 / 2 + \varepsilon / 2} \nonumber
  \end{align}
  so that setting $\rho = \alpha (1 - 3 \varepsilon) < \alpha$, $\gamma = 1 /
  2 + \varepsilon / 2 > 1 / 2$, we obtain
  \[ \| \Phi^X \|_{\mathcal{W}^{\gamma, \rho}_T} \lesssim Y^{1 - 2
     \varepsilon} \]
  which also implies that, for a suitable $C = C (\varepsilon)$, taking $\mu >
  0$ small enough, it holds
  \begin{equation}
    \mathbb{E} [\exp (\mu \| \Phi^X \|_{\mathcal{W}^{\gamma, \rho}}^{2 / (1 -
    2 \varepsilon)})] \leqslant \mathbb{E} [\exp (\mu C Y^2)] < \infty .
    \label{sec4.1 estimate norm rho irr}
  \end{equation}
  Since $\varepsilon > 0$, $2 / (1 - 2 \varepsilon) > 2$ and therefore
  from~\eqref{sec4.1 estimate norm rho irr} we immediately
  deduce~\eqref{sec4.1 exp integr rho irr}. The reasoning holds for any
  $\varepsilon > 0$, so we can invert the relations between $\rho$,
  $\varepsilon$ and $\gamma$ to deduce that for any given $\rho < \alpha$ we
  can take $\gamma (\rho) = 1 / 2 + (1 - \rho / \alpha) / 6$.
\end{proof}

Theorem~\ref{sec4.1 thm general criterion} and Corollary~\ref{sec4.1 corollary
exp integr rho} are well suited for establishing $\rho$\mbox{-}irregularity in
several situations, as we will show in the next section. However,
conditions~\eqref{sec4.1 general criterion integrability} and~\eqref{sec4.1
general criterion continuity} are in general difficult to check, due to their
exponential nature. We present here a weaker version of Theorem~\ref{sec4.1
thm general criterion} which relaxes condition~\eqref{sec4.1 general criterion
integrability}; in Section~\ref{sec4.3} we will show how it is possible to
relax condition~\eqref{sec4.1 general criterion continuity} instead.

\begin{corollary}
  \label{sec4.1 cor weaker rho irr}Let $(X_t)_{t \in [0, T]}$ be an
  $\mathbb{R}^d$-valued stochastic process with $\mathbb{P}$-a.s. measurable
  trajectories, $T < \infty$; assume that it satisfies the continuity
  condition~\eqref{sec4.1 general criterion continuity} from
  Theorem~\ref{sec4.1 thm general criterion} for $F (\xi) = | \xi |^{\alpha}$
  and that there exists $\theta > 0$ such that
  \[ \mathbb{P} \left( \int_0^T | X_t |^{\theta} \mathd t < \infty \right) =
     1. \]
  Then for any $\rho < \alpha$ there exists $\gamma = \gamma (\rho) > 1 / 2$
  such that $\mathbb{P}$-a.s. $X_{\cdot}$ is $(\gamma,
  \rho)$\mbox{-}irregular.
\end{corollary}

\begin{proof}
  Let $f \in C^0_t$ be a deterministic continuous function which is
  $\alpha$-irregular, whose existence is granted by Theorem~\ref{sec2.4 thm
  catelliergubinelli}.
  
  For any $N \in \mathbb{N}$, set $A = \left\{ \omega \in \Omega : \, \|
  X_{\cdot} (\omega) \|_{L^{\theta} ([0, T])} \leqslant N \right\}$ and define
  $X^N_{\cdot} \assign \mathbbm{1}_A X_{\cdot} + \mathbbm{1}_{A^c} f_{\cdot}$.
  Then it is easy to check that by construction
  \[ \mathbb{E} [\exp (\lambda \| X^N_{\cdot} \|_{L^{\theta}})] < \infty \quad
     \forall \, \lambda \in \mathbb{R}. \]
  Moreover, letting $\mu, \delta$ denote the constant under which $X$
  satisfies condition~\eqref{sec4.1 general criterion continuity}, for any $|
  s - t | < \delta$ we have
  \begin{align*}
    \mathbb{E} \left[ \exp \left( \mu \frac{| \Phi^{X^N}_{s, t} (\xi) |^2 |
    \xi |^{2 \alpha}}{| t - s |} \right) \right] & \leqslant \, \mathbb{E}
    \left[ \exp \left( \mu \frac{| \Phi^X_{s, t} (\xi) |^2 | \xi |^{2
    \alpha}}{| t - s |} \right) \right] + \exp \left( \mu \frac{| \Phi^f_{s,
    t} (\xi) |^2 | \xi |^{2 \alpha}}{| t - s |} \right)\\
    & \leqslant \, {\sup_{\substack{ 0 < s - t < \delta\\
      \xi \in \mathbb{R}^d}}}  \, \mathbb{E} \left[ \exp \left( \mu \frac{| \Phi^X_{s, t}
    (\xi) |^2 | \xi |^{2 \alpha}}{| t - s |} \right) \right] + \exp (\mu \|
    \Phi^f \|_{\mathcal{W}^{1 / 2, \alpha}}^2) < \infty
  \end{align*}
  which implies that $X^N$ also satisfies condition~\eqref{sec4.1 general
  criterion continuity}. Therefore, by Corollary~\ref{sec4.1 corollary exp
  integr rho}, for fixed $N$ and $\rho < \alpha$, $X^N$ is
  $\mathbb{P}$\mbox{-}a.s. $\rho$\mbox{-}irregular; but then
  \[ \mathbb{P} \left( X \text{ is } \rho \text{-irregular} \right) \geqslant
     \mathbb{P} \left( X = X^N, \, X^N \text{ is } \rho \text{-irregular}
     \right) = 1 -\mathbb{P} (\| X \|_{L^{\theta}} > N) \rightarrow 1 \quad
     \text{as } N \rightarrow \infty . \]
\end{proof}

The following lemma provides easy\mbox{-}to\mbox{-}check sufficient conditions
for~\eqref{sec4.1 general criterion continuity} to hold. It is a more general
version of the technique developed in~{\cite{catelliergubinelli}},
Theorem~4.3.

\begin{lemma}
  \label{sec4.1 lemma condition rho irr}Let $(X_t)_{t \in [0, T]}$ be an
  $\mathbb{R}^d$-valued stochastic process with measurable trajectories and $f
  \in C^0 ([0, T] \times \mathbb{R}^d ; \mathbb{C})$ such that $\| f \|_{C^0}
  \leqslant 1$. Assume that there exists a deterministic constant $C_f$ such
  that for any $s < t$ it holds
  \begin{equation}
    \mathbb{P} \left( \left| \mathbb{E} \left[ \int_s^t f (u, X_u) \mathd u |
    \mathcal{F}_s \right] \right| \leqslant C_f \right) = 1.
  \end{equation}
  Then there exists a constant $\mu > 0$ independent of $f, s, t$ such that
  \begin{equation}
    \sup_{s \neq t} \mathbb{E} \left[ \exp \left( \mu \frac{\left| \int_s^t f
    (u, X_u) \mathd u \right|^2}{C_f  | t - s |} \right) \right] < \infty
    \label{sec4.1 exp estimate lemma} .
  \end{equation}
\end{lemma}

\begin{proof}
  We can assume without loss of generality $f$ to be real valued, since otherwise we can just apply
  separately estimates for $\Re f$ and $\Im f$, which still satisfy the above
  assumptions. Suppose first that $| t - s | \leqslant C_f$. Then we have the
  trivial estimate
  \[ \left| \int_s^t f (u, X_u) \mathd u \right|^2 \leqslant | t - s |^2
     \leqslant | t - s | C_f, \]
  which immediately implies the estimate~\eqref{sec4.1 exp estimate lemma} in
  this case.
  
  Suppose now that $| t - s | / C_f  \geqslant 1$. Let $n \in \mathbb{N}$ to
  be fixed later; define $t_k = s + k (t - s) / n$ for $k \in \{ 0, \ldots, n
  \}$ and
  \[ Z_k =\mathbb{E} \left[ \int_s^t f (u, X_u) \mathd u | \mathcal{F}_{t_{k
     + 1}} \right] -\mathbb{E} \left[ \int_s^t f (u, X_u) \mathd u | 
     \mathcal{F}_{t_k} \right] . \]
  Setting $S_n = \sum_{k = 0}^{n - 1} Z_k$, $S_n$ is a martingale and it holds
  \[ \int_s^t f (u, X_u) \mathd u = S_n +\mathbb{E} \left[ \int_s^t f (u, X_u)
     \mathd u | \mathcal{F}_s \right] . \]
  In order to get a bound for such quantity, we want to apply
  Azuma\mbox{-}Hoeffding inequality to the martingale increments $Z_k$. We
  start by estimating the second term:
  \[ \left| \mathbb{E} \left[ \int_s^t f (u, X_u) \mathd u | \mathcal{F}_s
     \right] \right| \leqslant C_f \leqslant C_f^{1 / 2} | t - s |^{1 / 2}, \]
  so that uniformly in $s < t$ we have
  \begin{equation}
    \exp \left( \mu \left| \mathbb{E} \left[ \int_s^t f (u, X_u) \mathd u |
    \mathcal{F}_s \right]^2 \right| / (C_f | t - s |) \right) < \infty .
    \label{sec4.1 estimate inside lemma}
  \end{equation}
  For the martingale increments instead we have
  \[ Z_k =\mathbb{E} \left[ \int_{t_{k + 1}}^t f (u, X_u) \mathd u |
     \mathcal{F}_{t_{k + 1}} \right] -\mathbb{E} \left[ \int_{t_k}^t f (u,
     X_u) \mathd u | \mathcal{F}_{t_k} \right] + \int_{t_k}^{t_{k + 1}} f (u,
     X_u) \mathd u \]
  and so by the hypothesis and $\| f \|_{\infty} \leqslant 1$ we have the
  $\mathbb{P}$-a.s. estimate
  \[ | Z_k | \leqslant 2 C_f + \int_{t_k}^{t_{k + 1}} | f (u, X_u) | \mathd u
     \leqslant 2 C_f + | t - s | / n. \]
  By Azuma\mbox{-}Hoeffding inequality, there exist universal constants $\mu,
  K > 0$ such that
  \[ \mathbb{E} [\exp (\mu | S_n |^2 / C_n)] < K \]
  where the constant $C_n$ is given by
  \[ C_n = \sum_{k = 0}^{n - 1} (2 C_f + | t - s | / n)^2 \sim 4 n \, C_f^2 +
     | t - s |^2 / n. \]
  Choosing $n$ such that $n \sim | t - s | / C_f$ we obtain $C_n \sim | t - s
  | C_f$ and so for $\mu$ sufficiently small it holds
  \[ \mathbb{E} [\exp (\mu | S_n |^2 / C_f | t - s |)] < K. \]
  where the estimate is uniform over $s, t$. Together with
  estimate~\eqref{sec4.1 estimate inside lemma}, this implies the conclusion.
\end{proof}

\begin{remark}
  Lemma~\ref{sec4.1 lemma condition rho irr} is of independent interest, even
  outside the context of $\rho$\mbox{-}irregularity. For instance, let $X$ be
  a Markov process with semigroup $P_t = e^{t A}$ and let $- \lambda$ be an
  eigenvalue of $A$ with associated eigenvector $f^{\lambda}$; more precisely,
  assume $\lambda > 0$ and let $f^{\lambda} \in L^{\infty}$ (normalised so
  that $\| f^{\lambda} \|_{L^{\infty}} = 1$) satisfy $A f^{\lambda} = -
  \lambda f^{\lambda}$. Then it holds
  \[ \left| \mathbb{E} \left[ \int_s^t f^{\lambda} (X_u) \mathd u |
     \mathcal{F}_s \right] \right| = \left| \int_s^t (P_{u - s} f^{\lambda})
     (X_s) \mathd u \right| = \left| \int_s^t e^{- \lambda (u - s)}
     f^{\lambda} (X_s) \mathd u \right| \leqslant \frac{1}{\lambda} ; \]
  this implies for instance for $s = 0$, $t = 1$ that there exists a universal
  $\mu > 0$ such that
  \[ \sup_{\lambda > 0} \mathbb{E} \left[ \exp \left( \mu \lambda \left|
     \int_0^1 f^{\lambda} (X_u) \mathd u \right|^2 \right) \right] < \infty .
  \]
\end{remark}

\subsection{$\rho$\mbox{-}irregularity for $\beta$\mbox{-}(e)SLND Gaussian
processes and examples}\label{sec4.2}

We apply here the results of the previous section to prove part of
Theorem~\ref{sec3 thm1}. More quantitative results are given in the next two
statements.

\begin{theorem}
  \label{sec4.2 corollary rho gaussian}Let $\{ X_t \}_{t\in [0,T]}$ be a
  $\mathbb{R}^d$-valued separable Gaussian process with measurable paths s.t.
  \begin{equation}
    \int_0^T \mathbb{E} [| X_t |^2] \mathd t < \infty ; \label{sec4.2 gaussian
    integrability condition}
  \end{equation}
  suppose also that $X$ is $\beta$\mbox{-}SLND for some $\beta\in (0,\infty)$. Then $X$ satisfies the
  assumptions of Theorem~\ref{sec4.1 thm general criterion} for $F (\xi) = |
  \xi |^{1 / 2 \beta}$; therefore for any $\rho < (2 \beta)^{- 1}$ there
  exists $\gamma = \gamma (\rho) > 1 / 2$ such that $X$ is $(\gamma,
  \rho)$-irregular
  \[ \mathbb{E} [\exp (\lambda \| \Phi^X \|_{\mathcal{W}^{\gamma, \rho}}^2)] <
     \infty \quad \forall \, \lambda \in \mathbb{R}. \]
\end{theorem}

\begin{proof}
  It follows from~\eqref{sec4.2 gaussian integrability condition} that
  $X_{\cdot}$ is an $L^2 ([0, T] ; \mathbb{R}^d)$-valued Gaussian process and
  therefore by Fernique Theorem we can conclude that there exists $\mu > 0$
  such that
  \[ \mathbb{E} [\exp (\mu \| X \|_{L^2}^2)] < \infty, \]
  which implies that condition~\eqref{sec4.1 general criterion integrability}
  is satisfied. It remains to check condition~\eqref{sec4.1 general criterion
  continuity}, which we plan to do with the help of Lemma~\ref{sec4.1 lemma
  condition rho irr}. Taking $f (t, x) = e^{i \xi \cdot x}$ for a fixed $\xi
  \in \mathbb{R}^d$, we need to estimate the associate constant $C_f =
  C_{\xi}$; it is enough to provide estimates whenever $s$ and $t$ both lie in
  an interval of size at most $\delta$, $\delta$ being the parameter for which
  the $\beta$\mbox{-}SLND condition is satisfied. It holds
  
  \begin{align}
    \left| \mathbb{E} \left[ \int_s^t e^{i \xi \cdot X_u} \mathd u |
    \mathcal{F}_s \nobracket \right] \right| & \leqslant \int_s^t | \mathbb{E}
    [e^{i \xi \cdot X_u} | \mathcal{F}_s \nobracket] | \mathd u \nonumber\\
    & = \int_s^t \exp \left( - \frac{1}{2} \xi \cdot (\tmop{Var} (X_u |
    \mathcal{F}_s \nobracket) \xi) \right) \mathd u \nonumber\\
    & \leqslant \int_s^t \exp (- c | \xi |^2 | u - s |^{2 \beta}) \mathd u
    \nonumber\\
    & \leqslant \int_0^{\infty} \exp (- c | \xi |^2 r^{2 \beta}) \mathd r
    \sim | \xi |^{- 1 / \beta}, \nonumber
  \end{align}
  where in the third passage we used the $\beta$-SLND property. Therefore
  $C_{\xi} \leqslant c | \xi |^{- 1 / \beta}$ and thus there exists $\mu > 0$
  such that for any $s, t \in [0, T]$ satisfying $0 < s - t < \delta$ it holds
  \begin{align}
    \mathbb{E} \left[ \exp \left( \mu \frac{| \xi |^{1 / \beta} | \Phi_{s,
    t}^X (\xi) |^2}{| t - s |} \right) \right] \leqslant & \mathbb{E} \left[
    \exp \left( c \mu \frac{| \Phi_{s, t}^X (\xi) |^2}{C_{\xi} | t - s |}
    \right) \right] \leqslant K \nonumber
  \end{align}
  uniformly in $\xi \in \mathbb{R}^d$. By Theorem~\ref{sec4.1 thm general
  criterion} we obtain the conclusion for the choice $F (\xi) = | \xi |^{1 / 2
  \beta}$. The last statement follows from Corollary~\ref{sec4.1 corollary exp
  integr rho}.
\end{proof}

\begin{proposition}
  \label{sec4.2 prop exp irr}Let $\{X_t\}_{t\in [0,T]}$ be a $\mathbb{R}^d$-valued
  separable Gaussian process with measurable paths satisfying
  condition~\eqref{sec4.2 gaussian integrability condition} which is
  $\beta$\mbox{-}eSLND for some $\beta\in (0,\infty)$. Then there exist constants $c, \lambda > 0$ such that
  defining
  \[ Y \assign \sup_{s \neq t, \xi \in \mathbb{R}^d} \frac{| \Phi^X_{s, t}
     (\xi) | \exp (c | \xi |^{2 / (1 + \beta)})}{\varphi (| t - s |)} \]
  with $\varphi (x) = \sqrt{x | \log x |}$, it holds
  \[ \mathbb{E} [\exp (\lambda Y^2)] < \infty . \]
  In particular, if $\beta \leqslant 1$, then $X$ is exponentially irregular.
\end{proposition}

\begin{proof}
  As in the proof of Theorem~\ref{sec4.2 corollary rho gaussian},
  condition~\eqref{sec4.2 gaussian integrability condition} implies
  condition~\eqref{sec4.1 general criterion integrability} by Fernique
  Theorem. The rest of the proof is similar, relying on Lemma~\ref{sec4.1
  lemma condition rho irr} applied to $f (t, x) = e^{i \xi \cdot x}$, only
  this time we want to apply Remark~\ref{sec4.1 remark general criterion} for
  the choice $F (\xi) = \exp (| \xi |^{2 / (1 + \beta)})$, which requires a faster
  decay for $C_f = C_{\xi}$.
  
  For any $s < t$ such that $| t - s | < \delta$, where $\delta$ is the
  parameter in the $\beta$\mbox{-}eSLND condition, it holds
  \begin{align*}
    \left| \mathbb{E} \left[ \int_s^t e^{i \xi \cdot X_u} \mathd u |
    \mathcal{F}_s \nobracket \right] \right| & = \int_s^t \exp \left( -
    \frac{1}{2} \tmop{Var} (\xi \cdot X_u | \mathcal{F}_s \nobracket) \right)
    \mathd u\\
    & \leqslant \int_0^1 \exp (- c | \xi |^2 | \log r |^{- \beta}) \mathd r\\
    & \sim \int_0^{+ \infty} \exp \left( - c | \xi |^2 x - x^{-
    \frac{1}{\beta}} \right) x^{- \frac{\beta + 1}{\beta}} \mathd x\\
    & \lesssim \int_0^{\infty} \exp \left( - \tilde{c} \left( | \xi |^2 x +
    x^{- \frac{1}{\beta}} \right) \right) \mathd x
  \end{align*}
  where in the third line we used the change of variables $x = | \log r |^{-
  \beta}$. By the general inequality $a + b \gtrsim a^{\theta} b^{1 -
  \theta}$, valid for all $a, b > 0$ and $\theta \in (0, 1)$, it holds
  \[ | \xi |^2 x + x^{- \frac{1}{\beta}} \gtrsim | \xi |^{2 (1 - \theta)} x^{1
     - \theta \frac{\beta + 1}{\beta}} = | \xi |^{2 / (\beta + 1)} \]
  for the choice $\theta = \beta / (\beta + 1)$; therefore there exists a
  constant $c_1$ such that
  \[ \left| \mathbb{E} \left[ \int_s^t e^{i \xi \cdot X_u} \mathd u |
     \mathcal{F}_s \nobracket \right] \right| \lesssim \exp (- c_1 | \xi |^{2
     / (\beta + 1)}) \]
  and by Lemma~\ref{sec4.1 lemma condition rho irr} we deduce that there
  exists $c_2 > 0$ such that
  \[ \sup_{s \neq t, \xi \in \mathbb{R}^d} \mathbb{E} \left[ \exp \left( c_2 
     \frac{e^{c_1 | \xi |^{2 / (\beta + 1)}} | \Phi_{s, t}^X (\xi) |^2}{| t -
     s |} \right) \right] < \infty . \]
  The conclusion then follows from Remark~\ref{sec4.1 remark general
  criterion}. In particular if $\beta \in (0, 1]$, then $2 / (\beta + 1)
  \geqslant 1$, which implies that $X$ is exponentially irregular.
\end{proof}

The rest of this section is devoted to providing examples of Gaussian
processes with the above properties; in the following $(\mathcal{F}_t)_t$
denotes either the filtration generated by $B$ or the filtration generated by
a given process $X$ which will be clear from the context. Here are some
examples of $\beta$\mbox{-}SLND Gaussian processes:
\begin{enumerate}[label=\roman*.]
    \item Let $W^H$ be a $d$-dimensional fBm of Hurst parameter $H$. We have
    seen in Section~\ref{sec2.4} that it is $H$\mbox{-}SLND, therefore $W^H$
    is $\mathbb{P}$-a.s. $\rho$\mbox{-}irregular for any $\rho < (2 H)^{- 1}$,
    recovering the results from~{\cite{catelliergubinelli}}.
    
    \item Let $B$ be a standard Brownian motion in $\mathbb{R}^d$ and let $A
    \in \mathbb{R}^{d \times d}$, $x_0 \in \mathbb{R}^d$, $\sigma \in
    \mathbb{R}_{> 0}$ and a given measurable function $f : [0, T] \rightarrow
    \mathbb{R}^d$; consider a generalised $d$-dimensional Ornstein--Uhlenbeck
    process, solution to the SDE
    \[ \mathd X_t = \left( - A \, X_t + f_t \right) \mathd t + \sigma \mathd
       B_t, \quad X_0 = x_0 . \]
    The explicit expression for $X$ is given by
    
    \begin{align*}
      X_t & = \, e^{- t A} x_0 + \int_0^t e^{- (t - s) A} f_s \mathd s +
      \sigma \int_0^t e^{- (t - s) A} \mathd B_s\\
      & = \, e^{- (t - s) A} X_s + \int_s^t e^{- (t - r) A} f_r \mathd r +
      \sigma \int_s^t e^{- (t - r) A} \mathd B_r
    \end{align*}
    
    which implies that we have a decomposition in $\mathcal{F}_s$-adapted and
    $\mathcal{F}_s$-independent parts given by $X_t = X_{s, t}^{(1)} + X_{s,
    t}^{(2)} \assign \mathbb{E} [X_t | \mathcal{F}_s] + (X_t -\mathbb{E} [X_t
    | \mathcal{F}_s])$ where
    \[ X^{(1)}_{s, t} = e^{- (t - s) A} X_s + \int_s^t e^{- (t - r) A} f_r
       \mathd r, \quad X_{s, t}^{(2)} = \sigma \int_s^t e^{- (t - r) A} \mathd
       B_r . \]
    It follows that for any $s < t$ such that $| s - t | < \delta$ and any $v
    \in \mathbb{R}^d$ it holds
    \[ \tmop{Var} (X_t \cdot v | \mathcal{F}_s \nobracket) = \tmop{Var}
       (X^{(2)}_{s, t} \cdot v) = \sigma^2 \int_s^t | e^{- (t - r) A^T} v |^2
       \mathd r = \sigma^2 \int_0^{t - s} | e^{- r A^T} v |^2 \mathd r ; \]
    choosing $\delta$ small enough such that $\| I_d - e^{- r A^T} \|
    \leqslant 1 / 2$ for all $r < \delta$ we deduce that
    \[ \tmop{Var} (X_t \cdot v | \mathcal{F}_s \nobracket) \gtrsim | v |^2  |
       t - s |, \]
    namely $\tmop{Var} (X_t | \mathcal{F}_s \nobracket) \gtrsim | t - s |
    I_d$. We conclude that $X$ is $\rho$-irregular for any $\rho < 1$.
    
    \item There is a general class of Gaussian processes for which a
    $\beta$\mbox{-}SLND condition holds, given by so called {\tmem{moving
    averages of white noise}}, as already observed by Berman (see Section~3
    from~{\cite{berman73}}). Specifically, let $K : \mathbb{R}^+ \rightarrow
    \mathbb{R}^{d \times d}$ be a function such that $K K^T (t) \gtrsim | t
    |^{2 \beta - 1} I_d$, for some $\beta > 0$ and all $t$ small enough and
    define the process
    \[ X_t \assign \int_0^t K (t - r) \mathd B_r, \]
    where $B$ is a standard Bm in $\mathbb{R}^d$. Then for any $s < t$, we
    have a decomposition in $\mathcal{F}_s$-adapted and
    $\mathcal{F}_s$-independent parts given by
    \[ X_t = X_{s, t}^{(1)} + X_{s, t}^{(2)} = \int_0^s K (t - r) \mathd B_r +
       \int_s^t K (t - r) \mathd B_r \]
    which implies that
    \[ \tmop{Var} (X_t | \mathcal{F}_s \nobracket) = \int_s^t (K K^T) (t - r)
       \mathd r = \int_0^{t - s} (K K^T) (u) \mathd u \gtrsim | t - s |^{2
       \beta} \]
    for all $s, t$ such that $| t - s |$ is small enough; we deduce that in
    this case we have $\rho$\mbox{-}irregularity for any $\rho < (2 \beta)^{-
    1}$. The standard example for this type of processes is for the choice $K
    (t) = k (t) I_d$, where $k : \mathbb{R}_+ \rightarrow \mathbb{R}$ is such
    that $| k (t) | \gtrsim | t |^{\beta - 1 / 2}$. For the choice $k (t) = |
    t |^{H - 1 / 2}$ we obtain the $d$-dimensional. Levy fBm (sometimes also referred
    to as type\mbox{-}II fBm), for which again we have $\rho < (2 H)^{- 1}$.
    
    \item The class of moving averages is closed under integration. Given $K$
    and $X$ as above, defining $Y \assign \int_0^{\cdot} X_s \mathd s$, by
    stochastic Fubini it holds
    \[ Y_t = \int_0^t \int_0^s K (s - r) \mathd B_r \mathd s = \int_0^t
       \int_r^t K (s - r) \mathd s \mathd B_r = \int_0^t \tilde{K} (t - r)
       \mathd B_r \]
    for the choice $\tilde{K} (\cdot) = \int_0^{\cdot} K (s) \mathd s$. For
    the choice $K (t) = | t |^{H - 1 / 2} I_d$, i.e. $X$ being a Levy fBm of
    parameter $H$, then $\tilde{K} (t) = c_H | t |^{(H + 1) - 1 / 2}$, which
    implies that $Y$ is $\rho$\mbox{-}irregular for any $\rho < (2 + 2 H)^{-
    1}$. Technically speaking standard fBm does not belong to the moving
    average class, but it is clear that we can apply the same reasoning to
    formula~\eqref{sec2.4 non canonical representation}, obtaining
    
    \begin{align*}
      Y_t = \int_0^t W^H_r \mathd r & = \, \left[ \int_0^s W^H_r \mathd r +
      \int_s^t \mathbb{E} [W^H_r | \mathcal{F}_s \nobracket] \mathd r \right]
      + \int_s^t (W^H_r \nobracket -\mathbb{E} [W^H_r | \mathcal{F}_s]) \mathd
      r\\
      & = \, Y^{(1)}_{s, t} + c_H \int_s^t \int_s^r (r - u)^{H - 1 / 2}
      \mathd B_u \mathd r\\
      & = \, Y^{(1)}_{s, t} + \tilde{c}_H \int_s^t (t - u)^{(H + 1) - 1 / 2}
      \mathd B_u = Y^{(1)}_{s, t} + Y^{(2)}_{s, t} .
    \end{align*}
    
    Arguing as above, we see that $Y$ is $(1 + H)$-SLND and thus
    $\rho$-irregular for any $\rho < (2 + 2 H)^{- 1}$. Since $W^H$ as $C^{H -
    \varepsilon}$ trajectories, $Y$ has $C^{1 + H - \varepsilon}$
    trajectories, showing that it is still possible for differentiable paths
    to be $\rho$\mbox{-}irregular, for relatively small parameters $\rho$. The
    reasoning can be iterated: if $Y$ is the process obtained by integrating
    $n$ times $W^H$, so that $Y \in C^{n + H -}$ with full probability, then
    $Y_t$ has a decomposition into $\mathcal{F}_s$-adapted and
    $\mathcal{F}_s$-independent parts given by
    \[ Y_t = Y^{(1)}_{s, t} + Y^{(2)}_{s, t}, \quad Y^{(2)}_{s, t} = Y_t
       -\mathbb{E} [Y_t | \mathcal{F}_s] = c_{n, H} \int_s^t (t - u)^{(H + n)
       - 1 / 2} \mathd B_u \]
    which implies that it is $(n + H)$\mbox{-}SLND; moreover for any $k \in \{
    0, \ldots, n \}$, $D^k Y$ is $(H + n - k)$\mbox{-}SLND.
    
    \item The following example is taken from~{\cite{banosproske}} and it
    provides an explicit Gaussian process with continuous trajectories which
    are $\mathbb{P}$-a.s. $\rho$-irregular for all $\rho < \infty$. Functions
    with such properties can be also constructed by prevalence techniques,
    using the fact that countable intersection of prevalent sets is prevalent.
    
    Let $H_n$ be a sequence in $(0, 1)$ such that $H_n \downarrow 0$ and $\{
    B^{H_n} \}_n$ be a sequence of independent fBms in $\mathbb{R}^d$ with
    parameters $H_n$, defined on an interval $[0, T]$; also consider a
    sequence $\lambda_n$ of strictly positive numbers such that
    \[ \sum_n \lambda_n \mathbb{E} \left[ \, \| B^{H_n} \|_{C^0}  \right] <
       \infty \]
    (for instance one can take $\lambda_n = \left( 1 +\mathbb{E} \left[ \, \|
    B^{H_n} \|_{C^0}  \right] \right)^{- 1} n^{- 2}$). Then it holds
    \[ \mathbb{E} \left[ \sum_n \lambda_n  \| B^{H_n} \|_{C^0} \right] <
       \infty \]
    which implies that $\mathbb{P}$-a.s. the series $\sum_n \lambda_n B^{H_n}$
    is absolutely convergent, thus uniformly convergent to an element of
    $C^0_t$; we denote such limit by $Y$, which is therefore a Gaussian
    variable on $C^0_t$. By construction
    \[ \tmop{Var} (Y_t | \mathcal{F}_s \nobracket) \geqslant \lambda^2_n
       \tmop{Var} (B^{H_n}_t | \mathcal{F}_s \nobracket) = \lambda^2_n c_{H_n}
       | t - s |^{2 H_n}, \]
    which implies that $\mathbb{P}$-a.s. $Y$ is $\rho$-irregular for any $\rho
    < (2 H_n)^{- 1}$. As the reasoning holds for all $n$ and $H_n \downarrow
    0$ we conclude that $Y$ is $\mathbb{P}$\mbox{-}a.s.
    $\rho$\mbox{-}irregular for all $\rho < \infty$.
\end{enumerate}
We can also provide an example of a $\beta$\mbox{-}eSLND Gaussian process.

\begin{proposition}
  \label{sec4.2 prop wierd process}Let $\beta > 0$ and consider the
  $d$\mbox{-}dimensional Gaussian process $X$ defined by
  \begin{equation}
    X_t = \int_0^t (t - s)^{- 1 / 2} | \log (t - s) |^{- \beta / 2 - 1 / 2}
    \mathd B_s \quad \forall \, t \in [0, 1 / 2] \label{sec4.2 wierd
    process}
  \end{equation}
  where $B$ is a standard Bm in $\mathbb{R}^d$. Then $X$ admits a modification
  which is $\beta$\mbox{-}eSLND and satisfies the hypothesis of
  Proposition~\ref{sec4.2 prop exp irr}; moreover, $X$ has trajectories in
  $L^p_t$ for any $p < \infty$.
\end{proposition}

\begin{proof}
  The process $X$ is separable, as it is constructed from Bm, which is a
  separable process. Moreover, it is easy to check $X$ is stochastically
  continuous and therefore by Proposition~3.2 from~{\cite{daprato}} it admits
  a measurable modification; from now on we will work with this modification.
  It holds
  \[ \tmop{Var} (X_t) = I_d  \int_0^t (t - s)^{- 1} | \log (t - s) |^{- \beta
     - 1} \mathd s = c_{\beta}  | \log t |^{- \beta} \lesssim 1 \quad \forall
     \, t \in [0, 1 / 2] \]
  which also implies by properties of Gaussian variables that we can find
  $\lambda > 0$ small enough s.t.
  \[ \mathbb{E} \left[ \int_0^{1 / 2} \exp (\lambda | X_t |^2) \mathd t
     \right] = \int_0^{1 / 2} \mathbb{E} [\exp (\lambda | X_t |^2)] \mathd t <
     \infty \]
  which implies that $X \in L^p_t$ for all $p \in [1, \infty)$. Finally, since
  $X$ is the moving average of white noise associated to $K (r) = r^{- 1 / 2} 
  | \log r |^{- \beta / 2 - 1 / 2}$, for any $s < t$ it holds
  \[ \tmop{Var} (X_t | \mathcal{F}_s \nobracket) = I_d \int_0^{t - s} K^2 (r)
     \mathd r = I_d \int_0^{t - s} r^{- 1} | \log r |^{- \beta - 1} \mathd r
     \sim | \log (t - s) |^{- \beta} I_d \]
  which proves the $\beta$\mbox{-}eSLND property.
\end{proof}

\begin{remark}
  We have constructed the process $X$ on the interval $[0, 1 / 2]$ for
  simplicity, but up to rescaling, a process with the same properties can be
  constructed on any finite interval $[0, T]$. Recall that for $\beta
  \leqslant 1$, $X$ is exponentially irregular, thus Carath{\'e}odory and
  unbounded; therefore the $L^p$ (actually, exponential) integrability
  obtained is optimal in this case. While we were working on the present
  paper, we became aware of the work\mbox{-}{\cite{perkowski}}, in which the
  same process is independently introduced and called $p$\mbox{-}$\log$
  Brownian motion, see Definition~4.1 from\mbox{-}{\cite{perkowski}}; the
  regime $p > 1 / 2$ therein corresponds to our $\beta > 0$. However, the
  authors of~{\cite{perkowski}} establish infinite regularisation of the path
  $X$ only in the case $p > 1$ (equivalently $\beta > 1$), see Corollary~4.4
  from~{\cite{perkowski}}, and do not prove the exponential decay of $\Phi^X
  (\xi)$.
\end{remark}

Further examples of $\rho$-irregular functions can be produced by combining a
given $\beta$\mbox{-}SLND Gaussian process $X$ with suitable deterministic
functions.

\begin{proposition}
  \label{sec4.2 prop deterministic perturbations}The following hold:
  \begin{enumerate}[label=\alph*.]
    \item Given a measurable $f : [0, T] \rightarrow \mathbb{R}^d$ and a
    $\beta$\mbox{-}SLND Gaussian process $X$, $Y : = f + X$ is
    $\beta$\mbox{-}SLND; if moreover $X$ satisfies condition~\eqref{sec4.2
    gaussian integrability condition} and $f \in L^{\theta}_t$ for some
    $\theta \in (0, \infty)$, then $Y$ is also $\rho$\mbox{-}irregular for any
    $\rho < 1 / (2 \beta)$. A similar statement holds for $f$ as above and $X$
    $\beta$-eSLND.
    
    \item Given a measurable $f : [0, T] \rightarrow \mathbb{R}$ satisfying
    $c^{- 1} \leqslant | f_t | \leqslant c $ for a suitable constant $c > 0$
    and a process $X$ satisfying the assumptions of Theorem~\ref{sec4.2
    corollary rho gaussian}, $Y \assign f X$ also satisfies the assumptions
    and is therefore $\rho$\mbox{-}irregular for any $\rho < 1 / (2 \beta)$.
    
    \item Suppose $X$ is a $\beta$\mbox{-}SLND Gaussian process with $\beta \in (0, 1]$
    with trajectories in $C^{\alpha}_t$ and $A \in C^{\gamma} ([0, T] ;
    \mathbb{R}^{d \times d})$ is a deterministic function such that $\alpha +
    \gamma > 1$, satisfying
    \[ (A^T A)_t \geqslant c I_d \quad \forall \, t \in [0, T] . \]
    Then the process defined by the Young integral $Y_{\cdot} = \int_0^{\cdot}
    A_s \mathd X_s$ is also Gaussian, $\beta$\mbox{-}SLND, with trajectories
    in $C^{\alpha}_t$; it is $\rho$\mbox{-}irregular for any $\rho < 1 / (2
    \beta)$.
  \end{enumerate}
\end{proposition}

\begin{proof}
  Part~\tmtextit{a)} follows from Remark~\ref{sec2.4 remark invariance lnd}
  and the fact that if $f \in L^{\theta}_t$ and $X$ satisfies the
  integrability assumptions, then so does $f + X$. Regarding~\tmtextit{b)}, it
  is clear that the process $Y$ defined in this way is still Gaussian
  satisfying~\eqref{sec4.2 gaussian integrability condition}. The process $Y$
  is $\beta$-SLND since
  \[ \tmop{Var} (Y_t | \mathcal{F}_s \nobracket) = | f_t |^2 \tmop{Var} (X_t |
     \mathcal{F}_s \nobracket) \sim \tmop{Var} (X_t | \mathcal{F}_s
     \nobracket) \gtrsim | t - s |^{2 \beta} I_d . \]
  It remains to prove~\tmtextit{c)}. By properties of Young integral (see
  Appendix~\ref{appendixA1}) $\varphi \mapsto \int_0^{\cdot} A \mathd \varphi$
  is a bounded linear map from $C^{\alpha}_t$ to itself, therefore $Y$ is a
  Gaussian process on $C^{\alpha}_t$ since $X$ is so; by Fernique theorem, it
  holds $\mathbb{E} [\exp (\lambda \| Y \|_{C^{\alpha}})] < \infty$ for all
  $\lambda \in \mathbb{R}$. By definition of Young integral, it holds
  \[ Y_t = Y_s + A_s (X_t - X_s) + R_{s, t} \]
  where $| R_{s, t} | \lesssim | t - s |^{\alpha + \gamma} \| A
  \|_{C^{\gamma}} \| X \|_{C^{\alpha}}$. This implies that $Y_t$ must have a
  decomposition in $\mathcal{F}_s$\mbox{-}adapted and
  $\mathcal{F}_s$\mbox{-}independent parts of the form $Y_t = Y_{s, t}^{(1)} +
  Y_{s, t}^{(2)} =\mathbb{E} [Y_t | \mathcal{F}_s] + (Y_t -\mathbb{E} [Y_t |
  \mathcal{F}_s])$, so that
  \[ Y^{(1)}_t = Y_s + A_s (\mathbb{E} [X_t | \mathcal{F}_s] - X_s)
     +\mathbb{E} [R_{s, t} | \mathcal{F}_s], \quad Y^{(2)}_{s, t} = A_s (X_t
     -\mathbb{E} [X_t | \mathcal{F}_s]) + R_{s, t} -\mathbb{E} [R_{s, t} |
     \mathcal{F}_s] . \]
  Since conditional expectation is a projection on $L^2 (\Omega, \mathcal{F},
  \mathbb{P})$, it holds
  \[ \mathbb{E} \left[ \, | R_{s, t} -\mathbb{E} [R_{s, t} | \mathcal{F}_s]
     |^2 \right] \leqslant \mathbb{E} [| R_{s, t} |^2] \lesssim | t - s |^{2
     (\alpha + \gamma)}  \| A \|_{C^{\gamma}}^2 \mathbb{E} [\| X
     \|^2_{C^{\alpha}}] \lesssim | t - s |^{2 (\alpha + \gamma)}, \]
  where $\alpha + \gamma > 1 \geqslant \beta$; the variance above is of order
  $o (| t - s |^{2 \beta})$ as $| s - t | \sim 0$ and so the decay of
  $\tmop{Var} (Y_t | \mathcal{F}_s \nobracket)$ is governed by
  \[ \tmop{Var} (A_s  (X_t -\mathbb{E} [X_t | \mathcal{F}_s])) = A_s
     \tmop{Var} (X_t -\mathbb{E} [X_t | \mathcal{F}_s]) A^T_s \gtrsim | t - s
     |^{2 \beta} A_s A_s^T \gtrsim | t - s |^{2 \beta} \]
  whenever $| t - s |$ is small enough. This implies that $Y$ is
  $\beta$\mbox{-}SLND and $\rho$-irregular for any $\rho < 1 / (2 \beta)$.
\end{proof}

\begin{remark}
  The above examples can be further combined together, for instance
  considering
  \[ Y_t \assign \int_0^t A_s \mathd X_s + f_t X_t + g_t \]
  for $A$, $f$, $g$ satisfying the assumptions. One can moreover replace such
  deterministic objects by stochastic processes $Z^i$, independent of $X$,
  satisfying suitable regularity and integrability assumptions; this can be
  readily seen by conditioning on $(Z^i)_i$ first and applying the
  deterministic result. This allows to construct processes with
  $\rho$-irregular trajectories which are not Gaussian nor Markovian.
\end{remark}

\subsection{$\rho$\mbox{-}irregularity for $\alpha$\mbox{-}stable
processes}\label{sec4.3}

Section~\ref{sec4.2} deals exclusively with Gaussian processes, but the
criteria developed in Section~\ref{sec4.1} apply in more general situations,
including Markov processes. Here we treat the case of suitable
$\alpha$\mbox{-}stable processes.

\

Recall that a $d$-dimensional process $X$ is a symmetric $\alpha$-stable process, for some $\alpha\in (0,2)$,
with spherical measure $\mu$ (up to a renormalising constant) if it is a
L{\'e}vy process such that, for any $s < t$,
\[ \mathbb{E} [\exp (i \langle \xi, X_t - X_s \rangle)] = \exp \left( - (t -
   s) \int_{\mathbb{S}^{d - 1}} | \langle \xi, z \rangle |^{\alpha} \mu
   (\mathd z) \right) . \]
See~{\cite{sato}} for more details on the topic. From now on, we will say that
$X$ satisfies the \tmtextit{non\mbox{-}degeneracy condition} if there exists
$c > 0$ such that
\begin{equation}
  G (\xi) \assign \int_{\mathbb{S}^{d - 1}} | \langle \xi, z \rangle
  |^{\alpha} \mu (\mathd z) \geqslant c | \xi |^{\alpha} . \label{sec4.3 alpha
  stable non degeneracy}
\end{equation}
A similar condition has already appeared in the literature on regularisation
by noise, see~{\cite{priola2020}}.

\begin{proposition}
  \label{sec4.3 prop rho irr alpha}Let $X$ be a symmetric $\alpha$-stable
  process satisfying the non\mbox{-}degeneracy condition~\eqref{sec4.3 alpha
  stable non degeneracy}. Then $X$ is $\mathbb{P}$-a.s. $\rho$-irregular for
  any $\rho < \alpha / 2$.
\end{proposition}

\begin{proof}
  Let $\mathcal{F}_t$ be the natural filtration associated to $X$; \ for any
  $\xi \in \mathbb{R}^d$ and $s < t$, by the independence of increments and
  the non degeneracy condition, it holds
  \[ \left| \mathbb{E} \left[ \int_s^t e^{i \langle \xi, X_r \rangle} \mathd r
     | \mathcal{F}_s \nobracket \right] \right| = \left| e^{i \langle \xi, X_s
     \rangle} \int_s^t e^{- (r - s) G (\xi)} \mathd r \right| = \int_0^{t - s}
     e^{- r G (\xi)} \mathd r \leqslant \int_0^{\infty} e^{- c r | \xi
     |^{\alpha}} \mathd r \sim | \xi |^{- \alpha} . \]
  Applying Lemma~\ref{sec4.1 lemma condition rho irr} to the choice $f (t, x)
  = e^{i \xi \cdot x}$, since $C_f = C_{\xi} \lesssim | \xi |^{- \alpha}$,
  there exists $\lambda > 0$ s.t.
  \[ \sup_{s \neq t, \xi \in \mathbb{R}^d} \mathbb{E} \left[ \exp \left(
     \lambda \frac{| \xi |^{\alpha}  | \Phi_{s, t}^X (\xi) |^2}{| t - s |}
     \right) \right] < \infty . \]
  We would like to conclude that the process $X$ is $\rho$-irregular for any
  $\rho < \alpha / 2$, but the integrability condition from
  Theorem~\ref{sec4.1 thm general criterion} is not satisfied. However, the
  process $X$ belongs $\mathbb{P}$-a.s. to $L^{\theta}_t$ for any $\theta <
  \alpha$ (see for instance Example~25.10 from {\cite{sato}}) and therefore we
  can apply Corollary~\ref{sec4.1 cor weaker rho irr} to obtain the
  conclusion.
\end{proof}

\begin{remark}
  The non\mbox{-}degeneracy condition~\eqref{sec4.3 alpha stable non
  degeneracy} is for instance satisfied in the cases of $\mu_1$ being the
  uniform measure on $\mathbb{S}^{d - 1}$ and $\mu_2 = \sum_{i = 1}^d
  \delta_{e_i}$. The associated processes have respectively generators
  $\mathcal{L}_1 = (- \Delta)^{\alpha / 2}$ and $\mathcal{L}_2 = \sum_i (-
  \partial_{x_i}^2)^{\alpha / 2}$.
\end{remark}

The previous examples generalises to anisotropic Markov processes, which show
different irregularity behaviour in different directions.
Property~\eqref{sec4.3 alpha irregularity bound} below could be regarded as a
notion of ``$\alpha$\mbox{-}irregularity''.

\begin{corollary}
  Let $\alpha = (\alpha_1, \ldots, \alpha_d) \in (0, 2)^d$ and let $X$ be a
  $d$\mbox{-}dimensional process whose components $X^{(i)}$ are independent symmetric
  $\alpha_i$-stable processes, so that
  \[ \mathbb{E} [\exp (i \langle \xi, X_t - X_s \rangle)] = \exp \left( - | t
     - s | \sum_i | \xi_i |^{\alpha_i} \right) . \]
  Define $\interleave \xi \interleave^2_{\alpha} = \sum_i | \xi_i
  |^{\alpha_i}$. Then setting $\varphi (x) = \sqrt{x | \log x |}$, it holds
  \begin{equation}
    \sup_{s \neq t, \xi \in \mathbb{R}^d} \frac{\interleave \xi
    \interleave_{\alpha} | \Phi^X_{s, t} (\xi) |}{\sqrt{\log | \xi |} \varphi
    (| t - s |)} < \infty \quad \mathbb{P} \text{-a.s.} \label{sec4.3 alpha
    irregularity bound}
  \end{equation}
\end{corollary}

\begin{proof}
  Going through the same calculations as in Proposition~\ref{sec4.3 prop rho
  irr alpha}, we deduce that there exists $\lambda > 0$ such that
  \[ \sup_{s \neq t, \xi \in \mathbb{R}^d} \mathbb{E} \left[ \exp \left(
     \lambda \frac{\interleave \xi \interleave_{\alpha}^2  | \Phi^X_{s, t}
     (\xi) |^2}{| t - s |} \right) \right] < \infty . \]
  The conclusion then follows from Theorem~\ref{sec4.1 thm general criterion}
  applied to the choice $F (\xi) = \interleave \xi \interleave_{\alpha}$
  (together with a reasoning analogue to that of Corollary~\ref{sec4.1 cor
  weaker rho irr} in order to relax the integrability condition).
\end{proof}

\subsection{$\rho$\mbox{-}irregularity for $\beta$\mbox{-}LND Gaussian
processes}\label{sec4.4}

We now turn to study the case of $\beta$\mbox{-}LND processes. Before
proceeding further, let us explain the importance of $\beta$-LND in the
literature and why they require a different treatment from the
$\beta$\mbox{-}SLND case; for a broader overview, we refer
to~{\cite{geman,xiao}} and the discussion in the introduction
of~{\cite{xiao2007}}.

The LND condition was historically the first one introduced by
Berman~{\cite{berman73}}, then extended to the case of random fields by
Pitt~{\cite{pitt}}; it can be often checked using the \tmtextit{spectral
measure} associated to the Gaussian field. Although the SLND property usually
yields stronger results (e.g. for proving law of iterated logarithm for the
local time of a Gaussian field), it is harder to check in practice, with the
exception of the case where a \tmtextit{moving average representation} is
available. For such reasons, also in view of possible (quite natural) future
extensions of $\rho$-irregularity to the case of random fields, we find it
useful to present a proof solely relying on the LND property here.

As already mentioned in Section~\ref{sec2.4}, there are known examples of
Gaussian processes $X$ satisfy a LND condition but not a SLND one.
Specifically, Cuzick in~{\cite[Section 3]{cuzick1977}} constructs a stationary centered
process $X$, whose spectral measure $F$ is described explicitly and thus
characterizes it uniquely, such that for any $n \in \mathbb{N}$ there exists
$c_n > 0$ such that, for any \ $t_1 < \ldots < t_n$, it holds
\begin{equation}
	\tmop{Var} (X_{t_n} |X_{t_1}, \ldots, X_{t_{n - 1}}) \geqslant c_n  | t_n -
   t_{n - 1} | \label{eq:cuzick1}
\end{equation}
but also such that for any $s < t$ it holds
\begin{equation}
	\tmop{Var} (X_t |X_r, r \leqslant s) = 0. \label{eq:cuzick2}
\end{equation}
As argued in~{\cite{cuzick1977}} and the references therein, condition~\eqref{eq:cuzick1} implies that
$X$ satisfies the $1$-LND condition; on the other hand, not
only~\eqref{eq:cuzick2} clearly implies that a $\beta$-SLND condition must
fail for any $\beta \geqslant 0$, but it also informs us that the constants
$c_n$ appearing in~\eqref{eq:cuzick1} must satisfy $c_n \rightarrow 0$ as $n
\rightarrow \infty$.

For this reason, in the case of $\beta$-LND processes, we will be unable to
establish exponential estimates like those from Lemma~\ref{sec4.1 lemma
condition rho irr}, which are essential in order to apply Theorem~\ref{sec4.1
thm general criterion}: the behaviour of the constants $c_n$ is unclear and we
cannot make naive assumptions on them, cf. Remark~\ref{rem:LND-constants}
below. Therefore we need to proceed differently and find different ways to
establish $\rho$\mbox{-}irregularity. The proof is very closed in spirit to
the ones given by Kahane for fBm in Sections~17 and~18 of~{\cite{kahane}},
which are based on combinatorial techniques combined with analytical
properties of Fourier transform. Here the combinatorial argument is taken
from~{\cite{catelliergubinellipreliminary}}, but it goes back
to~{\cite{nualart2003}}.

\begin{theorem}
  \label{sec4.4 main thm}Let $X$ be a $d$\mbox{-}dimensional $\beta$\mbox{-}LND
  Gaussian process with continuous trajectories. Then for any $\gamma < 1 / 2$
  and any $\rho < (2 \beta)^{- 1}$, $X$ is $(\gamma, \rho)$\mbox{-}irregular
  with probability $1$ and moreover
  \[ \mathbb{E} [\| \Phi^X \|_{\mathcal{W}^{\gamma, \rho}}^n] < \infty \quad
     \forall \, n \in \mathbb{N}. \]
  Furthermore $\mathbb{P}$\mbox{-}a.s. $X$ is $\rho$\mbox{-}irregular for all
  $\rho < (2 \beta)^{- 1}$.
\end{theorem}

The major drawback, compared to Theorem~\ref{sec4.2 corollary rho gaussian},
is that in the $\beta$\mbox{-}LND case we loose exponential estimates,
although we still have moments of any order. The proof is split in several
parts; we start with two lemmas, of analytic and combinatorial nature
respectively. We then provide the main estimate, given in
Proposition~\ref{sec4.4 main proposition}, which allows to conclude the proof
of Theorem~\ref{sec4.4 main thm}.

\

The next lemma shows that, under suitable assumptions, we can obtain
$\mathcal{F} L^{\rho, \infty}$ bounds starting from $\mathcal{F} L^{\rho, p}$
bounds for some $p < \infty$.

\begin{lemma}
  \label{sec4.4 analytic lemma}Suppose $\varphi \in \mathcal{F} L^{\rho, p}$
  has support contained in $B_R$, $\rho \geqslant 0$. Then $\varphi \in
  \mathcal{F} L^{\rho, \infty}$ and there exists a constant $C (\rho, p) > 0$
  such that
  \[ \| \varphi \|_{\mathcal{F} L^{\rho, \infty}} \leqslant C \, R^{d / p} \|
     \varphi \|_{\mathcal{F} L^{\rho, p}} . \]
\end{lemma}

\begin{proof}
  Up to a rescaling argument, it suffices to prove the statement in the case
  $R = 1$. Let $g \in C^{\infty}_c$ such that $g \equiv 1$ on $B (0, 1)$, then
  by assumption $\varphi = \varphi g$ and so $\hat{\varphi} = \hat{\varphi}
  \ast \hat{g}$. But then
  \begin{eqnarray*}
    (1 + | \xi |)^{\rho} | \hat{\varphi} (\xi) | & \leqslant & (1 + | \xi
    |)^{\rho} \int | \hat{\varphi} (\xi - \eta) | \, | \hat{g} (\eta) | \mathd
    \eta\\
    & \lesssim_{\rho} & \int (1 + | \xi - \eta |)^{\rho} | \hat{\varphi} (\xi
    - \eta) | (1 + | \eta |)^{\rho} | \hat{g} (\eta) | \mathd \eta\\
    & \lesssim & \| \varphi \|_{\mathcal{F} L^{\rho, p}} \| g \|_{\mathcal{F}
    L^{\rho, p'}}\\
    & \thicksim_p & \| \varphi \|_{\mathcal{F} L^{\rho, p}}
  \end{eqnarray*}
  which gives the conclusion; in the above, $p'$ denoted the conjugate
  exponent to $p$.
\end{proof}

We now need to introduce some notations. Let $n \in \mathbb{N}$, $S_{2 n}$
denote the group of permutations of $\{ 1, \ldots, 2 n \}$; for $\sigma \in
S_{2 n}$, define $\Delta_{\sigma} = \Delta_{2 n, \sigma, [s, t]} = \left\{
(t_1, \ldots, t_{2 n}) \in [s, t]^{2 n} : \, s \leqslant t_{\sigma (1)}
\leqslant \ldots \leqslant t_{\sigma (2 n)} \leqslant t \right\}$. Moreover,
for a given $(t_1, \ldots, t_{2 n}) \in \Delta_{\sigma}$, define
$(a_k^{\sigma})_{1 \leqslant k \leqslant 2 n - 1}$ by
\[ \sum_{k = 1}^n (X_{t_k} - X_{t_{k + n}}) = \sum_{k = 1}^{2 n - 1}
   a_k^{\sigma} (X_{t_{\sigma (k + 1)}} - X_{t_{\sigma (k)}}) \]
and similarly $\varepsilon (k)$ by $\sum_{k = 1}^n (X_{t_k} - X_{t_{k + n}}) =
\sum_{k = 1}^{2 n} \varepsilon (k) X_{t_k}$ (equivalently, $\varepsilon (k) =
1$ for $k = 1, \ldots n$ and $\varepsilon (k) = - 1$ for $k = n + 1, \ldots, 2
n$). Then we have the following

\begin{lemma}
  \label{sec4.4 combinatorial lemma}For $\sigma \in S_{2 n}$ and $k = 1,
  \ldots, 2 n - 1$ it holds $a_k^{\sigma} = - \sum_{l = 1}^k \varepsilon
  (\sigma (l))$.
\end{lemma}

\begin{proof}
  Let $\sigma \in S_{2 n}$, then by definition we have $\sum_{k = 1}^{2 n}
  \varepsilon (k) X_{t_k} = \sum_{k = 1}^{2 n} \varepsilon (\sigma (k))
  X_{t_{\sigma (k)}}$ and therefore, setting as a convention $a^{\sigma}_0 =
  a^{\sigma}_{2 n} = 0$, by definition of $a^{\sigma}_k$ we have
  \[ \sum_{k = 1}^{2 n} \varepsilon (\sigma (k)) X_{t_{\sigma (k)}} = \sum_{k
     = 1}^{2 n - 1} a_k^{\sigma} (X_{t_{\sigma (k + 1)}} - X_{t_{\sigma (k)}})
     = \sum_{k = 1}^{2 n - 1} (a_{k - 1}^{\sigma} - a_k^{\sigma}) X_{t_{\sigma
     (k)}} \]
  which implies $a_{k - 1}^{\sigma} - a_k^{\sigma} = \varepsilon (\sigma (k))$
  for all $k$ and thus the conclusion.
\end{proof}

\begin{proposition}
  \label{sec4.4 main proposition}Under the assumptions of Theorem~\ref{sec4.4
  main thm}, for any $n \in \mathbb{N}$ there exist $C_n > 0$ s.t.
  \begin{equation}
    \mathbb{E} \left[ \left| \int_s^t e^{i \xi \cdot X_r} \, \mathd r
    \right|^{2 n} \right] \leqslant C_n | t - s |^n  | \xi |^{- n / \beta}
    \quad \forall \, \xi \in \mathbb{R}^d, \hspace{0.5em} s, t \in [0, T] .
  \end{equation}
\end{proposition}

\begin{proof}
  Without loss of generality we can assume $| t - s | < \delta_{2 n}$, where
  $\delta_{2 n}$ is taken from the definition of $\beta$\mbox{-}LND; the
  general case then follows from reasoning as in the beginning of the proof of
  Theorem~\ref{sec4.1 thm general criterion}. We can take $X$ of the form $X =
  m + Y$, where $Y$ is a centred Gaussian $\beta$\mbox{-}LND process and $m
  =\mathbb{E} [X]$ is a deterministic continuous function. We have
  \begin{eqnarray*}
    \mathbb{E} \left[ \left| \int_s^t e^{i \xi \cdot (m_r + Y_r)} \, \mathd r
    \right|^{2 n} \right] & = & \left| \int_{[s, t]^{2 n}} \mathbb{E} \left[
    \prod_{k = 1}^n e^{i \xi \cdot (Y_{t_k} - Y_{t_{k + n}})} \right] \prod_{k
    = 1}^n e^{i \xi \cdot (m_{t_k} - m_{t_{k + n}})} \, \mathd t_1 \ldots
    \mathd t_{2 n} \right|\\
    & \leqslant & \int_{[s, t]^{2 n}} \left| \mathbb{E} \left[ \prod_{k =
    1}^n e^{i \xi \cdot (Y_{t_k} - Y_{t_{k + n}})} \right] \right| \mathd t_1
    \ldots \mathd t_{2 n}\\
    & = & \sum_{\sigma \in S_{2 n}} \int_{\Delta_{\sigma}} \left| \mathbb{E}
    \left[ \prod_{k = 1}^n e^{i \xi \cdot (Y_{t_k} - Y_{t_{k + n}})} \right]
    \right| \mathd t_{\sigma (1)} \ldots \mathd t_{\sigma (2 n)}\\
    & = & \sum_{\sigma \in S_{2 n}} \int_{\Delta_{\sigma}} \left| \mathbb{E}
    \left[ \exp \left( i \xi \cdot \sum_{k = 1}^{2 n - 1} a_k^{\sigma}
    (Y_{t_{\sigma (k + 1)}} - Y_{t_{\sigma (k)}}) \right) \right] \right|
    \mathd t_{\sigma (1)} \ldots \mathd t_{\sigma (2 n)} ;
  \end{eqnarray*}
  introducing the notation $s_k = t_{\sigma (k)}$, we therefore find
  
  \begin{align*}
    \mathbb{E} \left[ \left| \int_s^t e^{i \xi \cdot (m_r + Y_r)} \, \mathd r
    \right|^{2 n} \right] & = \sum_{\sigma \in S_{2 n}} \int_{\Delta_{\sigma}}
    \exp \left( - \frac{1}{2} \tmop{Var} \left( \xi \cdot \sum_{k = 1}^{2 n -
    1} a_k^{\sigma} (Y_{s_{k + 1}} - Y_{s_k}) \right) \right) \, \mathd s_1
    \ldots \mathd s_{2 n}\\
    & \leqslant \sum_{\sigma \in S_{2 n}} \int_{\Delta_{\sigma}} \exp \left(
    - c_{2 n} | \xi |^2  \sum_{k = 1}^{2 n - 1} | s_{k + 1} - s_k |^{2 \beta}
    (a_k^{\sigma})^2 \right) \, \mathd s_1 \ldots \mathd s_{2 n},
  \end{align*}
  where in the last passage we used the $\beta$\mbox{-}LND property. It is
  clear by definition that $\varepsilon (k) \in \{ - 1, 1 \}$ for all $k$ and
  therefore by Lemma~\ref{sec4.4 combinatorial lemma} it follows that, for any
  odd $k$, $(a_k^{\sigma})^2 \geqslant 1$. Estimating $(a_k^{\sigma})^2$
  trivially by $0$ for even $k$, we obtain that, for any $\sigma \in S_{2 n}$,
  \begin{align*}
    \int_{\Delta_{\sigma}} \exp & \left( - c_{2 n} | \xi |^2 \sum_{k = 1}^{2 n
    - 1} | s_{k + 1} - s_k |^{2 \beta} (a_k^{\sigma})^2 \right) \, \mathd s_1
    \ldots \mathd s_{2 n}\\
    & \leqslant \int_{\Delta_{\sigma}} \exp \left( - c_{2 n} | \xi |^2
    \sum_{k = 1}^n | s_{2 k} - s_{2 k - 1} |^{2 \beta} \right) \mathd s_1
    \ldots \mathd s_{2 n} .
  \end{align*}
  Applying the change of variable $y_{2 k} = s_{2 k} - s_{2 k - 1}$ and
  estimating the integrals w.r.t. $y_{2 k}$ on suitable bounded intervals by
  the same integrals taken over $(0, \infty)$, we obtain
  \begin{align*}
    \int_{\Delta_{\sigma}} \exp & \left( - c_{2 n} | \xi |^2 \sum_{k = 1}^n |
    s_{2 k} - s_{2 k - 1} |^{2 \beta} \right) \, \mathd s_1 \ldots \mathd s_{2
    n}\\
    & \leqslant \int_{\Delta_{\sigma}} \, \mathd t_1 \ldots \mathd t_n
    \hspace{0.5em} \left( \int_0^{\infty} e^{- c_{2 n} | \xi |^2 | u |^{2
    \beta}} \, \mathd u \right)^n\\
    & \leqslant \frac{1}{n!} | t - s |^n \left( \int_0^{\infty} e^{- c_{2 n}
    | \xi |^2 | u |^{2 \beta}} \, \mathd u \right)^n .
  \end{align*}
  Summing over $\sigma \in S_{2 n}$ we conclude that
  \[ \mathbb{E} \left[ \left| \int_s^t e^{i \xi \cdot X_r} \, \mathd r
     \right|^{2 n} \right] \leqslant \frac{(2 n) !}{n!}  | t - s |^n \left(
     \int_0^{\infty} e^{- c_{2 n} | \xi |^2 | u |^{2 \beta}} \, \mathd u
     \right)^n \leqslant C_n | t - s |^n  | \xi |^{- n / \beta} . \]
\end{proof}

\begin{remark}
  \label{rem:LND-constants}The constants $c_{2 n}$, coming from the definition
  of $\beta$\mbox{-}LND, depend on $n$ in an unspecified way; this is the
  reason why we are not able to obtain exponential integrability as in
  Theorem~\ref{sec4.2 corollary rho gaussian}. We can overcome this difficulty
  by exploiting Lemma~\ref{sec4.4 analytic lemma}, up to the price of
  restricting ourselves to working with uniformly bounded paths and only
  obtaining moment estimates for the $\rho$\mbox{-}irregularity.
\end{remark}

\begin{proof}[Proof of Theorem~\ref{sec4.4 main thm}]
  By Proposition~\ref{sec4.4 main proposition}, for any $s$, $t \in [0, T]$ it
  holds
  \begin{eqnarray*}
    \mathbb{E} [\| \mu_{s, t}^X \|_{\mathcal{F} L^{\rho, 2 n}}^{2 n}] & = &
    \mathbb{E} \left[ \int (1 + | \xi |)^{2 \rho n} | \Phi^X_{s, t} (\xi) |^{2
    n} \, \mathd \xi \right]\\
    & = & \int (1 + | \xi |)^{2 \rho n} \mathbb{E} [| \Phi^X_{s, t} (\xi)
    |^{2 n}] \, \mathd \xi\\
    & \lesssim_n & | t - s |^n \int (1 + | \xi |)^{2 n \rho - n / \beta} \,
    \mathd \xi\\
    & \lesssim_{n, \rho} & | t - s |^n
  \end{eqnarray*}
  whenever $2 n \rho - n / \beta < - d$, namely $\rho < (1 / \beta - d / n) /
  2$. Then by Kolmogorov continuity criterion (and the fact that $\mu_0^X
  \equiv 0$ by definition), for any $\rho < (1 / \beta - d / n) / 2$ and
  $\gamma < (1 - 1 / n) / 2$, $\mathbb{P}$-a.s. $\mu^X \in C^{\gamma}_t
  \mathcal{F} L^{\rho, 2 n}_x$ and
  \[ \mathbb{E} [\| \mu^X \|^{2 n}_{C^{\gamma} \mathcal{F} L^{\rho, 2 n}}] <
     \infty . \]
  By hypothesis, $X$ is a Gaussian process with continuous trajectories,
  therefore by Fernique theorem $\| X \|_{C^0}$ admits moments of any order.
  Since $\text{supp} (\mu^X_{s, t}) \subset B_{\| X \|_{C^0}}$
  uniformly in $s, t$, by Lemma~\ref{sec4.4 analytic lemma} we deduce that
  \[ \| \mu^X \|_{C^{\gamma} \mathcal{F} L^{\rho, \infty}} \lesssim \| X
     \|_{C^0}^{d / 2 n} \| \mu^X \|_{C^{\gamma} \mathcal{F} L^{\rho, 2 n}} \]
  and thus
  \[ \mathbb{E} [\| \Phi^X \|_{\mathcal{W}^{\gamma, \rho}}^n] =\mathbb{E} [\|
     \mu^X \|_{C^{\gamma} \mathcal{F} L^{\rho, \infty}}^n] \lesssim \mathbb{E}
     [\| X \|_{C^0}^d]^{1 / 2} \mathbb{E} [\| \mu^X \|_{C^{\gamma} \mathcal{F}
     L^{\rho, 2 n}}^{2 n}]^{1 / 2} < \infty . \]
  This implies that for any fixed $n$, with probability~$1$ $X$ is $(\gamma,
  \rho)$\mbox{-}irregular for any $\gamma < (1 - 1 / n) / 2$ and $\rho < (1 /
  \beta - d / n) / 2$. We can then find a set of full probability such that
  $X$ is $(\gamma, \rho)$\mbox{-}irregular for any $\rho < 1 / (2 \beta)$ and
  $\gamma < 1 / 2$; finally by the interpolation argument from
  Lemma~\ref{sec2.2 lemma interpolation rho irr} we can conclude that, with
  probability $1$, $X$ is $\rho$\mbox{-}irregular for any $\rho < 1 / (2
  \beta)$.
\end{proof}

We are finally ready to complete the

\begin{proof}[Proof of Theorem~\ref{sec3 thm1}]
  The statement follows from a combination of Theorem~\ref{sec4.4 main thm}
  above, together with Theorem~\ref{sec4.2 corollary rho gaussian} and
  Proposition~\ref{sec4.2 prop exp irr} from Section~\ref{sec4.2}.
\end{proof}

\section{Analytic properties of $\rho$-irregularity}\label{sec5}

This section is devoted to the study of deterministic $(\gamma,
\rho)$\mbox{-}irregular paths. It includes the proof of Theorems~\ref{sec3
thm3} and~\ref{sec3 thm4}. In Section~\ref{sec5.3} we discuss also what we
call the perturbation problem.

\subsection{Fourier dimension and Salem sets}\label{sec5.1}

We highlight here the connection of $\rho$\mbox{-}irregularity and Fourier
dimension and provide the proof of Theorem~\ref{sec3 thm4}. This connection
was already noticed in~{\cite{choukgubinelli2}}. We start by recalling some
facts of geometric measure theory, which can be found in~{\cite{mattila}}.

\begin{definition}
  Given $E \subset \mathbb{R}^d$ Borel, denote by $\mathcal{M}_+ (E)$ the set
  of positive measures supported on $E$. The Fourier and Hausdorff dimension
  of $E$ correspond respectively to
  \[ \dim_F (E) = \sup \left\{ \alpha \in [0, d] \, : \, \exists \, \mu \in
     \mathcal{M}_+ (E), \, \mu \in \mathcal{F} L^{\alpha / 2, \infty}
     \right\}, \]
  \[ \dim_H (E) = \sup \left\{ \alpha \in [0, d] \, : \, \exists \, \mu \in
     \mathcal{M}_+ (E), \, I^{\alpha} (\mu) < \infty \right\}, \]
  where
  \[ I^{\alpha} (\mu) \assign \int_{\mathbb{R}^{2 d}} \frac{\mu (\mathd x) \mu
     (\mathd y)}{| x - y |^{\alpha}} = c_{\alpha, d} \int | \xi |^{\alpha - d}
     | \hat{\mu} (\xi) |^2 \, \mathd \xi = c_{\alpha, d} \| \mu
     \|^2_{\mathcal{F} L^{\alpha / 2 - d / 2, 2}} . \]
\end{definition}

It is clear from the definition and the embedding $\mathcal{F} L^{s, \infty}
\hookrightarrow \mathcal{F} L^{s - d / 2 - \varepsilon, 2}$ that
\begin{equation}
  0 \leqslant \dim_F (E) \leqslant \dim_H (E) \leqslant d ; \label{sec 5.1
  relation dimensions}
\end{equation}
moreover there are examples in which all inequalities are strict. This
motivates the following definition.

\begin{definition}
  A Borel set $E \subset \mathbb{R}^d$ is a Salem set if $\dim_F (E) = \dim_H
  (E)$.
\end{definition}

If $w$ is $(\gamma, \rho)$\mbox{-}irregular, it is clear that for any $[s, t]
\subset [0, T]$ it holds
\[ \mu^w_{s, t} \in \mathcal{F} L^{\rho, \infty}, \quad I^{\alpha} (\mu^w_{s,
   t}) \lesssim_{\alpha} | t - s |^{2 \gamma} \| \Phi^w
   \|^2_{\mathcal{W}^{\gamma \comma \rho}} \quad \forall \, \alpha < 2 \rho .
\]
In particular, since $\mu^w_{s, t}$ is a measure supported on $w ([s, t])$, it
holds
\[ \min (d, 2 \rho) \leqslant \dim_F (w ([s, t])) \quad \forall \, [s, t]
   \subset [0, T] . \]
On the other hand, recall that if $f \in C^{\delta}_t$ for some $\delta\in (0,1]$, then for any $[s, t]
\subset [0, T]$ it holds $\dim_H (f ([s, t])) \leqslant \delta^{- 1}$, see for
instance Proposition~3.3~(a) from~{\cite{falconer}}.

We are now ready to give the proof of Theorem~\ref{sec3 thm4}.

\begin{proof}[Proof of Theorem~\ref{sec3 thm4}]
  To prove Point~\tmtextit{i}., first assume $\delta \geqslant 1 / d$. By
  Theorem~\ref{sec3 thm2}, almost every $\varphi \in C^{\delta}_t$ is
  $\rho$-irregular for any $\rho < (2 \delta)^{- 1}$. It then follows
  from~\eqref{sec 5.1 relation dimensions} and the considerations above that
  \[ 2 \rho = \min (d, 2 \rho) \leqslant \dim_F (\varphi ([s, t])) \leqslant
     \dim_H (\varphi ([s, t])) \leqslant \delta^{- 1} ; \]
  since the inequality holds for all $\rho < (2 \delta)^{- 1}$, the conclusion
  follows. The case $\delta < 1 / d$ is even more direct, since in this case
  we can find $\rho < (2 \delta)^{- 1}$ such that $2 \rho > d$ and therefore
  we obtain
  \[ d = \min (d, 2 \rho) \leqslant \dim_F (\varphi ([s, t])) \leqslant \dim_H
     (\varphi ([s, t])) \leqslant d. \]
  Suppose now additionally that $\delta < (2 d)^{- 1}$, then almost every
  $\varphi \in C^{\delta}_t$ is $\rho$\mbox{-}irregular for some $\rho > d$;
  by Point~\tmtextit{ii.} of Lemma~\ref{sec2.2 lemma regularity averaging 2}
  it follows that $\mu^{\varphi}_{s, t}$ admits a continuous occupation
  density $\ell^{\varphi}_{s, t}$. Therefore there exists $x \in \varphi ([s,
  t])$ such that $\ell^{\varphi}_{s, t} (x) > \varepsilon > 0$ and by
  continuity the same must hold on an open ball $B (x, r)$ for some $r > 0$;
  this implies that $B (x, r) \subset \varphi ([s, t])$.
\end{proof}

It is possible to show that $\rho$\mbox{-}irregular paths cannot be
$\delta$\mbox{-}H\"{o}lder for $\delta$ too large reasoning by dimensionality,
since otherwise it wouldn't be true that $\dim_F (w ([s, t])) \geqslant \min
(d, 2 \rho)$; in the next section we are going to provide a much sharper
result.

\subsection{$\rho$\mbox{-}irregular paths are rough}\label{sec5.2}

The results of this section are inspired by the similar discussion carried out
in Sections~9\mbox{-}11 of~{\cite{geman}}, in which it is shown that functions
with sufficiently regular occupation densities must exhibit a quite erratic
behaviour. Let us point out however that here we only assume the function $w$
to be $(\gamma, \rho)$\mbox{-}irregular, which in general does not imply the
existence of an occupation density. Theorem~\ref{sec3 thm3} follows from the
results of this section and implies that the prevalence results from
Theorem~\ref{sec3 thm2} are sharp (cf. Remark~\ref{sec5.2 discussion
optimality}).

The next statement shows that $(\gamma, \rho)$\mbox{-}irregularity is indeed a
notion of irregularity, in a sense that can be explicitly quantified. We
recall to the reader that the critical parameter $\delta^{\ast}_{\gamma,
\rho}$ is given by
\[ \delta^{\ast}_{\gamma, \rho} = \frac{1 - \gamma}{\rho} . \]
Throughout this section, we will always assume $\delta^\ast_{\gamma,\rho}\in (0,1)$, which in view of Theorem \ref{sec3 thm2} is the generic scenario for $\delta$-H\"older continuous paths $w$ with $\delta\in (0,1)$, where $\delta^\ast_{\gamma,\rho}$ can get arbitrarily close to $\delta$.
In the case $\delta^\ast_{\gamma,\rho}>1$ (resp. $\delta>1$), many of the concepts introduced below don't have a direct analogue; here however, again by Theorem \ref{sec3 thm2}, we expect the path $w$ in consideration to be such that the highest order derivative $D^{(n)}w$, for $n=\lfloor \delta \rfloor$, is again $(\gamma',\rho')$-irregular for some new parameters $(\gamma',\rho')$, with $\delta^{\ast}_{\gamma',\rho'}<1$; the results presented here can then be applied successfully to $D^{(n)}w$.

\begin{theorem}
  \label{sec5.2 proposition irregularity}Let $w$ be a $(\gamma,
  \rho)$-irregular function. Then for any $\delta > \delta^{\ast}_{\gamma,
  \rho}$, $w$ is nowhere $\delta$\mbox{-}H\"{o}lder continuous. In particular,
  for any fixed $M > 0$ and any $s \in [0, T]$, the set of points $t$ around
  $s$ which satisfy an approximate $\delta$\mbox{-}H\"{o}lder condition with
  constant $M$ is a zero density set:
  \[ \lim_{\varepsilon \rightarrow 0^+} \frac{1}{\mathcal{L} (B (s,
     \varepsilon) \cap [0, T])} \mathcal{L} (t \in B (s, \varepsilon) \cap [0,
     T] : | w_t - w_s | \leqslant M | t - s |^{\delta}) = 0. \]
\end{theorem}

\begin{proof}
  First consider the case of $\rho \leqslant d$. Let $\psi \in C^{\infty}_c
  (\mathbb{R}^d)$ be a radially symmetric, decreasing function such that $\psi
  \equiv 1$ on $B_1$. Fix $M > 0$ and let us consider first $s \in (0, T)$, so
  that $B (s, \varepsilon) \cap [0, T] = (s - \varepsilon, s + \varepsilon)$
  for $\varepsilon$ small enough; up to translation we can assume $w_s = 0$.
  We have
  \begin{eqnarray*}
    \frac{1}{2 \varepsilon} \mathcal{L} (t \in (s - \varepsilon, s +
    \varepsilon) : | w_t - w_s | \leqslant M | t - s |^{\delta}) & \leqslant &
    \frac{1}{2 \varepsilon} \mathcal{L} (t \in (s - \varepsilon, s +
    \varepsilon) : | w_t | \leqslant M \varepsilon^{\delta})\\
    & \leqslant & \frac{1}{2 \varepsilon} \mathcal{L} \left( t \in (s -
    \varepsilon, s + \varepsilon) : \psi \left( \frac{w_t}{M
    \varepsilon^{\delta}} \right) \geqslant 1 \right)\\
    & \leqslant & \frac{1}{2 \varepsilon} \int_{s - \varepsilon}^{s +
    \varepsilon} \psi \left( \frac{w_t}{M \varepsilon^{\delta}} \right) \,
    \mathd t\\
    & = & \frac{1}{2 \varepsilon} \int_{\mathbb{R}^d} \psi \left( \frac{y}{M
    \varepsilon^{\delta}} \right) \mu_{s - \varepsilon, s + \varepsilon}^w
    (\mathd y)\\
    & = & \frac{1}{2} M^d \varepsilon^{d \delta - 1} \int_{\mathbb{R}^d}
    \hat{\psi} (M \varepsilon^{\delta} \xi) \overline{\hat{\mu}_{s -
    \varepsilon, s + \varepsilon}^w (\xi)} \mathd \xi\\
    & \lesssim_M & \varepsilon^{\gamma + \delta d - 1} \| \Phi^w
    \|_{\mathcal{W}^{\gamma \comma \rho}} \int_{\mathbb{R}^d} | \hat{\psi} (M
    \varepsilon^{\delta} \xi) |  (1 + | \xi |)^{- \rho} \, \mathd \xi .
  \end{eqnarray*}
  By H\"{o}lder inequality, for any $q > d / \rho \geqslant 1$, setting $1 / p
  = 1 - 1 / q$ it holds
  \[ \int_{\mathbb{R}^d} | \hat{\psi} (M \varepsilon^{\delta} \xi) |  (1 + |
     \xi |)^{- \rho} \, \mathd \xi \, \lesssim_{M, q} \, \varepsilon^{- \delta
     d / p}  \| \hat{\psi} \|_{L^p} . \]
  Therefore we obtain
  \[ \frac{1}{2 \varepsilon} \mathcal{L} (t \in (s - \varepsilon, s +
     \varepsilon) : | w_t - w_s | \leqslant M | t - s |^{\delta}) \,
     \lesssim_{M, q} \, \| \Phi^w \|_{\mathcal{W}^{\gamma \comma \rho}} \,
     \varepsilon^{\gamma + \delta d / q - 1}, \]
  where the last quantity is infinitesimal as $\varepsilon \rightarrow 0$ for
  any $q$ such that $d / q < \rho$ and $\gamma + \delta d / q > 1$. In
  particular if $\delta$ satisfies $\delta > \delta^{\ast}_{\gamma, \rho}$,
  then it's always possible to find such $q$, which gives the conclusion for
  $s \in [0, T]$. The reasoning at the endpoints $\{ 0, T \}$ is analogous:
  for instance in the case $s = 0$, similar calculations yield
  \[ \frac{1}{\varepsilon} \mathcal{L} (t \in (0, \varepsilon) : | w_t - w_0 |
     \leqslant M | t |^{\delta}) \lesssim \| \Phi^w \|_{\mathcal{W}^{\gamma
     \comma \rho}} \, \varepsilon^{\gamma + d \delta / q - 1} . \]
  This concludes the proof in the case $\rho \leqslant d$. For $\rho > d$, by
  Lemma~\ref{sec2.2 lemma interpolation rho irr} and
  equation~\eqref{eq:critical-parameter}, we can find a new pair
  $(\tilde{\gamma}, \tilde{\rho})$ with $\tilde{\rho} \leqslant d$ such that
  $w$ is $(\tilde{\gamma}, \tilde{\rho})$-irregular and $\delta_{\gamma,
  \rho}^{\ast} = \delta_{\tilde{\gamma}, \tilde{\rho}}^{\ast}$; the conclusion
  then follows from the previous case.
\end{proof}

\begin{remark}
  It is clear from the proof that the statement can be localised as follows.
  For a fixed $s \in [0, T]$, if the map $t \mapsto \mu^w_t$ satisfies an
  approximate $\gamma$-Holder condition in $\mathcal{F} L^{\rho, \infty}$
  around s, namely there exist constants $C, r > 0$ such that
  \[ \| \mu^w_{s, t} \|_{\mathcal{F}^{\rho, \infty}} \leqslant C | t - s
     |^{\gamma}  \enspace \text{ for all } t \in (s - r, s + r), \]
  then $w$ cannot satisfy an approximate $\delta$-Holder condition around $s$
  for any $\delta > \delta^{\ast}_{\gamma, \rho}$.
\end{remark}

Let us recall that for a given path $w \in C^0_t$, its $p$-variation on an
interval $[s, t] \subset [0, T]$ is defined as
\[ \| w \|_{p - \tmop{var} ; [s, t]} = \left( \sup_{\Pi} \sum_i | w_{t_i, t_{i
   + 1}} |^p \right)^{1 / p} \]
where the supremum is taken over all possible finite partitions $\Pi = \{ s =
t_0 < \ldots < t_n = t \}$ of $[s, t]$.

\begin{corollary}
  \label{cor:p-variation}Let $w$ be a $(\gamma, \rho)$-irregular function.
  Then for any $p < (\delta^{\ast}_{\gamma, \rho})^{- 1}$ and any $[s, t]
  \subset [0, T]$, $\| w \|_{p - \tmop{var} ; [s, t]} = + \infty$.
\end{corollary}

\begin{proof}
  Since the $\rho$\mbox{-}irregularity property is scaling invariant and the
  $p$\mbox{-}variation is invariant under reparametrization, it suffices to
  show that if $w$ is $(\gamma, \rho)$\mbox{-}irregular, then $\| w \|_{p -
  \tmop{var}, [0, 1]} = \infty$ for any $p$ as above. Going through analogous
  computations to those of Theorem~\ref{sec5.2 proposition irregularity}, it
  can be shown that for any $\delta > \delta^{\ast}_{\gamma, \rho}$ it holds
  \[ \lim_{\varepsilon \rightarrow 0^+} \sup_{s \in [0, 1 - \varepsilon]}
     \frac{1}{\varepsilon} \mathcal{L} (t \in (s, s + \varepsilon) : | w_{s,
     t} | \leqslant \varepsilon^{\delta}) = 0. \]
  In particular, for all $\varepsilon > 0$ small enough it holds
  \[ \sup_{s \in [0, 1 - \varepsilon]} \mathcal{L} (t \in (s, s + \varepsilon)
     : | w_{s, t} | > \varepsilon^{\delta}) > \frac{\varepsilon}{2} > 0 ; \]
  thus for any $s \in [0, 1 - \varepsilon]$, there exists $t \in (s, s +
  \varepsilon)$ such that $| w_{s, t} | > \varepsilon^{\delta}$. Taking $n
  \sim 1 / \varepsilon$, we can construct a partition $\{ 0 = t_0 < \ldots <
  t_{2 n} = 1 \}$ such that $t_{2 k} = k \varepsilon$ and $t_{2 k + 1} \in
  (t_{2 k}, t_{2 k + 2})$ has the above property. We obtain
  \[ \| w \|_{p - \tmop{var}}^p \geqslant \sum_{k = 0}^{2 n - 1} | w_{t_k,
     t_{k + 1}} |^p \gtrsim \varepsilon^{p \delta - 1} ; \]
  since $\varepsilon$ can be taken arbitrarily small, if $p < 1 / \delta$ then
  $\| w \|_{p - \tmop{var}} = \infty$. As the reasoning holds for any $\delta
  > \delta^{\ast}_{\gamma, \rho}$, the conclusion follows.
\end{proof}

Theorem~\ref{sec5.2 proposition irregularity} suggests that the behaviour of
$w$ is quite wild. This intuition can be captured by the following notions of
irregularity introduced in~{\cite{frizdoobmeyer}} and nicely presented
in~{\cite{frizhairer}}.

\begin{definition}
  We say that a path $w \in C^{\delta}_t$ is rough at time $s$, $s \in (0,
  T)$, if
  \[ \limsup_{t \rightarrow s} \frac{| v \cdot w_{s, t} |}{| t - s |^{2
     \delta}} = + \infty \quad \forall \, v \in \mathbb{S}^{d - 1} . \]
  A path $w$ is truly rough if it is rough on some dense set of $[0, T]$.
\end{definition}

\begin{definition}
  A path $w \in C^{\delta}_t$ is $\theta$-H{\"o}lder rough for $\theta \in (0,
  1)$ on scale $\varepsilon_0$ if there exists a constant $L \assign
  L_{\theta} (w) \assign L (\theta, \varepsilon_0, T ; w) > 0$ such that for
  every $v \in \mathbb{S}^{d - 1}$, $s \in [0, T]$ and $\varepsilon \in (0,
  \varepsilon_0]$ there exists $t \in [0, T]$ such that
  \[ | t - s | < \varepsilon \quad \tmop{and} \quad | v \cdot (w_{s, t}) |
     \geqslant L_{\theta} (w) \, \varepsilon^{\theta} . \]
  The largest such value of $L$ is called the modulus of $\theta$-H{\"o}lder
  roughness of $w$. 
\end{definition}

\begin{corollary}
  \label{cor:rho-holder-roughness}Let $w$ be a $(\gamma, \rho)$-irregular
  path; then for any $\theta > \delta^{\ast}_{\gamma, \rho}$, $w$ is
  $\theta$-H{\"o}lder rough with infinite modulus of $\theta$-H{\"o}lder
  roughness.
\end{corollary}

\begin{proof}
  For simplicity we show all the properties for $s$ away from the endpoints
  $\{ 0, T \}$, but it is easy to check how all the reasonings can be adapted
  in the other case. Recall that, if $w$ is $(\gamma, \rho)$\mbox{-}irregular
  and $v \in \mathbb{S}^{d - 1}$, then $v \cdot w$ is $(\gamma,
  \rho)$-irregular and $\| \Phi^{v \cdot w} \|_{\mathcal{W}^{\gamma \comma
  \rho}} \leqslant \| \Phi^w \|_{\mathcal{W}^{\gamma \comma \rho}}$. The
  calculations in the proof of Theorem~\ref{sec5.2 proposition irregularity}
  show that, for any $\delta > \delta^{\ast}_{\gamma, \rho}$ and any $M > 0$,
  \[ \lim_{\varepsilon \rightarrow 0^+} \frac{1}{2 \varepsilon} \mathcal{L} \{
     t \in (s - \varepsilon, s + \varepsilon) : | v \cdot \nobracket w_{s, t}
     | \nobracket \geqslant M \varepsilon^{\delta} \} = 1, \]
  where the rate of convergence only depends on $M$ and $\| \Phi^{v \cdot w}
  \|_{\mathcal{W}^{\gamma \comma \rho}} \leqslant \| \Phi^w
  \|_{\mathcal{W}^{\gamma \comma \rho}}$ and it is thus uniform in $s$ and
  $v$. For fixed $M$ we can then find $\varepsilon_0 = \varepsilon_0 (M,
  \delta)$ sufficiently small such that, it holds
  \[ \frac{1}{2 \varepsilon} \mathcal{L} \{ t \in (s - \varepsilon, s +
     \varepsilon) : | v \cdot w_{s, t} \nobracket | \nobracket \geqslant M
     \varepsilon^{\delta} \} \geqslant \frac{1}{2}, \]
  for all $\varepsilon \in (0, \varepsilon_0]$, uniformly in $s$ and $v$.
  Since the set has non-zero Lebesgue measure, it's always possible to find $t
  \in (s - \varepsilon, s + \varepsilon)$ such that $| v \cdot \nobracket
  w_{s, t} | \nobracket \geqslant M \varepsilon^{\delta}$, which shows that
  the definition of $\theta$\mbox{-}H\"{o}lder roughness is satisfied with
  $\theta = \delta$ and $L \geqslant M$. By the arbitrariness of $M$ we can
  conclude.
\end{proof}

\begin{proof}[Proof of Theorem~\ref{sec3 thm3}]
  Follows from a combination of Corollaries~\ref{cor:p-variation}
  and~\ref{cor:rho-holder-roughness}.
\end{proof}

We conclude this subsection with the promised claim from Remark~\ref{sec5.2
discussion optimality} that typical fBm paths cannot be $(\gamma,
\rho)$-irregular for values $\rho > (2 H)^{- 1}$; the proof can also be
readapted to consider other Gaussian processes.

\begin{lemma}
  \label{lem:negative-answer-fbm}Let $W^H$ be a fBm of parameter $H$, then for
  any $s < t$ and any $\rho > (2 H)^{- 1}$ it holds
  \[ \mathbb{E} [\| \mu_{s, t} \|_{\mathcal{F} L^{\rho, \infty}}^2] = \infty
     . \]
\end{lemma}

\begin{proof}
  Up to rescaling, we can assume $s = 0$, $t = 1$. Since $\mathcal{F} L^{\rho,
  \infty} \hookrightarrow H^{\rho - d / 2 -}$, in order to conclude it
  suffices to show that
  \[ \mathbb{E} \left[ \left\| \mu_1 \right\|^2_{H^{1 / (2 H) - d / 2}}
     \right] = \infty . \]
  This quantity can now be computed explicitly:
  \begin{eqnarray*}
    \mathbb{E} \left[ \left\| \mu_1 \right\|^2_{H^{1 / (2 H) - d / 2}} 
    \right] & = & \mathbb{E} \left[ \int_{\mathbb{R}^d} | \xi |^{1 / H - d}
    \left| \int_0^1 e^{i \xi \cdot W^H_s} \mathd s \right|^2 \mathd \xi
    \right]\\
    & = & \int_{\mathbb{R}^d} \int_{[0, 1]^2} | \xi |^{1 / H - d} \mathbb{E}
    [e^{i \xi \cdot (W^H_{s, t})}] \mathd t \mathd s \mathd \xi\\
    & = & \int_{\mathbb{R}^d} \int_{[0, 1]^2} | \xi |^{1 / H - d} \exp \left(
    - \frac{| \xi |^2 | t - s |^{2 H}}{2} \right) \mathd t \mathd s \mathd
    \xi\\
    & = & \int_{\mathbb{R}^d} | \tilde{\xi} |^{1 / H - d} e^{- | \tilde{\xi}
    |^2 / 2} \mathd \tilde{\xi} \cdot \int_{[0, 1]^2} | t - s |^{^{- 1}}
    \mathd t \mathd s = \infty
  \end{eqnarray*}
  where in the last passage we use the change of variables $\tilde{\xi} = \xi
  | t - s |^H$.
\end{proof}

\subsection{The general perturbation problem}\label{sec5.3}

As in the previous section, here we will work exclusively with $(\gamma,\rho)$-irregular paths $w$ such that $\delta^\ast_{\gamma,\rho}<1$. 

The perturbation problem was first introduced in Section~12 of~{\cite{geman}},
in the context of paths which admit an occupation density. In the case of
$\rho$\mbox{-}irregularity, it can be reformulated as:

\begin{center}
  If $w$ is $\rho$-irregular and $\varphi$ is a sufficiently regular function,
  is $w + \varphi$ still $\rho$-irregular?
\end{center}

We address here the more general question:

\begin{center}
  Which classes of transformations preserve the property of $(\gamma,
  \rho)$\mbox{-}irregularity?
\end{center}

It follows from the results of the previous section that good candidates are
transformations which preserve the very oscillatory behaviour of $w$, namely
at least the property that
\begin{equation}
  \limsup_{t \rightarrow s} \frac{| v \cdot w_{s, t} |}{| t - s |^{\delta}} =
  + \infty \quad \text{for all } s \in (0, T), \, v \in \mathbb{S}^{d - 1}, \,
  \delta > \delta^{\ast}_{\gamma, \rho} . \label{sec5.3 basic property}
\end{equation}
Interestingly, it turns out that several transformations $F : C^0_t
\rightarrow C^0_t$ have the property that if $w$ is $(\gamma,
\rho)$\mbox{-}irregular, then $F (w)$ is $(\tilde{\gamma},
\tilde{\rho})$\mbox{-}irregular with parameters such that
$\delta^{\ast}_{\gamma, \rho} = \delta^{\ast}_{\tilde{\gamma}, \tilde{\rho}}$,
so that property~\eqref{sec5.3 basic property} is preserved. In many cases we
are however unable to show that $(\gamma, \rho) = (\tilde{\gamma},
\tilde{\rho})$, which remains a major open problem. A notable exception is
given by the additive perturbations $F (w) = w + \varphi$ with $\varphi \in
C^{\infty}_t$, whose treatment is postponed to the next subsection.

\

We start by showing that $(\gamma, \rho)$\mbox{-}irregularity is invariant
under sufficiently regular time reparametrization.

\begin{lemma}
  \label{sec5.2 lemma time reparametrization}Let $w$ be $(\gamma,
  \rho)$-irregular, $g \in C^{\beta}_t$ with $\beta + \gamma > 1$. Then
  \[ \left| \int_s^t e^{i \xi \cdot w_r} \, g_r \mathd r \right| \lesssim
     \tmcolor{black}{{\| \Phi^w \|_{\mathcal{W}^{\gamma \comma \rho}}} } \| g
     \|_{C^{\beta}} | t - s |^{\gamma} | \xi |^{- \rho} \quad \text{uniformly
     in $\xi \in \mathbb{R}^d$.} \]
  In particular, for $\beta$ as above, let $\tau : [0, T] \rightarrow [\tau
  (0), \tau (T)]$ be a $C_t^{1 + \beta}$-diffeomorphism, i.e. $\tau \in C_t^{1
  + \beta}$ is invertible on its image with inverse of class $C_t^{1 +
  \beta}$. Then $\tilde{w}_r \assign w_{\tau^{- 1} (r)}$ is also
  $\rho$\mbox{-}irregular and
  \begin{equation}
    {\| \Phi^{\tilde{w}} \|_{\mathcal{W}^{\gamma \comma \rho}}}  \lesssim \|
    \tau^{- 1} \|_{C^1}^{\gamma} \| \tau \|_{C^{1 + \beta}} \| \Phi^w
    \|_{\mathcal{W}^{\gamma \comma \rho}} . \label{sec5.3 estimate
    reparametrization}
  \end{equation}
\end{lemma}

\begin{proof}
  Let $w$, $g$ be as above. Then by properties of Young integral it holds
  \begin{eqnarray*}
    \left| \int_s^t e^{i \xi \cdot w_r} \, g_r \mathd r \right| & = & \left|
    \int_s^t g_r \mathd \left( \int_s^r e^{i \xi \cdot w_u} \mathd u \right)
    \right|\\
    & \lesssim & | g_s | \left| \int_s^t e^{i \xi \cdot w_r} \mathd r \right|
    + | t - s |^{\beta + \gamma} \llbracket g \rrbracket_{C^{\beta}}
    \left\llbracket \int_s^{\cdot} e^{i \xi \cdot w_r} \mathd r
    \right\rrbracket_{C^{\gamma}}\\
    & \lesssim & | t - s |^{\gamma} | \xi |^{- \rho} \| g \|_{C^{\beta}}
    \tmcolor{black}{\| \Phi^w \|_{\mathcal{W}^{\gamma \comma \rho}}},
  \end{eqnarray*}
  which gives the first claim. Applying the change of variables $\tilde{r} =
  \tau^{- 1} (r)$, we then have
  
  \begin{align*}
    \left| \int_s^t e^{i \xi \cdot \tilde{w}_r} \, \mathd r \right| = \left|
    \int_{\tau^{- 1} (s)}^{\tau^{- 1} (t)} e^{i \xi \cdot w_r} \tau'_r \mathd
    r \right| & \lesssim | \tau^{- 1} (t) - \tau^{- 1} (s) |^{\gamma} | \xi
    |^{- \rho} \| \tau' \|_{C^{\beta}} \| \Phi^w \|_{\mathcal{W}^{\gamma,
    \rho}}\\
    & \lesssim | t - s |^{\gamma} | \xi |^{- \rho} \| \tau^{- 1}
    \|^{\gamma}_{C^1} \| \tau \|_{C^{1 + \beta}} \| \Phi^w
    \|_{\mathcal{W}^{\gamma \comma \rho}}
  \end{align*}
  which implies~\eqref{sec5.3 estimate reparametrization}.
\end{proof}

\begin{remark}
  We have only defined the notion of $\rho$\mbox{-}irregularity in terms of
  push-forward under $w$ of the Lebesgue measure $\mathcal{L}_{[s, t)}$, but
  we could consider more generally the Fourier transform of the push-forward
  under $w$ of a bounded Borel measure $\nu$ on $[0, T]$, namely
  \[ \widehat{w_{\ast} \nu_{s, t}} (\xi) = \int_{[s, t)} e^{i \xi \cdot w_r}
     \, \mathd \nu (r) . \]
  The first part of previous lemma could then be interpreted as follows: if
  $w$ is $(\gamma, \rho)$\mbox{-}irregular w.r.t. $\mathcal{L}$, then it is
  also $(\gamma, \rho)$\mbox{-}irregular w.r.t. any measure $\nu \ll
  \mathcal{L}$ with a sufficiently regular density $g$.
\end{remark}

\begin{remark}
  Lemma~\ref{sec5.2 lemma time reparametrization} can also be used to further
  enlarge the class of stochastic processes $X$ which are
  $\rho$\mbox{-}irregular: given any such $X$ and any random $C^{3 /
  2}_t$\mbox{-}diffeomorphism, $Y_t : = X_{\tau^{- 1} (t)}$ is still a
  $\rho$\mbox{-}irregular process.
\end{remark}

Let us make some considerations based on the result above. Recall that if $f
\in C^{\delta}_t$ for $\delta \in (0, 1)$ and $\tau$ is sufficiently regular
(i.e. bi\mbox{-}Lipschitz), then $f \circ \tau^{- 1}$ is still $C^{\delta}_t$,
but this is not true for a general homeomorphism $\tau$. On the other hand, if
$f \in C^0_t$ has finite $1 / \delta$\mbox{-}variation, then there exist a
homeomorphism $\tau$ and $g \in C^{\delta}_t$ such that $f \circ \tau^{- 1} =
g$ (see for instance Proposition~5.15 from~{\cite{frizvictoir}}). Moreover the
$1 / \delta$\mbox{-}variation is a quantity invariant under time
reparametrization. Lemma~\ref{sec5.2 lemma time reparametrization} suggests
that the situation here could be similar: $(\gamma, \rho)$\mbox{-}irregularity
is preserved only if the reparametrization is smooth enough, but there might
exist another underlying property which is invariant under a larger class of
homeomorphism $\tau$. We formulate this as a conjecture.

\begin{conjecture}
  \label{conj:rho}For any pair $(\gamma, \rho)$, there exists a property
  $\mathcal{P}$ such that:
  \begin{enumerate}[label=\arabic*.]
    \item For any $f \in C^0_t$ with property $\mathcal{P}$ there exists a
    homeomorphism $\tau$ such that $g = f \circ \tau^{- 1}$ is $(\gamma,
    \rho)$\mbox{-}irregular.
    
    \item The property $\mathcal{P}$ is invariant under time
    reparametrization.
  \end{enumerate}
\end{conjecture}

In the rest of the section, we will address the perturbation problem only for
transformations $z = F (w)$ with a very specific structure, which makes $z$
locally look like $w$. \ The treatment is a bit abstract, but simple examples
will be given in Remark~\ref{sec5.3 remark examples}.

\begin{definition}
  \label{sec5.3 defn controlled}Let $w$ be $(\gamma, \rho)$-irregular. We say
  that $z : [0, T] \rightarrow \mathbb{R}^d$ is controlled by $w$ with
  ``derivative'' $z'$ if there exist $z' \in C^0 ([0, T] ; \mathbb{R}^{d
  \times d})$ and $R \in C^{\beta}_2 ([0, T] ; \mathbb{R}^d)$ with $\beta >
  \delta^{\ast}_{\gamma, \rho}$ such that
  \[ z_{s, t} = z'_s w_{s, t} + R_{s, t} \quad \text{for all } (s, t) \in
     \Delta_2 . \]
\end{definition}

Here $\Delta_2 = \{ (s, t) : 0 \leqslant s < t \leqslant T \}$ and $R \in
C^{\beta}_2 ([0, T] ; \mathbb{R}^d)$ means that $R : \Delta_2 \rightarrow
\mathbb{R}^d$ and it satisfies
\[ \| R \|_{\beta} \assign \sup_{s < t} \frac{| R_{s, t} |}{| t - s |^{\beta}}
   < \infty . \]
The definition of controlled paths is usually given in the rough paths
framework, see for instance~{\cite{gubinelli}} and~{\cite{frizhairer}}; \
however here we do not impose $w, z \in C^{\alpha}_t$ with $R \in C_2^{2
\alpha}$ and we do not require $w$ to admit a rough lift.

In order for Definition~\ref{sec5.3 defn controlled} to be meaningful, we need
to verify that the pair $(z', R)$ is unique; this follows from the condition
$\beta > \delta^{\ast}_{\gamma, \rho}$ and property~\eqref{sec5.3 basic
property}, which follows from Corollary~\ref{cor:rho-holder-roughness}.
Indeed, let $(\tilde{z}', \tilde{R})$ be another such pair and set $A = z' -
\tilde{z}'$, $B = R - \tilde{R}$. Choosing $\delta \in (0, 1)$ such that
$\delta^{\ast}_{\gamma, \rho} < \delta < \beta$, for any $s \in (0, T)$ and
any $v \in \mathbb{S}^{d - 1}$ it holds
\[ \limsup_{t \rightarrow s} \frac{| (A_s^T v) \cdot w_{s, t} |}{| t - s
   |^{\delta}} = \limsup_{t \rightarrow s} \frac{| B_{s, t} |}{| t - s
   |^{\delta}} \leqslant \| B \|_{\beta} \limsup_{t \rightarrow s} | t - s
   |^{\beta - \delta} = 0 \]
which implies by~\eqref{sec5.3 basic property} that $A^T_s v = 0$ for all $v
\in \mathbb{S}^{d - 1}$ and $s \in (0, T)$, thus $A \equiv 0$ and $B \equiv 0$
as well.

We will from now on assume in addition that there exists $c > 0$ such that
\begin{equation}
  z_s'  (z'_s)^T \geqslant c^2 I_d \quad \forall \, s \in [0, T] .
  \label{sec5.3 nondegeneracy}
\end{equation}
In particular, the above non-degeneracy condition implies that $z$ satisfies
property~\eqref{sec5.3 basic property} as well.

\begin{proposition}
  \label{sec5.3 prop perturbations}Let $w$ be $(\gamma,
  \rho)$\mbox{-}irregular, $z$ controlled by $w$ with $z'$
  satisfying~\eqref{sec5.3 nondegeneracy}. Then there exists $\tilde{\gamma} >
  \gamma$ such that $z$ is $(\tilde{\gamma}, \tilde{\rho})$\mbox{-}irregular and
  $\tilde{\rho}$ is given by
  \begin{equation*}
  	\tilde{\rho} = \frac{\beta}{1 - \gamma + \beta} \rho - \frac{1 -
     \gamma}{1 - \gamma + \beta} > 0;
  \end{equation*}
  moreover $\delta^{\ast}_{\gamma, \rho} = \delta^{\ast}_{\tilde{\gamma},
  \tilde{\rho}}$ and we have the estimate
  \[ \| \Phi^z \|_{\mathcal{W}^{\tilde{\gamma}, \tilde{\rho}}} \lesssim 1 + \| R
     \|_{\beta} + c^{- \rho} \| \Phi^w \|_{\mathcal{W}^{\gamma, \rho}}
     . \]
\end{proposition}

\begin{proof}
  Let $\xi \in \mathbb{R}^d$, $| \xi | \geqslant 1$ be fixed. For any $s < t$,
  it holds
  
  \begin{align*}
    \left| \int_s^t e^{i \xi \cdot z_r} \mathd r \right| & = \, \left|
    \int_s^t e^{i \xi \cdot z_{s, r}} \mathd r \right|\\
    & \leqslant \left| \int_s^t [e^{i \xi \cdot z_{s, r}} - e^{i \xi \cdot
    z'_s w_{s, r}}] \mathd r \right| + \left| \int_s^t e^{i \xi \cdot z'_s
    w_{s, r}} \mathd r \right|\\
    & \lesssim \, \int_s^t | \xi | | R_{s, r} | \mathd r + c^{- \rho} \|
    \Phi^w \|_{\mathcal{W}^{\gamma, \rho}} | \xi |^{- \rho} | t - s
    |^{\gamma}\\
    & \lesssim \, \| R \|_{\beta} | \xi | | t - s |^{1 + \beta} + c^{- \rho}
    \| \Phi^w \|_{\mathcal{W}^{\gamma, \rho}} | \xi |^{- \rho} | t - s
    |^{\gamma} .
  \end{align*}
  
  First assume that $| t - s |^{1 - \gamma + \beta} | \xi |^{1 + \rho}
  \leqslant 1$, so that $| \xi | | t - s |^{1 + \beta} \leqslant | \xi |^{-
  \rho} | t - s |^{\gamma}$, then in this case we trivially get
  \begin{equation}
    \left| \int_s^t e^{i \xi \cdot z_r} \mathd r \right| \lesssim (\| R
    \|_{\beta} + c^{- \rho} \| \Phi^w \|_{\mathcal{W}^{\gamma, \rho}}) | \xi
    |^{- \rho} | t - s |^{\gamma} . \label{sec5.3 eq1}
  \end{equation}
  Assume now that $| t - s |^{1 - \gamma + \beta} | \xi |^{1 + \rho} > 1$;
  choose $N \in \mathbb{N}$ such that $N^{1 - \gamma + \beta} \sim | t - s
  |^{1 - \gamma + \beta} | \xi |^{1 + \rho}$ and split the interval $[s, t]$
  in $N$ subintervals of size $| t - s | / N$. Applying the previous estimate
  to each of them and summing over we obtain
  \begin{equation}\label{eq:proof-pert}
  \begin{split}
    \left| \int_s^t e^{i \xi \cdot z_r} \mathd r \right| \lesssim & \| R
    \|_{\beta} N^{- \beta} | \xi | | t - s |^{1 + \beta} + c^{- \rho} \|
    \Phi^w \|_{\mathcal{W}^{\gamma, \rho}} N^{1 - \gamma} | t - s |^{\gamma} |
    \xi |^{- \rho}\\
    \sim & (\| R \|_{\beta} + c^{- \rho} \| \Phi^w \|_{\mathcal{W}^{\gamma,
    \rho}}) | t - s |  | \xi |^{- \tilde{\rho}}
  \end{split}\end{equation}
  where
  \[ \tilde{\rho} = \frac{\beta}{1 - \gamma + \beta} \rho - \frac{1 -
     \gamma}{1 - \gamma + \beta} = \lambda \rho + \lambda - 1 \]
  for suitable choice of $\lambda \in (0, 1)$, which implies that $\tilde\rho < \rho$; moreover $\tilde\rho>0$, since $\rho>(1-\lambda)/\lambda=(1-\gamma)/\beta$ thanks to the condition $\beta>\delta^\ast_{\gamma,\rho}$.
  
  Next consider $\theta\in (0,1)$ such that $\tilde{\rho} =\theta \rho$; when $|t-s|^{1-\gamma+\beta}|\xi|^{1+\rho}\leqslant 1$, interpolating between \eqref{sec5.3 eq1} and the trivial bound $|\int_s^t e^{i\xi\cdot z_r} \mathd r|\leqslant |t-s|$, we obtain
\begin{equation}\label{eq:proof-pert-eq2}
	\left| \int_s^t e^{i \xi \cdot z_r} \mathd r \right| \lesssim (\| R
    \|_{\beta} + c^{- \rho} \| \Phi^w \|_{\mathcal{W}^{\gamma, \rho}})^{\theta} | t - s |^{\tilde\gamma} | \xi|^{- \tilde\rho},
\end{equation}
where $\tilde\gamma=\theta \gamma+1-\theta$ is such that $\tilde\gamma>\gamma$ and $\delta_{\gamma,\rho}^\ast=\delta_{\tilde\gamma,\tilde\rho}^\ast$ (cf. Lemma~\ref{sec2.2 lemma interpolation rho irr}).
Combining the bounds \eqref{eq:proof-pert} and \eqref{eq:proof-pert-eq2}, using the simple bound $\max\{a^\theta,a\}\leqslant 1+a$, we arrive at
  \[ \left| \int_s^t e^{i \xi \cdot z_r} \mathd r \right| \lesssim_T (1 + \| R
     \|_{\beta} + c^{- \rho}\| \Phi^w \|_{\mathcal{W}^{\gamma, \rho}} )
     | \xi |^{- \tilde{\rho}} | t - s |^{\tilde{\gamma}} \quad \forall \, |
     \xi | \geqslant 1, \, s, t \in [0, T] \]
  without conditions on $| t - s |^{1 - \gamma + \beta} | \xi |^{1 + \rho}$,
  which gives the conclusion.
%  Now observe that since $\tilde{\rho} < \rho$, by Lemma~\ref{sec2.2 lemma
%  interpolation rho irr} we can always find $\tilde{\gamma} \in (\gamma, 1)$
%  such that $w$ is $(\tilde{\gamma}, \tilde{\rho})$\mbox{-}irregular and
%  $\delta^{\ast}_{\tilde{\gamma}, \tilde{\rho}} = \delta^{\ast}_{\gamma,
%  \rho}$ and $\| \Phi^w \|_{\mathcal{W}^{\tilde{\gamma}, \tilde{\rho}}}
%  \lesssim \| \Phi^w \|_{\mathcal{W}^{\gamma \comma \rho}}^{\theta} \lesssim 1
%  + \| \Phi^w \|_{\mathcal{W}^{\gamma, \rho}}$; estimate~\eqref{sec5.3 eq1}
%  can then be applied with $(\gamma, \rho, \| \Phi^w \|_{\mathcal{W}^{\gamma,
%  \rho}})$ replaced by $(\tilde{\gamma}, \tilde{\rho}, \| \Phi^w
%  \|_{\mathcal{W}^{\tilde{\gamma}, \tilde{\rho}}})$, since $| t - s |^{1 -
%  \gamma + \beta} | \xi |^{1 + \tilde{\rho}} \leqslant | t - s |^{1 - \gamma +
%  \beta} | \xi |^{1 + \rho} \leqslant 1$. Combining this estimate
%  with~\eqref{eq:proof-pert} then implies
%  \[ \left| \int_s^t e^{i \xi \cdot z_r} \mathd r \right| \lesssim (\| R
%     \|_{\beta} + c^{- \rho}) (1 + \| \Phi^w \|_{\mathcal{W}^{\gamma, \rho}})
%     | \xi |^{- \tilde{\rho}} | t - s |^{\tilde{\gamma}} \quad \forall \, |
%     \xi | \geqslant 1, \, s, t \in [0, T] \]
%  without conditions on $| t - s |^{1 - \gamma + \beta} | \xi |^{1 + \rho}$,
%  which gives the conclusion.
\end{proof}

\begin{remark}
  \label{sec5.3 remark examples}If for instance $w$ is $\rho$\mbox{-}irregular
  and $\beta \geqslant 1$, then we obtain that $z$ is
  $\tilde{\rho}$\mbox{-}irregular with
  \[ \tilde{\rho} \geqslant \frac{2}{3} \rho - \frac{1}{3} . \]
  Examples of $z$ satisfying the assumptions of Proposition~\ref{sec5.3 prop
  perturbations} are the following:
  \begin{itemize}
    \item Take $z_t = \varphi_t w_t$ with $\varphi \in C^{\beta} ([0, T] ;
    \mathbb{R})$ satisfying $\varphi_t \geqslant c > 0$, then $z'_s =
    \varphi_s I_d$, $R_{s, t} = w_t \varphi_{s, t} \in C^{\beta}_2$ (here
    $\beta \leqslant 1$).
    
    \item Suppose $w \in C^{\delta}_t$ with $\delta \in (0, 1)$ and take $z_t
    = \int_0^t A_s \mathd w_s$, where $A \in C^{\alpha} ([0, T] ;
    \mathbb{R}^{d \times d})$ satisfies~\eqref{sec5.3 nondegeneracy}, $\alpha
    + \delta > 1$ and the integral is defined in the Young sense. Then $z'_t =
    A_t$ and $\beta = \alpha + \delta > 1$.
    
    \item Finally, if $z = w + \varphi$ with $\varphi \in C^{\beta}_t$, then
    $z' \equiv I_d$ and $R_{s, t} = \varphi_{s, t} \in C_2^{\beta}$ (with
    $\beta \leqslant 1$); this case is however quite special and better
    estimates are available, see Section~\ref{sec5.3.1} below.
  \end{itemize}
  Let us highlight the difference between the purely analytical result of
  Proposition~\ref{sec5.3 prop perturbations} compared to the probabilistic
  result of Proposition~\ref{sec4.2 prop deterministic perturbations}, in
  which instead we have examples of Gaussian processes which are
  $\rho$\mbox{-}irregular with parameter $\rho$ invariant under any of the
  deterministic transformations from the list above.
\end{remark}

There is another notable class of transformations which preserve some
properties of the occupation measure $\mu^w$. In this case however it is
rather complicated to consider the $(\gamma, \rho)$\mbox{-}irregularity
property and it is instead more natural to reason with occupation densities.
Suppose that $w$ admits an occupation density $\ell_{s, t}^w$ (which we know
to be true by Lemma~\ref{sec2.2 lemma regularity averaging 2} if for instance
$w$ is $\rho$\mbox{-}irregular with $\rho > d / 2$) and let $F : \mathbb{R}^d
\rightarrow \mathbb{R}^d$ be a global diffeomorphism; define $z_t = F (w_t)$.
Then $z$ still admits an occupation density $\ell^z_{s, t}$, since
\[ \int_s^t \varphi (z_r) \mathd r = \int_s^t \varphi (F (w_r)) \mathd r =
   \int_{\mathbb{R}^d} \varphi (F (x)) \ell^w_{s, t} (x) \mathd x =
   \int_{\mathbb{R}^d} \varphi (x)  | \det (D F^{- 1} (x)) | \ell^w_{s, t}
   (F^{- 1} (x)) \mathd x \]
which shows that
\[ \mu^z_{s, t} (\mathd x) = | \det (D F^{- 1} (x)) | \ell^w_{s, t} (F^{- 1}
   (x)) \mathd x = \ell^z_{s, t} (x) \mathd x. \]
This also implies that $\ell^z_{s, t}$ inherits the regularity of $\ell^w_{s,
t}$ and $F$; for instance if $\ell^w \in C^{\gamma}_t L^2_x$, then

\begin{align*}
  \| \ell^z_{s, t} \|_{L^2}^2 & = \int_{\mathbb{R}^d} | \det (D F^{- 1} (x))
  |^2  | \ell^w_{s, t} (F^{- 1} (x)) |^2 \mathd x\\
  & = \int_{\mathbb{R}^d} | \det (D F (x)) |^{- 1}  | \ell^w_{s, t} (x) |^2
  \mathd x\\
  & \leqslant \| D F^{- 1} \|_{L^{\infty}}  \| \ell^w \|^2_{C^{\gamma} L^2} |
  t - s |^{2 \gamma} .
\end{align*}
Similar estimates hold if $\ell^w \in C^{\gamma}_t L^p_x$ or if $\ell^w \in
C^{\gamma}_t C^0_x$.

\subsection{The additive perturbation problem}\label{sec5.3.1}

In this section we treat for simplicity only the case $w \in C^{\delta}_t$
with $\delta \in (0, 1)$. In view of Theorem~\ref{sec3 thm2}, we will always
assume $\rho > 1 / 2$ (equivalently $(2 \rho)^{- 1} < 1$).

\

We first present a partial result, which is a slight improvement of
Theorem~1.6 from~{\cite{catelliergubinelli}}.

\begin{lemma}
  \label{sec5.3 lemma additive perturbation}Let $w$ be $(\gamma,
  \rho)$\mbox{-}irregular and $\varphi \in C_t^{\beta}$ for some $\beta \in
  (0, 1]$ such that $\beta > \delta^{\ast}_{\gamma, \rho}$. Then for any
  choice of $\delta \leqslant \beta$ satisfying $1 - \gamma < \delta < \beta
  \rho$, $w + \varphi$ is $(\tilde{\gamma}, \tilde{\rho})$\mbox{-}irregular
  for the choice
  \begin{equation}
    \tilde{\gamma} = \gamma \left( 1 - \frac{\delta}{\beta \rho} \right) +
    \frac{\delta}{\beta \rho}, \quad \tilde{\rho} = \rho -
    \frac{\delta}{\beta}, \label{sec5.3 parameters additive perturbation}
  \end{equation}
  and it holds
  \[ \tmcolor{black}{\| \Phi^{w + \varphi} \|_{\mathcal{W}^{\tilde{\gamma}
     \comma \tilde{\rho}}} \lesssim (1 + \| \Phi^w \|_{\mathcal{W}^{\gamma
     \comma \rho}}) (1 + \| \varphi \|^{\delta / \beta}_{C^{\beta}}) .} \]
  If $w$ is $\rho$-irregular and $\beta > \max \{ 1 / 2, (2 \rho)^{- 1} \}$,
  then $w + \varphi$ is $\left( \frac{1}{2} + \frac{1}{4 \beta \rho}, \rho -
  \frac{1}{2 \beta} \right)$\mbox{-}irregular with
  \[ \tmcolor{black}{\| \Phi^{w + \varphi} \|_{\mathcal{W}^{\tilde{\gamma}
     \comma \tilde{\rho}}} \lesssim (1 + \| \Phi^w \|_{\mathcal{W}^{\gamma
     \comma \rho}}) (1 + \| \varphi \|^{1 / 2 \beta}_{C^{\beta}}) .} \]
\end{lemma}

\tmcolor{black}{\begin{proof}
  Since $\varphi \in C^{\beta}_t$, so does $e^{i \xi \cdot \varphi}$, for all
  $\xi \in \mathbb{R}^d$. For any $\delta > 1 - \gamma$ we can then apply the
  estimates from Young integration as follows:
  
  \begin{align*}
    \left| \int_s^t e^{i \xi \cdot (w_r + \varphi_r)} \, \mathd r \right| & =
    \, \left| \left. \int_s^t e^{i \xi \cdot \varphi_r} \mathd \left( \int_s^r
    e^{i \xi \cdot w_u} \, \mathd \right. u \right) \right|\\
    & \lesssim_{\delta + \gamma} \, | e^{i \xi \cdot \varphi_s} | \left|
    \int_s^t e^{i \xi \cdot w_r} \, \mathd r \right| + | t - s |^{\gamma +
    \delta} \llbracket e^{i \xi \cdot \varphi} \rrbracket_{C^{\delta}}
    \left\llbracket \int_s^{\cdot} e^{i \xi \cdot w_r} \, \mathd r
    \right\rrbracket_{C^{\gamma}}\\
    & \leqslant \, \left| \int_s^t e^{i \xi \cdot w_r} \, \mathd r \right| +
    | t - s |^{\gamma + \delta} | \xi |^{- \rho}
    \tmcolor{blue}{\tmcolor{black}{\| \Phi^w \|_{\mathcal{W}^{\gamma \comma
    \rho}}}} \llbracket e^{i \xi \cdot \varphi} \rrbracket_{C^{\delta}} .
  \end{align*}
  
  Now since $\varphi \in C^{\beta}_t$, by interpolation we have (we are using
  the hypothesis $\delta \leqslant \beta$)
  \[ | e^{i \xi \cdot \varphi_t} - e^{i \xi \cdot \varphi_s} | \leqslant 2,
     \hspace{0.4em} | e^{i \xi \cdot \varphi_t} - e^{i \xi \cdot \varphi_s} |
     \leqslant | \xi | | t - s |^{\beta} \llbracket \varphi
     \rrbracket_{C^{\beta}} \enspace \Rightarrow \enspace | e^{i \xi \cdot
     \varphi_t} - e^{i \xi \cdot \varphi_s} | \lesssim \| \varphi \|^{\delta /
     \beta}_{C^{\beta}} | \xi |^{\delta / \beta} | t - s |^{\delta} ; \]
  similarly for any $\theta \in (0, 1)$ it holds
  \[ \left| \int_s^t e^{i \xi \cdot w_r} \, \mathd r \right| \lesssim
     \tmcolor{blue}{\tmcolor{black}{\| \Phi^w \|^{\theta}_{\mathcal{W}^{\gamma
     \comma \rho}}}} | t - s |^{\gamma \theta + 1 - \theta} | \xi |^{- \theta
     \rho} \]
  and so putting everything together we obtain
  \[ \left| \int_s^t e^{i \xi \cdot (w_r + \varphi_r)} \, \mathd r \right|
     \lesssim \tmcolor{black}{\| \Phi^w \|^{\theta}_{\mathcal{W}^{\gamma
     \comma \rho}}} | t - s |^{\gamma \theta + 1 - \theta} | \xi |^{- \theta
     \rho} + | t - s |^{\gamma + \delta} | \xi |^{- \rho + \delta / \beta}
     \tmcolor{blue}{\tmcolor{black}{\| \Phi^w \|_{\mathcal{W}^{\gamma \comma
     \rho}}}} \| \varphi \|^{\delta / \beta}_{C^{\beta}} . \]
  Choosing $\theta \in (0, 1)$ such that $\theta \rho = \rho - \delta /
  \beta$, namely $\theta = 1 - \delta / (\beta \rho)$ \ we obtain the first
  statement. The second one simply follows from the assumption $\gamma > 1 /
  2$, taking $\delta = 1 / 2$.
\end{proof}}

The partial result above implies that, even if we consider a perturbation
$\varphi \in C^1_t$, we should expect a loss in spatial regularity of order $1
/ 2$, which is only partially recovered by an improvement in time regularity
of order $1 / 4 \rho$. The new parameters $(\tilde{\gamma}, \tilde{\rho})$
given by Lemma~\ref{sec5.3 parameters additive perturbation} satisfy
$\delta^{\ast}_{\gamma, \rho} = \delta^{\ast}_{\tilde{\gamma}, \tilde{\rho}}$,
which implies that $w + \varphi$ still satisfies property~\eqref{sec5.3 basic
property}, as can be checked directly using the fact that $\varphi \in
C^{\beta}_t$ for some $\beta > \delta^{\ast}_{\gamma, \rho}$. This hints that
the above result, while not being fully satisfactory, might be optimal, even
if we cannot exclude the existence of other pairs $(\gamma', \rho')$ with
$\rho' > \tilde{\rho}$ such that $w + \varphi$ is $(\gamma',
\rho')$\mbox{-}irregular.

\

The proof above cannot provide better results in the case $\varphi \in
C^{\beta}_t$ with $\beta > 1$. Even if it were false in general that $w +
\varphi$ is $(\gamma, \rho)$-irregular whenever $w$ is so and $\varphi \in
C^{\beta}_t$ with $\beta > \delta^{\ast}_{\gamma, \rho}$, we would at least
expect the claim to be true whenever $\varphi$ is $C^{\infty}_t$; this is a
conjecture left open in~{\cite{catelliergubinelli}}.

We can give it a positive answer, up to strengthening the notion of
$\rho$-irregularity. Before giving the rigorous statement, let us give an
intuition by considering the following case. Suppose that $\varphi \in C^{1 +
\beta}_t$ for some $\beta \in [0, 1]$ and suppose that $w$ satisfies the
following property: for any $a \in \mathbb{R}^d$, $t \mapsto w_t + a \, t$ is
$\rho$-irregular, uniformly in $a$, in the sense that $\sup_a \left\|
\tmcolor{blue}{\tmcolor{black}{\Phi^{w + a \, t}}}
\right\|_{\mathcal{W}^{\gamma \comma \rho}} < \infty$. Then we can improve the
previous estimates as follows:

\begin{align*}
  \left| \int_s^t e^{i \xi \cdot (w_r + \varphi_r)} \, \mathd r \right| & =
  \left| e^{- i \xi \cdot (\varphi_s - s \varphi'_s)} \int^t_s e^{i \xi \cdot
  (w_r + \varphi_r)} \mathd r \right|\\
  & = \left| \int_s^t e^{i \xi \cdot (\varphi_{s, r} - \varphi'_s (r - s))}
  \mathd \left( \int_s^r e^{i \xi \cdot (w_u + \varphi'_s u)} \, \mathd u
  \right) \right|\\
  & \lesssim \| \Phi^{w + \varphi'_s t} \|_{\mathcal{W}^{\gamma, \rho}} | t -
  s |^{\gamma} | \xi |^{- \rho} + \| \Phi^{w + \varphi'_s t}
  \|_{\mathcal{W}^{\gamma, \rho}} | t - s | \, \| e^{i \xi \cdot (\varphi_{s,
  \cdot} - \varphi'_s (\cdot - s))} \|_{C^{1 / 2}} | \xi |^{- \rho}\\
  & \lesssim_w | t - s |^{\gamma} | \xi |^{- \rho} + | t - s | \, | \xi |^{-
  \rho} \| e^{i \xi \cdot (\varphi_{s, \cdot} - \varphi'_s (\cdot - s))}
  \|_{C^{1 / 2}}
\end{align*}
where the last norm is taken over the interval $[s, t]$. As before, we can
estimate it using simple interpolation arguments, only this time we have
\[ \begin{array}{lll}
     | \varphi_{s, u} - \varphi'_s (u - s) - \varphi_{s, v} + \varphi'_s (v -
     s) | & = & | \varphi_{u, v} - \varphi_s' (v - u) |\\
     & = & \left| \int_v^u \varphi'_{r, s} \, \mathd r \right|\\
     & \leqslant & \int_v^u \llbracket \varphi' \rrbracket_{C^{\beta}} | r -
     s |^{\beta} \, \mathd r\\
     & \lesssim_{\varphi} & | u - v |^{1 / 2} | t - s |^{1 / 2 + \beta}
   \end{array} \]
where we used the fact that $[v, u] \subset [s, t]$. Therefore we obtain
\[ \left| \int_s^t e^{i \xi \cdot (w_r + \varphi_r)} \, \mathd r \right|
   \lesssim | t - s |^{\gamma} | \xi |^{- \rho} + | t - s |^{3 / 2 + \beta} |
   \xi |^{1 - \rho} \]
and we can now reason as in the proof of Lemma~\ref{sec5.3 prop
perturbations}, i.e. split the interval $[s, t]$ into $N$ subintervals of size
$| t - s | / N$, apply the estimate on such intervals, sum over $N$ and choose
$N \sim | \xi |^{1 / (1 + \beta)}$ to obtain
\[ \left| \int_s^t e^{i \xi \cdot (w_r + \varphi_r)} \, \mathd r \right|
   \lesssim | t - s |^{\gamma} | \xi |^{- \rho + 1 / 2 (1 + \beta)} . \]
This shows that $w + \varphi$ is $(\rho - (2 + 2 \beta)^{-
1})$\mbox{-}irregular. In particular, even if we are not able to recover
$\rho$\mbox{-}irregularity, the loss of regularity for $\varphi \in C^{1 +
\beta}$ is now expected to be $(2 + 2 \beta)^{- 1}$, which suggests that more
generally for $\varphi \in C^{\beta}$, $w + \varphi$ should be $(\rho - (2
\beta)^{- 1})$\mbox{-}irregular, for any $\beta \in [1 / 2, + \infty)$.

\

This motivates the following definition; here $F (\xi) = | \xi |^{\rho} /
\sqrt{\log | \xi |}$, $\psi (x) = \sqrt{x | \log x |}$.

\begin{definition}
  \label{sec5.3 defn strong rho}We say that $w \in C^0_t$ is strongly
  $\rho$\mbox{-}irregular if the following holds: for any $n \in \mathbb{N}$,
  given $\eta \in \mathbb{R}^n$ and denoting by $g^{\eta}_r \assign \sum_{k =
  1}^n \eta_k \, r^k$, then
  \begin{equation}
    \sup_{\xi \in \mathbb{R}^d, \eta \in \mathbb{R}^n, s \neq t} \frac{\left|
    \int_s^t e^{i \xi \cdot w_r + i g^{\eta}_r} \, \mathd r \right| F
    (\xi)}{\sqrt{\log (1 + | \eta |)} \psi (| t - s |)} < \infty .
  \end{equation}
\end{definition}

The notion formalises the idea that the irregularity of $w$ should be only
mildly affected by polynomial perturbations of any degree; this allows to
proceed as above by locally expanding a more general additive perturbation
$\varphi$ in its Taylor series, centred at $s$.

\begin{theorem}
  \label{sec5.3 thm perturbation strong}Let $w$ be strongly $\rho$-irregular.
  Then for any $\varphi \in C^{\beta}_t$, $\beta \in [1 / 2, \infty)$ and for
  any $\tilde{\rho} < \rho$, $w + \varphi$ is $(\tilde{\rho} - 1 / 2
  \beta)$-irregular. In particular, if $\varphi \in C_t^{\infty}$, then $w +
  \varphi$ is $\tilde{\rho}$-irregular for any $\tilde{\rho} < \rho$.
\end{theorem}

\tmcolor{blue}{\tmcolor{black}{\begin{proof}
  Let $\xi \in \mathbb{R}^d$, $0 \leqslant s < t \leqslant T$ be fixed. We
  write $\varphi$ as its Taylor series of order $\lfloor \beta \rfloor$
  centred in $s$ plus a reminder term:
  \[ \varphi_r = p_r + R_r = \sum_{k = 0}^{\lfloor \beta \rfloor} D^{(k)}
     \varphi_s \frac{(r - s)^k}{k!} + R_r  \quad \forall \, r \in [s, t] . \]
  Observe that all the terms in $p_r$ depending exclusively on $s$ do not play
  any role when estimating $\left| \int_s^t e^{i \xi \cdot (w_r + p_r)} \,
  \mathd r \right|$, since they just go outside the integral and provide a
  term with modulus $1$; therefore we will systematically drop them in the
  calculations. Then
  \[ \xi \cdot p_r = \sum_{k = 1}^{\lfloor \beta \rfloor} \xi \cdot D^{(k)}
     \varphi_s \frac{(r - s)^k}{k!} = \sum_{k = 1}^{\lfloor \beta \rfloor} \xi
     \cdot D^{(k)} \varphi_s \sum_{j = 1}^k \frac{r^j (- s)^{k - j}}{j! (k -
     j) !} = \sum_{j = 1}^{\lfloor \beta \rfloor} \left( \sum^{\lfloor \beta
     \rfloor}_{k = j} \xi \cdot D^{(k)} \varphi_s \frac{(- s)^{k - j}}{j! (k -
     j) !} \right) r^j \]
  which implies by the definition of strong $\rho$-irregularity for $n =
  \lfloor \beta \rfloor$ that
  \[ \left| \int_s^t e^{i \xi \cdot (w_r + p_r)} \mathd r \right| \lesssim_{n,
     T} | \xi |^{- \rho} \, \sqrt{\log | \xi |} \sqrt{\log (1 + | \xi | \|
     \varphi \|_{C^{\beta}})} \psi (| t - s |) . \]
  By the usual interpolation arguments, we deduce that for any $\tilde{\rho} <
  \rho$ there exists $\gamma = \gamma (\tilde{\rho}) > 1 / 2$ such that $w_r +
  p_r$ is $\tilde{\rho}$\mbox{-}irregular; from now on we work with such fixed
  pair $(\gamma, \tilde{\rho})$.
  
  On the other hand, $\| R \|_{C ([s, t] ; \mathbb{R}^d)} \lesssim \| \varphi
  \|_{C^{\beta}} | t - s |^{\beta}$ and that $\| R^{'} \|_{C ([s, t] ;
  \mathbb{R}^d)} \lesssim \| \varphi \|_{C^{\beta}} | t - s |^{\beta - 1}$,
  where $R'$ is the derivative of $R$; by interpolation this gives $\| R
  \|_{C^{1 / 2} ([s, t] ; \mathbb{R}^d)} \lesssim \| \varphi \|_{C^{\beta}} |
  t - s |^{\beta - 1 / 2}$. Therefore we have
  
  \begin{align*}
    \left| \int_s^t e^{i \xi \cdot (w_r + \varphi_r)} \, \mathd r \right| & =
    \, \left| \int_s^t e^{i \xi \cdot R_r} \mathd \left( \int_s^r e^{i \xi
    \cdot (p_u + w_u)} \, \mathd u \right) \right|\\
    & \lesssim \, \| \Phi^{p + w} \|_{\mathcal{W}^{\gamma, \tilde{\rho}}} |
    \xi |^{- \tilde{\rho}} (| t - s |^{\gamma} + \| e^{i \xi \cdot R} \|_{C^{1
    / 2}} | t - s |)\\
    & \lesssim_{\varphi} \, | t - s |^{\gamma} | \xi |^{- \tilde{\rho}} + | t
    - s |^{\beta + 1 / 2} | \xi |^{1 - \tilde{\rho}} .
  \end{align*}
  Splitting the integral $[s, t]$ in $N$ subintervals of length $| t - s | /
  N$, applying the estimate on such subintervals and then summing over $N$ we
  obtain
  \[ \left| \int_s^t e^{i \xi \cdot (w_r + \varphi_r)} \, \mathd r \right|
     \lesssim N^{1 / 2} | t - s |^{\gamma} | \xi |^{- \tilde{\rho}} + N^{-
     \beta + 1 / 2} | t - s |^{\beta + 1 / 2} | \xi |^{1 - \tilde{\rho}} . \]
  Choosing $N$ such that $N \sim | \xi |^{1 / \beta}$ allows to conclude.
\end{proof}}}

\begin{theorem}
  \label{sec5.3 thm prevalence strong rho}For any $\delta \in (0, 1)$, almost
  every $\varphi \in C^{\delta}_t$ is strongly $\rho$-irregular for any $\rho
  < (2 \delta)^{- 1}$. Almost every $\varphi \in C^0_t$ is strongly
  $\rho$-irregular for any $\rho < \infty$.
\end{theorem}

\begin{proof}
  We only briefly sketch the proof as it mostly relies on the same techniques
  used to prove prevalence of $\rho$-irregularity and Theorem~\ref{sec4.1 thm
  general criterion}; it is easy to guess the our candidate transverse
  measures are the laws $\mu^H$ of fBm with parameter $H > \delta$. The proof
  that the strong $\rho$\mbox{-}irregularity property defines Borel sets in
  $C^{\delta}$, $\delta \in [0, 1)$, is identical to the one of
  Lemma~\ref{sec2.2 lemma borel rho irr}, so we will omit it. Now fix $\varphi
  \in C^{\delta}_t$; we need to show that
  \[ \mu^H \left( w \in C^{\delta}_t : \varphi + w \text{ is strongly
     $\rho$\mbox{-}irregular for any } \rho < (2 H)^{- 1} \right) = 1. \]
  Once this is proven, the conclusion follows from the usual argument
  regarding countable intersection of prevalent sets. We switch to the
  probabilistic notation, $W^H$ being an fBm with law $\mu^H$.
  
  Fix $n \in \mathbb{N}$. In the following, all estimates depend on $n$.
  Using the same technique as in Lemma~\ref{sec4.1 lemma condition rho irr}
  and Theorem~\ref{sec4.2 corollary rho gaussian}, it can be shown that there
  exists $\overline{\lambda}$ such that
  \[ \sup_{\xi, s, t} \, \mathbb{E} \left[ \exp \left( \overline{\lambda} 
     \frac{\left| \int_s^t e^{i \xi \cdot (\varphi_r + W^H_r) + i g^{\eta}_r}
     \, \mathd r \right|^2 | \xi |^{1 / H}}{| t - s |} \right) \right]
     \leqslant K \quad \text{uniformly in $\eta \in \mathbb{R}^n$.} \]
  Let us fix $s < t$ and define
  
  \begin{align}
    J_{s, t} (\lambda) \assign & \sum_{N \in \mathbb{N}} 2^{- N} \sum_{\xi \in
    2^{- N} \mathbb{Z}^d, \eta \in 2^{- N} \mathbb{Z}^n} 2^{- N (d + n + 2)}
    (1 + | \xi |)^{- (d + 1)} (1 + | \eta |)^{- (n + 1)} \times \nonumber\\
    & \qquad \times \exp \left( \lambda \frac{\left| \int_s^t e^{i \xi \cdot
    (\varphi_r + W^H_r) + i g^{\eta}_r} \, \mathd r \right|^2 | \xi |^{1 /
    H}}{| t - s |} \right) . \nonumber
  \end{align}
  It is clear that $\mathbb{E} [J_{s, t} (\lambda)] \lesssim K$ for all
  $\lambda \leqslant \overline{\lambda}$ and that by Jensen inequality $J_{s,
  t} (\lambda)^{\beta} \lesssim J_{s, t} (\beta \lambda)$ for all $\beta
  \geqslant 1$. Let us also define
  \[ Y_{s, t} = \sup_{\xi \in \mathbb{R}^d, \eta \in \mathbb{R}^n}
     \frac{\left| \int_s^t e^{i \xi \cdot (\varphi_r + W^H_r) + i g^{\eta}_r}
     \, \mathd r \right| | \xi |^{\frac{1}{2 H}}}{\sqrt{\log (1 + | \eta |)} 
     \sqrt{\log (1 + | \xi |)} | t - s |^{1 / 2}}  . \]
  In order to prove the theorem it suffices to show that there exists $\lambda
  > 0$ such that $\mathbb{E} [e^{\lambda Y_{s, t}^2}] \leqslant K$, since in
  that case we are in the conditions to apply Lemma~\ref{appendixA2 chaining
  lemma}. Let $(\xi, \eta)$ be fixed, then for any $N$ there exists
  $(\tilde{\xi}, \tilde{\eta}) \in (2^{- N} \mathbb{Z}^d, 2^{-N} \mathbb{Z}^n)$ such that $| (\xi, \eta) -
  (\tilde{\xi}, \tilde{\eta}) | \lesssim 2^{- N}$ and we have
  \[ \left| \int_s^t (e^{i \xi \cdot (\varphi_r + W^H_r) + i g^{\eta}_r} -
     e^{i \tilde{\xi} \cdot (\varphi_r + W^H_r) + i g^{\tilde{\eta}}_r})
     \mathd r \right| \lesssim | t - s |^{1 / 2} (1 + \| \varphi + W^H \|^{1 /
     2}_{L^1}) | (\xi, \eta) - (\tilde{\xi}, \tilde{\eta}) |^{1 / 2} . \]
  On the other hand, for $(\tilde{\xi}, \tilde{\eta})$ it holds
  \[ \left| \int_s^t e^{i \tilde{\xi} \cdot (\varphi_r + W^H_r) + i
     g^{\tilde{\eta}}_r} \, \mathd r \right| \lesssim \lambda^{- 1 / 2} |
     \tilde{\xi} |^{- 1 / 2 H} | t - s |^{1 / 2} \sqrt{\log \, J_{s, t}
     (\lambda) + N + \log (1 + | \tilde{\xi} |) + \log (1 + | \tilde{\eta} |)}
  \]
  and therefore, putting everything together, after some computations we
  obtain
  
  \begin{align*}
    \left| \int_s^t e^{i \xi \cdot (\varphi_r + W^H_r) + i g^{\eta}_r} \,
    \mathd r \right| \, | \xi |^{1 / 2 H} & \lesssim | t - s |^{1 / 2} 2^{- N
    / 2} (1 + \| \varphi + W^H \|^{1 / 2}_{L^1}) | \xi |^{1 / 2 H}\\
    & + | t - s |^{1 / 2} \lambda^{- 1 / 2} \sqrt{\log \, J_{s, t} (\lambda)
    + N + \log (1 + | \eta |) + \log (1 + | \xi |) + c} .
  \end{align*}
  
  Choosing $N$ such that $2^{- N / 2} \sim | \xi |^{1 / 2 H}$, dividing by
  $\sqrt{\log (1 + | \eta |)}  \sqrt{\log (1 + | \xi |)} |t-s|^{1/2}$ and taking the
  supremum we get
  \[ Y_{s, t} \lesssim 1 + \| \varphi + W^H \|^{1 / 2}_{L^1} + \lambda^{- 1 /
     2} + \lambda^{- 1 / 2} \sqrt{\log \, J_{s, t} (\lambda)} \]
  and so there exists a constant $C$ such that
  \[ e^{\lambda Y_{s, t}^2} \lesssim e^{\lambda C (1 + \| \varphi + W^H
     \|_{L^1})} \, J_{s, t} (\lambda)^C \lesssim e^{2 \lambda C (1 + \|
     \varphi + W^H \|_{L^1})} + J_{s, t} (\lambda)^{2 C} \lesssim e^{2 \lambda
     C (1 + \| \varphi + W^H \|_{L^1})} + J_{s, t} (2 C \lambda) . \]
  Invoking Fernique theorem and choosing $\lambda$ such that $2 C \lambda <
  \overline{\lambda}$ we obtain the result for fixed $n$. Since $n$ is
  arbitrary, \ we obtain the conclusion.
\end{proof}

\begin{remark}
  The same proof shows that any $\beta$\mbox{-}SLND Gaussian process with
  suitable integrability conditions is strongly $(2 \beta)^{-
  1}$\mbox{-}irregular with probability~$1$.
\end{remark}

\begin{proof}[Proof of Theorem~\ref{sec3 thm5}]
  Follows from a combination of Theorems~\ref{sec5.3 thm perturbation strong}
  and~\ref{sec5.3 thm prevalence strong rho}.
\end{proof}

\appendix\section{Auxiliary results}\label{appendix}

\subsection{A primer on Young integration theory}\label{appendixA1}

Estimates involving Young integrals are frequently used throughout this work;
for this reason we present here a brief account on the topic for the
interested reader. Young integrals go back to~{\cite{young}}, where a quantity
of the form
\[ \int_0^t f_s \mathd g_s \]
is defined in terms of the limit of Riemann--Stieltjes sums, under the
condition that $f$ and $g$ are respectively of finite $p$\mbox{-} and
$q$\mbox{-}variation with $1 / p + 1 / q > 1$; for a modern account on the
theory see for instance Section~6 of~{\cite{frizvictoir}}.

Here we restrict ourselves to $f$ and $g$ belonging to suitable H\"{o}lder
spaces and we follow the modern approach of constructing Young integrals by
means of the Sewing lemma.

Consider an interval $[0, T]$ and a Banach space $E$; let $\Delta_n$ denote
the $n$\mbox{-}simplex on $[0, T]$, so that $\Delta_n = \{ (t_1, \ldots, t_n)
: 0 \leqslant t_1 \leqslant \ldots \leqslant t_n \leqslant T \}$. Given a map
$\Gamma : \Delta_2 \rightarrow E$, we define $\delta \Gamma : \Delta_3
\rightarrow E$ by
\[ \delta \Gamma_{s, u, t} : = \, \Gamma_{s, t} - \Gamma_{s, u} - \Gamma_{u,
   t} . \]
We say that $\Gamma \in C^{\alpha, \beta}_2 ([0, T] ; E)$ if $\Gamma_{t, t} =
0$ for all $t \in [0, T]$ and $\| \Gamma \|_{\alpha, \beta} < \infty$, where
\[ \| \Gamma \|_{\alpha} \assign \sup_{s < t} \frac{\| \Gamma_{s, t} \|_E}{| t
   - s |^{\alpha}}, \quad \left\| \delta \, \Gamma \right\|_{\beta} \assign
   \sup_{s < u < t} \frac{\left\| \delta \, \Gamma_{s, u, t} \right\|_E}{| t -
   s |^{\beta}}, \quad \| \Gamma \|_{\alpha, \beta} \assign \| \Gamma
   \|_{\alpha} + \left\| \delta \, \Gamma \right\|_{\beta} . \]
Let us remark that for a map $f : [0, T] \rightarrow E$, we still denote by
$f_{s, t}$ the increment $f_t - f_s$.

\begin{lemma}[Sewing lemma]
  \label{appendixA1 sewing lemma}Let $\alpha$, $\beta$ be such that $0 <
  \alpha \leqslant 1 < \beta$. For any $\Gamma \in C^{\alpha, \beta}_2 ([0, T]
  ; E)$ there exists a unique map $\mathcal{I} \, \Gamma \in C^{\alpha} ([0,
  T] ; E)$ such that $\left( \mathcal{I} \, \Gamma \right)_0 = 0$ and
  \begin{equation}
    \left\| \left( \mathcal{I} \, \Gamma \right)_{s, t} - \Gamma_{s, t}
    \right\|_E \leqslant C \, \| \delta \Gamma \|_{\beta}  | t - s |^{\beta}
    \label{appendixA1 sewing property 1}
  \end{equation}
  where the constant $C$ only depends on $\beta$. In particular, the map
  $\mathcal{I} : C^{\alpha, \beta}_2 \rightarrow C^{\alpha}$ is linear and
  bounded and there exists a constant $C'$ which only depends on $\beta$ and
  $T$ such that
  \begin{equation}
    \left\| \mathcal{I} \, \Gamma \right\|_{C^{\alpha}} \leqslant C'  \|
    \Gamma \|_{\alpha, \beta} . \label{appendixA1 sewing property 2}
  \end{equation}
  For given $\Gamma$, the map $\mathcal{I} \, \Gamma$ is characterised as the
  unique limit of Riemann--Stieltjes sums: for any $t > 0$
  \[ \left( \mathcal{I} \, \Gamma \right)_t = \lim_{| \Pi | \rightarrow 0}
     \sum_i \Gamma_{t_i, t_{i + 1}} . \]
  The notation above means that for any sequence of partitions $\Pi_n = \{ 0 =
  t_0 < t_1 < \ldots < t_{k_n} = t \}$ with mesh $| \Pi_n | = \sup_{i = 1,
  \ldots, k_n} | t_i - t_{i - 1} | \rightarrow 0$ as $n \rightarrow \infty$,
  it holds
  \[ \left( \mathcal{I} \, \Gamma \right)_t = \lim_{n \rightarrow \infty}
     \sum_{i = 0}^{k_n - 1} \Gamma_{t_i, t_{i + 1}} . \]
\end{lemma}

For a proof, see Lemma~4.2 from~{\cite{frizhairer}}. With this tool at hand,
it is possible to define Young integrals in a variety of situations; we start
by the most general one.

Given two Banach spaces $E$ and $F$, we denote by $\mathcal{L} (E ; F)$ the
set of all bounded linear operators from $E$ to $F$, which is a Banach space
with the norm
\[ \| A \|_{\mathcal{L} (E ; F)} = \sup_{\varphi \in E \setminus \{ 0 \}}
   \frac{\| A \varphi \|_F}{\| \varphi \|_E} . \]

\begin{corollary}[Young integral]
  \label{appendixA1 corollary young}Let $A \in C^{\alpha} ([0, T] ;
  \mathcal{L} (E ; F))$ and $\varphi \in C^{\beta} ([0, T] ; E)$ such that
  $\alpha, \beta \in [0, 1]$ and $\alpha + \beta > 1$. Then for any $t \in [0,
  T]$, the limit in $F$ of the following Riemann--Stieltjes sums exists
  and is unique:
  \[ \int_0^t A_{\mathd s} \varphi_s \assign \lim_{| \Pi | \rightarrow 0}
     \sum_i A_{t_i, t_{i + 1}} \varphi_{t_i} . \]
  Moreover there exists a constant $C$ which only depends on $\alpha + \beta$
  such that
  \begin{equation}
    \left\| \int_s^t A_{\mathd r} \varphi_r - A_{s, t} \, \varphi_s \right\|_F
    \leqslant C \llbracket A \rrbracket_{C^{\alpha} \mathcal{L}}  \llbracket
    \varphi \rrbracket_{C^{\beta} E} . \label{appendixA1 young inequality}
  \end{equation}
  If $A \in C^1 ([0, T] ; \mathcal{L} (E, F))$, then
  \[ \int_s^t A_{\mathd r} \varphi_r = \int_s^t A'_r \varphi_r \mathd r \quad
     \text{for all } 0 \leqslant s < t \leqslant T. \]
\end{corollary}

\begin{proof}
  Define the map $\Gamma : \Delta_2 \rightarrow F$ by setting $\Gamma_{s, t}
  \assign A_{s, t} \, \varphi_s$; by definition $\Gamma_{t, t} = 0$ and
  moreover
  \begin{eqnarray*}
    \| \Gamma_{s, t} \|_F & \leqslant & \| A_{s, t} \|_{\mathcal{L}} \|
    \varphi_s \|_E \leqslant | t - s |^{\alpha} \llbracket A
    \rrbracket_{C^{\alpha} \mathcal{L}}  \| \varphi \|_{C^{\beta} E}\\
    \left\| \delta \, \Gamma_{s, u, t} \right\|_F & = & \| A_{u, t}
    \varphi_{s, u} \|_F \leqslant \| A_{u, t} \|_{\mathcal{L}} \| \varphi_{s,
    u} \|_E \leqslant \llbracket A \rrbracket_{C^{\alpha} \mathcal{L}} \,
    \llbracket \varphi \rrbracket_{C^{\beta} E} | t - s |^{\alpha + \beta}
  \end{eqnarray*}
  which implies that we can we can apply the Sewing lemma for such choice of
  $\Gamma$; inequality~\eqref{appendixA1 young inequality} is then an
  immediate consequence of the above estimates combined with~\eqref{appendixA1
  sewing property 1}. The last statement follows from the fact that if $A \in
  C^1 ([0, T] ; \mathcal{L} (E, F))$, then by standard theory for any $t > 0$
  it holds
  \[ \int_0^t A'_r \varphi_r \mathd r = \lim_{| \Pi | \rightarrow 0} \sum_i
     A_{t_i, t_{i + 1}} \varphi_{t_i} . \]
\end{proof}

Let us stress that in the above statements whether $E$ and $F$ are finite or
infinite dimensional does not play any role and that the constants $C$ and
$C'$ do not depend on them.

The most basic definition of Young integral is in the case $f \in C^{\alpha}
([0, T] ; \mathbb{R})$ and $g \in C^{\beta} ([0, T], \mathbb{R})$ with $\alpha
+ \beta > 1$, in which case for the choice $E = F =\mathbb{R}$ and the
identification $\mathbb{R}= \mathcal{L} (\mathbb{R}; \mathbb{R})$ we can
define both
\[ \int_0^{\cdot} f_s \mathd g_s \in C^{\beta} ([0, T] ; \mathbb{R}), \quad
   \int_0^{\cdot} g_s \mathd f_s \in C^{\alpha} ([0, T] ; \mathbb{R}) . \]
In the case $f \in C^{\alpha} ([0, T] ; \mathbb{R}^{m \times n})$ and $g \in
C^{\beta} ([0, T] ; \mathbb{R}^n)$, both above integrals can be defined, this
time being $\mathbb{R}^m$\mbox{-}valued functions, either by reasoning
component\mbox{-}by\mbox{-}component or using identifications between dual
spaces ($v \in \mathbb{R}^n$ can be identified with the map $A \mapsto A v$
which is an element of $\mathcal{L} (\mathbb{R}^{m \times n} ; \mathbb{R}^m)$.

Another important case is the following: let $E$, $F$ and $G$ be Banach spaces
and let $A : E \times F \rightarrow G$ be a bilinear bounded map, i.e. such
that
\[ \| A \|_{\mathcal{L}^2 (E \times F ; G)} = \sup_{\tmscript{\begin{array}{l}
     v \in E \setminus \{ 0 \}\\
     w \in F \setminus \{ 0 \}
   \end{array}}} \frac{\| A (v, w) \|_G}{\| v \|_E \| w \|_F} < \infty . \]
Then given $\varphi \in C^{\alpha} ([0, T] ; E)$ and $\psi \in C^{\beta} ([0,
T] ; F)$ we can define respectively
\[ \int_0^{\cdot} A (\varphi_s, \mathd \psi_s) \in C^{\beta} ([0, T] ; G),
   \quad \int_0^{\cdot} A (\mathd \varphi_s, \psi_s) \in C^{\alpha} ([0, T] ;
   G) ; \]
this follows from the fact that the map $A$ allows to define an embedding $E
\hookrightarrow \mathcal{L} (F ; G)$ by $v \mapsto A (v, \cdot)$, similarly
for $F \hookrightarrow \mathcal{L} (E ; G)$, and that H\"{o}lder continuity of
paths are preserved by these maps.

\subsection{Some useful tools}\label{appendixA2}

We give here a chaining lemma which was frequently used in Section~\ref{sec4}.
It is a slight variation on Lemma~3.1 from~{\cite{catelliergubinelli}}.

\begin{lemma}
  \label{appendixA2 chaining lemma}Let $E$ be a Banach space and let $X : [0,
  T] \rightarrow E$ be a continuous stochastic process such that, for some
  $\lambda > 0$,
  \[ \mathbb{E} \left[ \exp \left( \lambda \frac{\| X_t - X_s \|^2_E}{| t - s
     |^{2 \alpha}} \right) \right] \leqslant C \quad \forall \, s \neq t \in
     [0, T] . \]
  Then $\mathbb{P}$-a.s. $X \in C^{\varphi} ([0, T] ; E)$ for the modulus of
  continuity $\varphi (x) = x^{\alpha} \sqrt{| \log x |}$ and there exists
  $\beta > 0$ such that
  \[ \mathbb{E} [\exp (\beta \llbracket X \rrbracket^2_{C^{\varphi} E})] <
     \infty . \]
  In particular, if $X_0 \equiv 0$, then for any $\gamma < \alpha$ there
  exists $\beta > 0$ such that
  \[ \mathbb{E} [\exp (\beta \| X \|^2_{C^{\gamma} E})] < \infty . \]
\end{lemma}

\begin{proof}
  Without loss of generality we can assume $T = 1$. Also, we will only show
  that proof in the case $\alpha = 1 / 2$, the other cases being entirely
  analogue. Let us define the random variable
  \[ R (\lambda) = \sum_{n \in \mathbb{N}} \sum_{k = 0}^{2^n - 1} 2^{- 2 n} \,
     \exp \left( \mu \frac{\| X_{(k + 1) 2^{- n}} - X_{k 2^{- n}} \|^2_E}{2^{-
     n}} \right) . \]
  Then it follows from the assumption that $\mathbb{E} [R (\lambda)] \leqslant
  C$. We can then apply Lemma~3.1 from~{\cite{catelliergubinelli}} to deduce
  that there exist deterministic positive constants $K, \beta$ such that
  \[ \exp \left( \beta \frac{\| X_t - X_s \|^2_E}{| t - s |} \right) \lesssim
     | t - s |^{- K} \, R (\lambda) \quad \forall \, s \neq t \]
  which implies by taking the logarithm and dividing by $- \log | t - s |$
  that
  \[ \exp \left( \beta \left( \sup_{s \neq t} \frac{\| X_t - X_s \|_E}{| t - s
     |^{1 / 2} \sqrt{- \log | t - s |}} \right)^2 \right) = \sup_{s \neq t} \,
     \exp \left( \beta \frac{\| X_t - X_s \|^2_E}{| t - s | (- \log | t - s
     |)} \right) \lesssim R (\lambda) \]
  which yields the conclusion. Alternatively, it follows from the assumption
  that
  \[ \mathbb{E} [B] \assign \mathbb{E} \left[ \int_{[0, T]^2} \exp \left(
     \lambda \frac{\| X_t - X_s \|^2_E}{| t - s |^{2 \alpha}} \right) \,
     \mathd t \, \mathd s \right] < \infty \]
  which implies that we can apply Garsia--Rodemich--Rumsey Theorem
  (see~{\cite{garsia}}) for the choice $\psi (x) = e^{\lambda x^2}$, $p (x) =
  x^{\alpha}$, which gives
  
  \begin{align}
    \| X_t - X_s \|_E & \lesssim \int_0^{| t - s |} \sqrt{\log B - \log \, u} \,
    u^{\alpha - 1} \, \mathd u \lesssim \left( \sqrt{\log B} + \sqrt{- \log | t - s
    |} \right) | t - s |^{\alpha} \nonumber
  \end{align}
  and from which we can again deduce that
  \[ \sup_{s \neq t} \frac{\| X_t - X_s \|_E}{| t - s |^{\alpha} \sqrt{- \log
     | t - s |}} \lesssim 1 + \sqrt{\log B} \]
  and thus the bound $\exp( \beta \llbracket X\rrbracket_{C^\varphi E}^2) \lesssim B^{C \beta}$, for a suitable constant $C>0$ and any $\beta>0$. Choosing $\beta=1/C$, using the fact that $\mathbb{E}[B]<\infty$, we obtain the exponential integrability bound.
  
  The final claim follows
  immediately.
\end{proof}

We also present here some details on Fourier--Lebesgue spaces.

\begin{definition}
  \label{appendixA2 defn fourier lebesgue} Let $\alpha \in \mathbb{R}$, $p \in
  [1, \infty]$; we define the Fourier--Lebesgue space $\mathcal{F} L^{\alpha,
  \rho} (\mathbb{R}^d)$ as
  \[ \mathcal{F} L^{\alpha, p} (\mathbb{R}^d) = \left\{ f \in \mathcal{S}'
     (\mathbb{R}^d) \, : \, \langle \xi \rangle^{\alpha} | \hat{f} (\xi) | \in
     L^p (\mathbb{R}^d) \right\} \]
  where $\langle \xi \rangle = (1 + | \xi |^2)^{1 / 2}$. It is a Banach space
  endowed with the norm
  \[ \| f \|_{\mathcal{F} L^{\alpha, p}} = \| \langle \cdot \rangle^{\alpha}
     \hat{f} \|_{L^p} . \]
\end{definition}

It follows immediately from the definition that we could replace $\langle
\cdot \rangle$ with any other function having the same behaviour at infinity,
for instance with $(1 + | \cdot |)$; $\langle \cdot \rangle$ is usually
considered as it is the Fourier symbol associated to the operator $(I -
\Delta)^{1 / 2}$. Here is a list of relations of Fourier--Lebesgue spaces with
other known functional spaces:
\begin{itemize}
  \item For any $\alpha \in \mathbb{R}$, $\mathcal{F} L^{\alpha, 2}$ coincides
  the classical fractional Sobolev space $H^{\alpha} = (I - \Delta)^{\alpha /
  2} L^2$.
  
  \item By Hausdorff--Young inequality, for $p \in [1, 2]$ we have the
  embedding $L^p \hookrightarrow \mathcal{F} L^{0, p'}$, where $p'$ denotes
  the conjugate exponent to $p$; similarly for $p$ as above, for the Bessel
  spaces $L^{\alpha, p} = (I - \Delta)^{\alpha / 2} L^p$ we have $L^{\alpha,
  p} \hookrightarrow \mathcal{F} L^{\alpha, p'}$.
  
  \item In the case $f \in L^1$ the result is slightly stronger, namely
  $\hat{f}$ is uniformly continuous, bounded and $\hat{f} (\xi) \rightarrow 0$
  as $\xi \rightarrow \infty$ by Riemann--Lebesgue lemma; if $f$ is a finite
  measure on $\mathbb{R}^d$, then $\hat{f}$ is still uniformly continuous and
  bounded.
  
  \item We have the embedding $\mathcal{F} L^{0, 1} \hookrightarrow C^0$ and
  more generally $\mathcal{F} L^{\alpha, 1} \hookrightarrow C^{\alpha}$, where
  for $\alpha = n \in \mathbb{N}$ we mean the classical $C^n$ space, while for
  $\alpha$ fractional or negative $C^{\alpha} = B^{\alpha}_{\infty, \infty}$,
  the latter being a Besov--H\"{o}lder space.
  
  \item Similarly by Hausdorff--Young for $p \in [1, 2]$ we have the embedding
  $\mathcal{F} L^{\alpha, p} \hookrightarrow L^{\alpha, p'}$.
\end{itemize}
There are also embeddings in different scales of Fourier--Lebesgue spaces.

\begin{lemma}
  \label{appendixA2 fourier lebesgue embedding}For any $q < p$ and any
  $\varepsilon > 0$ it holds
  \[ \mathcal{F} L^{\alpha, p} \hookrightarrow \mathcal{F} L^{\alpha - d
     \left( \frac{1}{q} - \frac{1}{p} \right) - \varepsilon, q} . \]
\end{lemma}

\begin{proof}
  For any $q < p$ and $s > 0$ we have
  \[ \| f \|_{\mathcal{F} L^{\alpha - s, q}} = \left( \int_{\mathbb{R}^d}
     (\langle \xi \rangle^{\alpha} | \hat{f} (\xi) |)^q \langle \xi \rangle^{-
     s q} \mathd \xi \right)^{1 / q} \leqslant \| f \|_{\mathcal{F} L^{\alpha,
     p}}  \left( \int_{\mathbb{R}^d} \langle \xi \rangle^{- s \frac{p q}{p -
     q}} \right)^{\frac{1}{q} - \frac{1}{p}} \]
  where the integral is convergent if and only if $- s p q / (p - q) < - d$,
  namely
  \[ s > d \left( \frac{1}{q} - \frac{1}{p} \right) . \]
\end{proof}

The above statement can be combined with other embeddings like the ones
mentioned above. For instance we have $\mathcal{F} L^{\alpha, \infty}
\hookrightarrow \mathcal{F} L^{\alpha - d / 2 - \varepsilon, 2} = H^{\alpha -
d / 2 - \varepsilon, 2}$ and $\mathcal{F} L^{\alpha, \infty} \hookrightarrow
\mathcal{F} L^{\alpha - d - \varepsilon, 1} \hookrightarrow C^{\alpha - d -
\varepsilon}$.

One of the main motivations to introduce Fourier--Lebesgue spaces is that they
behave nicely under convolution, due to the properties of Fourier transform.

\begin{lemma}
  \label{appendixA2 fourier lebesgue convolution}Let $f \in \mathcal{F}
  L^{\alpha, p}$, $g \in \mathcal{F} L^{\beta, q}$ with $\frac{1}{p} +
  \frac{1}{q} \leqslant 1$. Then $f \ast g \in \mathcal{F} L^{\alpha + \beta,
  r}$ where $\frac{1}{r} = \frac{1}{p} + \frac{1}{q}$ and
  \[ \| f \ast g \|_{\mathcal{F} L^{\alpha + \beta, r}} \leqslant \| f
     \|_{\mathcal{F} L^{\alpha, p}} \| g \|_{\mathcal{F} L^{\beta, q}} . \]
\end{lemma}

\begin{proof}
  By the properties of Fourier transform $\widehat{f \ast g} = \widehat{f} 
  \hat{g}$, therefore
  \[ \| f \ast g \|_{\mathcal{F} L^{\alpha + \beta, r}} = \left(
     \int_{\mathbb{R}^d} (\langle \xi \rangle^{\alpha} | \hat{f} (\xi) |)^r 
     (\langle \xi \rangle^{\beta} | \hat{g} (\xi) |)^r \mathd \xi \right)^{1 /
     r} \leqslant \| f \|_{\mathcal{F} L^{\alpha, p}} \| g \|_{\mathcal{F}
     L^{\beta, q}} \]
  where in the last passage we used the generalised H\"{o}lder inequality $\|
  \varphi \psi \|_{L^r} \leqslant \| \varphi \|_{L^p} \| \psi \|_{L^q}$ for
  $r$, $p$ and $q$ as above.
\end{proof}

It follows in particular from the above that any bounded Fourier symbol acts
continuously on $\mathcal{F} L^{\alpha, p}$, for any choice of $\alpha$ and
$p$. We also have $\mathcal{F} L^{\alpha, p} \ast \mathcal{F} L^{\beta,
\infty} \hookrightarrow \mathcal{F} L^{\alpha + \beta, p}$.

\bibliography{myBiblio-rho}{}
\bibliographystyle{plain}
\end{document}